\documentclass[11pt,reqno,a4paper]{amsart}

\usepackage{mathdots}
\usepackage{tikz}
\usetikzlibrary{decorations.markings}
\usetikzlibrary{decorations.pathreplacing,shapes,arrows}

\usetikzlibrary{arrows,matrix}
\usetikzlibrary{decorations}
\usepgflibrary{arrows}
\tikzset{
    >=stealth',
    punkt/.style={
           rectangle,
           rounded corners,
           draw=black, very thick,
           text width=6.5em,
           minimum height=2em,
           text centered},
    pil/.style={
           ->,
           thick,
           shorten <=2pt,
           shorten >=2pt,}
}
\tikzset{every loop/.style={min distance=2mm,in=225,out=135,looseness=10}}
\tikzset{->-/.style={decoration={
  markings,
  mark=at position #1 with {\arrow{>}}},postaction={decorate}}}

\newcommand{\arrowdraw}[3]{\draw[->-=#1, thick] (#2) to (#3);}

\usepackage{amssymb,amsfonts,amsmath,amsthm}
\usepackage{mathrsfs}
\usepackage{graphicx}
\usepackage{color}
\usepackage{verbatim}
\usepackage[hidelinks]{hyperref}
\usepackage{pdfsync}
\usepackage{mathtools}
\usepackage{todonotes}
\usepackage[all]{xy}
\usepackage{caption}
\usepackage{subcaption}

\newtheorem{thm}{Theorem}[section]
\newtheorem{pro}[thm]{Proposition}
\newtheorem{lem}[thm]{Lemma}
\newtheorem{cor}[thm]{Corollary}

\theoremstyle{definition}
\newtheorem{defn}[thm]{Definition}

\newtheorem*{ackn}{Acknowledgements}

\theoremstyle{remark}
\newtheorem{rmk}[thm]{Remark}

\newtheorem{prb}[thm]{Problem}

\numberwithin{equation}{section}

\newcommand{\lb}{\langle}
\newcommand{\rb}{\rangle}

\def\J{\mathscr{J}}
\def\D{\mathscr{D}}
\def\R{\mathscr{R}}
\def\L{\mathscr{L}}
\def\H{\mathscr{H}}

\def\es{\varnothing}
\def\La{\Leftarrow}
\def\Ra{\Rightarrow}
\def\Lra{\Leftrightarrow}
\def\ol#1{\overline{#1}}

\def\gen#1{\langle{#1}\rangle}
\def\pre#1#2{\langle{#1}\,|\,{#2}\rangle}

\def\Sch{Sch\"u\-tzen\-ber\-ger }

\newcommand{\gh}{\H}
\newcommand{\gr}{\R}

\DeclareMathOperator\Mon{Mon} \DeclareMathOperator\Inv{Inv} \DeclareMathOperator\Gp{Gp} 
\DeclareMathOperator\pref{pref}  \DeclareMathOperator\red{red} 
  \DeclareMathOperator\id{id}

\DeclareMathOperator\MonRC{MonRC}

\begin{document}


\title{On right units of special inverse monoids} 




\thanks{The research of the first named author is supported by the Personal Grant F-121 ``Problems of combinatorial semigroup and group theory'' of the Serbian Academy of Sciences and Arts. 
The research of the second named author was supported by the EPSRC Fellowship Grant EP/V032003/1 ``Algorithmic, topological and geometric aspects of infinite groups, monoids and inverse semigroups".}



\maketitle

\begin{center}
IGOR DOLINKA%
\footnote{Department of Mathematics and Informatics, University of Novi Sad, Trg Dositeja Obra\-do\-vi\-\'ca 4, 21101 Novi Sad, Serbia.
\\ \emph{Email address:} \texttt{dockie@dmi.uns.ac.rs}.}
\ and \ 
ROBERT D. GRAY%
\footnote{School of Engineering, Mathematics and Physics, University of East Anglia, Norwich NR4 7TJ, UK.
\\ \emph{Email address:} \texttt{Robert.D.Gray@uea.ac.uk}.} 
\end{center}

\markleft{IGOR DOLINKA AND ROBERT D.\ GRAY}


\begin{abstract}
We study the class of monoids that arise as the submonoid of right units of finitely presented special inverse monoids. 
Classical results of Makanin show that the monoids of right units in finitely presented special monoids decompose as a free product of the group of units and a finite rank free monoid. 
Gray and Ru\v{s}kuc (2024) gave the first example of a finitely presented special inverse monoid whose submonoid of right units does not admit such a decomposition, which left open the question of determining the structure of such monoids. 
In the first part of this paper we prove a general result which shows that the only instances where the right units of a finitely presented special inverse monoid can admit such a free product decomposition is when their group of units is finitely presented. 
This implies that the typical behaviour of the right units is not to admit such a free product decomposition. 
In showing this, we establish some general results about finite generation and presentability of subgroups of special inverse monoids. 
In particular, we give an exact characterisation of when an arbitrary subgroup is finitely generated in terms of connectedness properties of unions of its cosets in its $\R$-class, and also a characterisation of when an arbitrary subgroup is finitely presented. 
We also give a sufficient condition for finite generation and presentability of an arbitrary subgroup given in terms of a geometric finiteness property called boundary width. 
As a consequence, we show that the classes of monoids of right units of finitely presented special inverse monoids and prefix monoids of finitely presented groups are independent, in the sense that neither of them is contained in the other. 
In the second part of the paper, we show that every finitely generated submonoid of a finitely RC-presented monoid is isomorphic to a submonoid $N$ of a finitely presented special inverse monoid $M$ such that $N$ is a submonoid of the right units of $M$, and $N$ contains the group of units of $M$. 
This result generalises and extends the classification of groups of units of finitely presented special inverse monoids recently obtained by Gray and Kambites (2025). 
From this, we derive a number of surprising properties of RC-presentations for right cancellative monoids contrasting the classical theory of monoid presentations.
\end{abstract}


\section{Introduction} \label{sec:intro}

The positive solution of the decidability of the word problem for one-relator groups, proved by Magnus \cite{Ma1,Ma2} almost a century ago (see \cite{LSch} for a more contemporary treatment), is one the cornerstones of combinatorial group theory. 
The study of one-relator groups remains a highly active area of research in contemporary geometric group theory; see e.g.\ \cite{JZL1,JZL2,LW1,LW2}.
Magnus' solution to the word problem immediately raised the question whether 
an analogous result might be true for one-relator \emph{monoids} $\Mon\pre{A}{u=v}$ (where $u,v$ are two words from the free monoid $A^*$), and this problem amazingly 
still remains open. Numerous strides have been made towards this goal, most notably by S.I.Adyan \cite{Adj} and his students \cite{AO}, and many connections were 
established with other decision problems for various algebraic structures, some of them to be mentioned shortly. For example, one important case for which Adyan 
provided a solution for the word problem is the one of \emph{special} one-relator monoids $\Mon\pre{A}{w=1}$ (see e.g.\ \cite{Lal,Zh1,Zh2} for shorter and more 
modern proofs of this result) and today the structure of such monoids is better understood. Efforts to solve the decidability problem for one-relator monoids have paved the way for the arrival and establishment of combinatorial/geometric semigroup theory as a field of study. 
See \cite{CF} for a recent comprehensive overview of the current state of the art regarding the word problem for one-relator monoids.

One relatively small but extremely far reaching realisation, which completely changed the course of the quest to solve the one-relator monoid problem, came at the turn of 
the century from Ivanov, Margolis and Meakin \cite{IMM}. They have shown that there is yet another class of algebras (besides groups and semigroups) that
plays an indispensable role in this area: \emph{inverse} semigroups/monoids. An inverse monoid is a monoid equipped with an additional unary operation $\hbox{}^{-1}$
subject to the laws $(ab)^{-1}=b^{-1}a^{-1}$, $(a^{-1})^{-1}=a$, $aa^{-1}a=a$ and $aa^{-1}bb^{-1}=bb^{-1}aa^{-1}$, with the latter effectively telling us that
idempotent elements commute, thus rendering the (semigroup-theoretical) inverse for each element unique (but not necessarily meaning that $aa^{-1}=a^{-1}a=1$,
the latter condition describing the elements of the group of \emph{units} of the inverse monoid). The class of inverse monoids has free objects, elegantly
described by Munn \cite{Munn} and Scheiblich \cite{Sch}, and this makes inverse monoid presentations possible. See the monographs \cite{Law,Pet} for background 
in inverse semigroup/monoid theory as well as \cite{CP,How} for general semigroup theory. 

Bearing in mind that the paper \cite{AO} reduced the general one-relator monoid problem to two special cases of presentations, namely $\Mon\pre{a,b}{aub=ava}$ and 
$\Mon\pre{a,b}{aub=a}$ (both of them yielding right cancellative monoids), it was shown in \cite{IMM} that, for example, the word problem for monoids of the 
first type reduces to the word problem for the special inverse monoid 
$\Inv\pre{a,b}{auba^{-1}v^{-1}a^{-1}=1}$, as the initial monoid embeds into the latter inverse monoid. 
An analogous result is true for the monoids of the second type. 
Now, since both one-relator groups $\Gp\pre{A}{w=1}$ and special one-relator monoids $\Mon\pre{A}{w=1}$ have decidable word
problems, it was tempting to conjecture that a similar result holds for special one-relator inverse monoids $\Inv\pre{A}{w=1}$ (which would then instantly
solve the one-relator monoid problem). Surprisingly, recently this was shown to be false by the second-named author of this paper in \cite{Gr-Inv},
which, because of the specific form of the counterexamples, still does not invalidate the approach set forth 
by \cite{IMM}. Furthermore, as shown in \cite{GR}, the structure of finitely presented special inverse monoids is significantly less well-behaved than that
of special (plain) monoids. Once again, right unit monoids played a key role in \cite{Gr-Inv}, since in the counterexample constructed the membership for
this submonoid was algorithmically undecidable, rendering the word problem undecidable as well (because then there is no algorithm to determine whether
$ww^{-1}=1$ holds in our inverse monoid for a given word $w$ over the generators).

Going back to groups, an inverse monoid given by a presentation maps homomorphically onto the group given by the same presentation, and that natural homomorphism
maps the monoid of right units of the initial inverse monoid to the \emph{prefix monoid} of the group: the submonoid generated by all elements represented by
the prefixes of relator words. When the initial special inverse monoid enjoys the nice structural feature of being \emph{$E$-unitary} then the restriction of this
homomorphism to the monoid of right units is actually an isomorphism, and the word problem of the inverse monoid in question reduces the the membership problem for
the prefix monoid in the corresponding group \cite{IMM}. In addition, for the second of the remaining cases of the one-relator monoid problem, Guba \cite{Guba} constructs another 
group (unlike just described, not obviously linked to the initial presentation) so that the word problem for the initial monoid reduces to the prefix membership
problem for the group. We direct the reader to \cite{Mea,DG21} for an overview of the numerous connections between monoids, groups, and inverse monoids of 
the flavour just described. However, it certainly becomes immediately apparent that the monoids of right units in finitely presented special inverse monoids and 
the prefix monoids of finitely presented groups are exceptionally relevant objects of study in the course of understanding the complexity and the essence of 
the one-relator monoid problem.

The main departing point for this paper is the previous work \cite{DG23} of the authors, together with the quite recent work \cite{GK}. Namely, each prefix monoid 
of a finitely presented group is necessarily finitely generated and group-embeddable (and thus a \emph{recursively presented} group-embeddable monoid, in the sense 
of Higman's Embedding Theorem, see \cite{DG23,LSch}). However, not every monoid with those properties arises as a prefix monoid: for example, a group arises as 
a prefix monoid if and only if it is finitely presented \cite[Lemma 3.4]{DG23}. Still, one of the main results of \cite{DG23}, recalled here as Theorem \ref{thm:prefix} 
below, states that when a finitely generated submonoid $M$ of a finitely presented group does not arise as a prefix monoid, this can be ``fixed'' just by adding 
a free monoid of finite rank as a free factor, and all ranks starting from a certain lower bound $\mu_M$ work. Now, in \cite[Section 5]{DG23} some partial results 
on monoids of right units of finitely presented special inverse monoids were gathered, with the (implicit) hope that a similar---if not entirely analogous---result 
might hold for this class. In particular, by \cite[Theorem 5.3]{DG23} this class includes all finitely RC-presented monoids (see the next section for basics on 
RC-presentations of right cancellative monoids), and a finitely generated group belongs to it if and only if it is finitely presented \cite[Lemma 5.2]{DG23}.
In particular, this shows that not every finitely generated submonoid of a finitely RC-presented monoid arises as the monoid of right units of a finitely presented 
special inverse monoid: any finitely generated group that is recursively but not finitely presented serves as a counterexample.

In the first part of the paper we show that conjectures of such type fail in a dramatic fashion: namely, it turns out (see Theorem \ref{thm:nofreeprod} below)
that if $G$ is a finitely generated group and $T$ a finitely generated monoid with a trivial group of units (and this includes the possibility of free monoids 
of finite rank) then the free product $G*T$ arises as right unit monoid of a finitely presented special inverse monoid only if $G$ is finitely presented. 
In other words, these right units monoids do not admit a free decomposition with respect to their group of units unless the group has a finite presentation. 
This result is shown in Section \ref{sec:freeprod}, and to be able to prove it first we must establish some general results about finite generation and 
presentability of subgroups of special inverse monoids. 
These results which are given in Section \ref{sec:boundary} are expressed in terms of the geometry of Sch\"{u}tzenberger graphs
As a consequence, the classes of monoids of right units of finitely presented special inverse monoids and of prefix monoids of finitely presented
groups are independent, in the sense that neither is contained in the other. Also, this provides a wealth of non-group examples of finitely generated submonoids 
of finitely RC-presented (right cancellative) monoids that do not arise as the monoid of right units of a finitely presented special inverse monoids.

In Section \ref{sec:gk} we show that the class of finitely generated submonoids of finitely RC-presented monoids is still very much intertwined with the class of 
monoids of right units of a finitely presented special inverse monoids: we show that for each monoid $T$ from the former class there is a monoid $R$ from the
latter class with the property that $R$ contains a submonoid that contains the entire group of units of $R$ and is isomorphic to $T$. This 
generalises and extends the classification of groups of units of finitely presented special inverse monoids given in   
\cite[Theorem 4.1]{GK}, and indeed we use methods inspired by the paper \cite{GK} to show this. We start by generalising a key construction
from \cite{GK} and proceed by explicitly determining an RC-presentation for the monoid of right units of the special inverse monoids in question.

In the concluding section, this RC-presentation is utilised to show that 
the class of right unit monoids of finitely presented special inverse monoids
does contain right cancellative monoids that are not finitely RC-presented,
and we exhibit a concrete example. This is accompanied with a revisitation of a construction from \cite[Section 6]{GR} (a paper that inspired the work \cite{GK})
for which we determine the monoid of right units. This monoid turns to be always finitely RC-presented, even though, as shown in \cite{GR}, it might be not finitely
presented as a monoid. Besides this, we derive further conclusions demonstrating sharp contrasts between the theories of plain monoid presentations and RC-presentations.


\section{Basic notions and summary of results} \label{sec:summary}

\subsection{Semigroup theory basics}

A fundamental tool in studying the structure of semigroups are five equivalence relations $\R,\L,\J,\H,\D$ called \emph{Green's relations}. Here we
give their definition for a monoid $S$:
$$
a\,\R\,b \Lra aS=bS, \quad a\,\L\,b \Lra Sa=Sb, \quad a\,\J\,b \Lra SaS=SbS.
$$
Further, $\H=\R\cap\L$ and $\D=\R\vee\L$, which is just $\R\circ\L$ as it may be shown that $\R$ and $\L$ commute. We say that $a\in S$ is \emph{regular} if $a=axa$ for some $x\in S$, 
furthermore every regular element $a\in S$ has an \emph{inverse} $x\in S$ such that $axa=a$ and $xax=x$. In a single $\D$-class, either all elements are
regular or none of them; in the former case, all inverses of an element $a$ are contained in $D_a$, the $\D$-class of $a$ (we use similar notation for
classes of other relations). If $x$ is an inverse of $a$ then $ax\in R_a$ and $xa\in L_a$ are idempotents, and all idempotents in a $\D$-class arise in 
this way; hence, in regular $\D$-classes each $\R$-class and each $\L$-class contain at least one idempotent. It is useful to think of a $\D$-class as a
rectangular scheme where rows represent $\R$-classes, the columns are $\L$-classes, and individual boxes are $\H$-classes. Each such ``box'' contains at
most one idempotent, and when it does it turns out to be a maximal subgroup of the monoid/semigroup in question. Every maximal subgroup arises in
this way.  

An element $a$ of a monoid $S$ is called a \emph{right unit} if $ax=1$
for some $x\in S$, and left unit is defined dually. An element that is both a left and right unit it called a \emph{unit}, and the collection of all such elements 
forms a group called the \emph{group of units} of the monoid $S$. The set of right units of $S$ forms a submonoid. From the definitions it is immediate that 
the group of units is equal to the $\gh$-class of the identity element $1$, and the monoid of right units is equal to the $\gr$-class of $1$. The submonoid 
of right units of any monoid is easily seen to be a right cancellative monoid.      

When every element of a monoid $S$ has a unique inverse then $S$ is called an \emph{inverse monoid}. In that case, all $\D$-classes are regular, and it
follows that they must be square in the sense that each $\D$-class contains the same number of $\R$- and $\L$-classes, and each $\R$-/$\L$-class contains
precisely one idempotent. An alternative approach is to consider inverse monoids as algebraic structures of the signature $(2,1,0)$, denoting the unique
inverse of $a$ by $a^{-1}$; then it can be shown that inverse monoids form a variety of unary monoids defined by the identities $(x^{-1})^{-1}=x$, 
$(xy)^{-1}=y^{-1}x^{-1}$, $xx^{-1}x=x$ and $xx^{-1}yy^{-1}=yy^{-1}xx^{-1}$. 
An element $a$ of an inverse monoid $S$ is now a right unit of $S$ if and only if $aa^{-1}=1$, and $a$ is a unit of $S$ if and only if $aa^{-1}=a^{-1}a=1$.

If $S$ is an inverse monoid then the family of congruences $\rho$ on $S$ such that $S/\rho$ is a group is certainly non-empty (as it contains the 
total relation) and it is not difficult to show that it is closed under arbitrary intersections. Therefore, there is a smallest congruence $\sigma$, with 
respect to containment, and so $S/\sigma$ is the \emph{greatest group image} of $S$. Clearly, in $\sigma$ (as indeed in all congruences $\rho$ in the 
considered family) all idempotents of $S$ must be contained in a single class. When $\sigma$ is \emph{idempotent-pure}, that is, when $E(S)$, the set of 
all idempotents of $S$, forms a $\sigma$-class (i.e.\ no non-idempotent elements of $S$ is $\sigma$-related  to an idempotent one), we say that $S$ is 
\emph{$E$-unitary}. 

\subsection{Presentations}

Let $A$ be a non-empty alphabet. The set $A^*$ of all words (finite sequences of letters) is then the free monoid on $A$, with the operation of concatenation
and the empty word as the identity element. To define the free group and the free inverse monoid we first ``double'' the alphabet to $\ol{A}=A\cup A^{-1}$.
The \emph{free group} $FG(A)$ on $A$ is defined on the subset of $\ol{A}^*$ consisting of \emph{reduced words} that is, words without occurrences of subwords
of the form $aa^{-1}$, $a^{-1}a$, $a\in A$. Given $w\in \ol{A}^*$ we have the \emph{reduced form} $\red(w)$ of $w$ obtained from $w$ by successively removing 
such subwords (if any), and so the operations in $FG(A)$ are given by $u\cdot v=\red(uv)$ and $(a_1\dots a_k)^{-1}=a_k^{-1}\dots a_1^{-1}$ for all 
$a_1,\dots,a_k\in\ol{A}$, where $(a^{-1})^{-1}=a$ for all $a\in A$. Finally, the \emph{free inverse monoid} $FIM(A)$ is obtained as the quotient of 
the free monoid $\ol{A}^*$ by the so-called \emph{Wagner congruence}, generated by the pairs $(uu^{-1}u,u)$ and $(uu^{-1}vv^{-1},vv^{-1}uu^{-1})$ for all 
$u,v\in \ol{A}^*$. An elegant geometric description of $FIM(A)$, with elements represented as finite connected birooted subgraphs of the Cayley graph of $FG(A)$, 
was given by Munn \cite{Munn} (see also Scheiblich \cite{Sch}); such graphs are often called \emph{Munn trees}.

The philosophy behind defining algebraic structures by \emph{presentations} is that we specify a set of \emph{generators} which serves as the alphabet of the
corresponding free object, and then the structure is determined by \emph{defining relations} (or \emph{relators}) so that we take a quotient of the free object
by the congruence naturally determined by the relations. For example, we write $M=\Mon\pre{A}{u_i=v_i\; (i\in I)}$  if $M\cong A^*/\theta$ where $\theta$ is 
the congruence of the free monoid $A^*$ generated by the set of pairs of words $\{(u_i,v_i):\ i\in I\}$. Similarly, $T=\Inv\pre{A}{u_i=v_i\; (i\in I)}$ defines 
the quotient of the free inverse monoid $FIM(A)$ by the congruence on $FIM(A)$ generated by the set of pairs of words $\{(u_i,v_i):\ i\in I\}$, while we have
$G=\Gp\pre{A}{w_i=1\; (i\in I)}$ if $G\cong FG(A)/N$ where $N$ is a normal subgroup of $FG(A)$ generated (as a normal subgroup) by the words  $w_i$, $i\in I$. 
We say that a monoid, inverse monoid, or a group is \emph{finitely
presented} if it admits a presentation of the corresponding kind with finitely many generators and finitely many defining relations. A monoid or an inverse
monoid is \emph{special} if it admits a presentation in which all the defining relations are of the form $w=1$. 

\begin{defn}\label{defn:prefix}
Let $G$ be a finitely presented group. We say that $M$ is a \emph{prefix monoid in $G$} if there exists a finite presentation of $G$,
$$
G = \Gp\pre{A}{w_i=1\; (i\in I)},
$$
such that $M$ is isomorphic to the submonoid of the group $G$ generated by all elements represented by prefixes of $w_i$, $i\in I$.
A monoid $M$ is said to be a \emph{prefix monoid} if it is a prefix monoid in some finitely presented group. We denote by $\mathcal{P}$ the class
of all prefix monoids. 
\end{defn}

\begin{defn}\label{defn:ru}
Let $M$ be a finitely presented special inverse monoid,
$$
M = \Inv\pre{A}{w_i=1\; (i\in I)}.
$$
Then $\mathrm{RU}(M) = \{ m \in M: mm^{-1}=1 \}$ is the (plain) submonoid of $M$ of right units of $M$ (or the \emph{RU-monoid of} $M$). We call a monoid an 
\emph{RU-monoid} if it is isomorphic to the monoid of right units of some finitely presented special inverse monoid. We denote by $\mathcal{RU}$ the class of 
all RU-monoids.
\end{defn}

As mentioned before, all members of $\mathcal{P}$ are group-embeddable and finitely generated, and all members of $\mathcal{RU}$ are right cancellative.
Furthermore, 
every monoid in $\mathcal{RU}$ is finitely generated which may be shown by using the
argument in the proof of \cite[Proposition 4.2]{IMM},
demonstrating that the RU-monoid of $M = \Inv\pre{A}{w_i=1\; (i\in I)}$ is generated by the set of all elements represented by prefixes of the relator words $w_i$,
$i\in I$. So, since it easily follows that the greatest group image of $M$ is the group given by the same presentation as $M$, namely, $G=\Gp\pre{A}{w_i=1\; (i\in I)}$,
in the natural (surjective) homomorphism $M\to G$, the RU-monoid of $M$ maps onto the prefix monoid of $G$ with respect to the presentation in question.
Note that the prefix monoid is in general sensitive to the choice of the presentation of a particular group. It also follows from the results of \cite{IMM} that when 
$M$ is $E$-unitary, the restriction of the natural homomorphism $M\to G$ to the RU-monoid of $M$ is an isomorphism between $\mathrm{RU}(M)$ and the prefix monoid.

Bearing in mind the fact that RU-monoids are right cancellative, there is yet another type of presentation that can be very useful and convenient in studying 
the class $\mathcal{RU}$. These are the so-called \emph{RC-presentations} \cite{Cain, CRR} 
(introduced by Adjan \cite{Adj} who called them \emph{right-cancellative presentations}).
Namely, for an alphabet $A$ and a family of pairs 
$\mathfrak{R}=\{(u_i,v_i):\ i\in I\}$ of words from $A^*$ we write
$$
M = \MonRC\pre{A}{u_i=v_i,\; i\in I}
$$
if $M\cong A^*/\mathfrak{R}^{\mathrm{RC}}$ where $\mathfrak{R}^{\mathrm{RC}}$ is the intersection of all congruences $\rho$ of $A^*$ containing $\mathfrak{R}$ 
with the property that $A^*/\rho$ is right cancellative. Similarly to the case of the greatest group image of an inverse monoid, it is not difficult to show
that the intersection of a family of congruences with the property that the corresponding quotient is right cancellative has the same property. The paper 
\cite{CRR} also specifies syntactic rules corresponding to this notion. Namely, it turns out that two words $u,v\in A^*$ represent the same element of $M$ 
if and only if there is a \emph{right cancellative $\mathfrak{R}$-chain} from $u$ to $v$:
$$
u=u_0\to u_1 \to \dots \to u_l=v
$$
with $u_0,u_1,\dots,u_l\in (A\cup A^\mathsf{R})^*$ ($A^\mathsf{R}$ being a disjoint copy of $A$) such that for all $0\leq i<l$, we have $u_i=p_iq_ir_i$ and 
$u_{i+1}=p_iq_i'r_i$ for some $p_i\in A^*$, $r_i\in (A\cup A^\mathsf{R})^*$ such that one of the following holds:
\begin{itemize}
\item[(a)] either $(q_i,q_i')\in\mathfrak{R}$ or $(q_i',q_i)\in\mathfrak{R}$;
\item[(b)] $q_i$ is empty, while $q_i'=aa^\mathsf{R}$ for some $a\in A$;
\item[(c)] $q_i=aa^\mathsf{R}$ for some $a\in A$, while $q_i'$ is empty.
\end{itemize}
These steps are called \emph{$\mathfrak{R}$-steps}, (right) \emph{insertions}, and (right) \emph{deletions}, respectively. Note that since $u,v\in A^*$, to each insertion step 
$u_i\to u_{i+1}$ introducing the letter $a^\mathsf{R}$ corresponds a deletion $u_{i\delta}\to u_{i\delta+1}$, with $i<i\delta$, where this letter $a^\mathsf{R}$ is removed.
The corresponding insertion and deletion steps form a stack-like structure, that is, if $u_i\to u_{i+1}$ and $u_j\to u_{j+1}$ are both insertion steps with $i<j$ then
$j\delta<i\delta$. 

\begin{defn}\label{defn:rc12}
We denote by $\mathcal{RC}_1$ the class of all finitely generated, right cancellative monoids that embed into some finitely RC-presented monoid. On the other hand, if
$$S=\MonRC\pre{A}{u_i=v_i,\; i\in I}$$ 
for a finite set $A$ we say that $S$ is \emph{recursively RC-presented} if $\{(u_i,v_i):\ i\in I\}\subseteq A^*\times A^*$ is a  
recursively enumerable set of pairs of words. By $\mathcal{RC}_2$ we denote the class of all  right cancellative monoids that are recursively RC-presented. 
It is not difficult  to see (either by employing the previous definition of the right cancellative $\mathfrak{R}$-chain or e.g.\ \cite[Exercise 2.6.12]{How} 
and some basic techniques of recursion theory) that $\mathcal{RC}_1\subseteq\mathcal{RC}_2$. Indeed, the one-sided analogue of  \cite[Exercise 2.6.12]{How} can be used 
to show every finitely RC-presented monoid is recursively presented as a monoid. Hence any monoid $T$ in $\mathcal{RC}_1$ embeds in a recursively presented monoid. 
Since every finitely generated submonoid of a recursively presented monoid is itself recursively presented, it follows that the right cancellative monoid $T$ is 
recursively presented as a monoid. Then we can take any recursive presentation for $T$ and that presentation will also be a recursive RC-presentation for $T$, 
thus $T$ is recursively RC-presented.  
\end{defn}

Clearly, the notion of ``recursively presented'' can be in a straightforward manner extended to  finitely generated groups, monoids, and inverse monoids.

\begin{rmk}
It is not difficult to see that we have $\mathcal{P}\subseteq\mathcal{RC}_1$. Indeed, by definition every prefix monoid $M$ is finitely generated and embeds into a finitely presented group $G$. Hence, $M$ is right cancellative. On the other hand, the finitely presented group $G$ also has a finite monoid presentation, $G=\Mon\pre{A}{\mathfrak{R}}$. But then also $G=\MonRC\pre{A}{\mathfrak{R}}$, showing that $M\in\mathcal{RC}_1$.
\end{rmk}

\begin{rmk}
The already mentioned Higman Embedding Theorem \cite{LSch} states that a finitely generated group embeds into a finitely presented one if and only if it is
recursively presented. An analogous result for monoids and monoid presentations was established by Murski\u{\i} \cite{Mur}, and for inverse monoids and inverse monoid
presentations by Belyaev \cite{Bel}. However, at present there is no such result for right cancellative monoids with respect to RC-presentations, and therefore
we do not know if $\mathcal{RC}_1=\mathcal{RC}_2$ holds or there is a proper containment between the two classes. It can be shown, however, that $\mathcal{RC}_2$
coincides with the class of recursively presented monoids that happen to be right cancellative. For this reason, we must have $\mathcal{RU}\subseteq\mathcal{RC}_2$. 
Indeed, if $T$ is an RU-monoid then $T$ is finitely generated and by definition embeds into a finitely presented inverse monoid, and it is known (and, again, not difficult 
to show) that any such inverse monoid $M$ is recursively presented as an ordinary monoid, which means by \cite{Mur} that $M$ embeds into a finitely presented one. 
It follows that $T$ is recursively presented as a monoid and any recursive monoid presentation for $T$ will also be a recursive RC-presentation for $T$. Hence 
$T$ belongs to $\mathcal{RC}_2$.  
\end{rmk}

\subsection{Summary of results}

A key fact (see \cite{CF}) in the study of finitely presented special monoids is that their right units decompose as $U \ast X^*$ where $U$ is the group of units and $X^*$ is 
a finite rank free monoid. This structural information about the right units 
has been applied to prove numerous results about special monoids e.g. their topological and homological finiteness properties \cite{GS}.
In \cite{DG23} a similar link was discovered between prefix monoids of finitely presented groups and monoids of the form $S\ast X^*$. In a slightly rephrased form, the main
pertinent result is as follows.

\begin{thm} \label{thm:prefix}
Let $S$ be a finitely generated monoid that embeds into a finitely presented group.
\begin{itemize}
\item[(1)] If $S$ is finitely presented then $S\in\mathcal{P}$ (see \cite[Proposition 3.5]{DG23}).
\item[(2)] If $S$ is not finitely presented then there exists a finite set $C$ such that $S\ast C^*\in\mathcal{P}$ (see \cite[Theorem 3.6]{DG23}).
\end{itemize}
\end{thm}
We know from the results in \cite{DG23} that in the second case above we cannot always take $C = \es$: this is e.g.\ the case with all non-finitely presented groups 
\cite[Lemma 3.4]{DG23}. Hence, both monoids of right units of finitely presented special monoids and prefix monoids of finitely presented groups admit decompositions 
as a free product of a monoid from a particular class together with a finite rank free monoid. 

A natural question that arises next is: To what extent is the same true for RU-monoids? Namely, is there an analogous ``free product decomposition'' description for 
RU-monoids (relative to their group of units)? As it turns out, the answer is a categorical ``no''. We are going to show that the only instances where such a description is 
possible is when the group of units is finitely presented, so that Theorem \ref{thm:prefix} fails quite dramatically for RU-monoids.

In this paper we are going to clarify the relationship between the classes $\mathcal{P}$ and $\mathcal{RU}$. It is instantly clear that $\mathcal{RU}\not\subseteq\mathcal{P}$
since by \cite[Theorem 5.3]{DG23} all finitely RC-presented monoids are RU-monoids, and there is such a right cancellative monoid that is not group-embeddable and hence not in 
$\mathcal{P}$ (see the example after Corollary 5.4 in \cite{DG23}). A far less obvious fact, that will follow from the results we prove in this paper, is that 
$\mathcal{P}\not\subseteq\mathcal{RU}$, so that the two classes considered are incomparable with respect to containment. Namely, the main result of Section \ref{sec:freeprod}, 
Theorem \ref{thm:nofreeprod}, will be that if $M$ is a finitely presented special inverse monoid with group of units $U$ and RU-monoid $R$ such that $R\cong U\ast T$ for some 
finitely generated monoid $T$ with a trivial group of units, then $U$ must be finitely presented. In particular, this also holds when $T$ is a free monoid of finite rank, and 
so when $H$ is a finitely generated group that is not finitely presented there exists a finite set $C$ such that $H\ast C^*\in\mathcal{P}$ but $H\ast C^*\not\in\mathcal{RU}$. 
As a byproduct, this yields numerous examples of finitely generated submonoids of finitely RC-presented monoids (including examples that are not groups) that do not belong to 
$\mathcal{RU}$.

In order to be able to prove this, we shall establish in Section \ref{sec:boundary} some general results about finite generation and presentability of subgroups of 
special inverse monoids. There has been a lot of interest recently in the subgroup structure of special inverse monoids particularly in understanding their groups of units 
and more generally their maximal subgroups; see e.g.\ \cite{GK} and \cite{GR}. In particular those results show that the group of units of a finitely presented special 
inverse monoid need not be finitely presented, and also that the maximal subgroups of such a monoid need not even be finitely generated. As a consequence, for finitely 
presented special monoids there is interest in understanding under what conditions finitely generation and presentability are inherited by the (maximal) subgroups---in 
particular the group of units---from a finitely presented special inverse monoid. Some sufficient conditions ensuring this do exist in the literature, e.g.\ having 
finitely many $\H$-classes in $\R$-classes, but those results are useless for studying special inverse monoids since that condition essentially never holds; see the first paragraph of Section~\ref{sec:boundary} below for a more precise explanation of this. 
So other weaker sufficient conditions are needed when studying sugbroups of special inverse monoids. Here we identify a geometric condition called \emph{boundary width} which: (a) generalises the finite index results e.g.\ from 
\cite{ruvskuc1999presentations}, but also (b) is sufficient for finite generation and presentability to be inherited. In particular, we provide:
\begin{itemize} 
\item an exact characterisation of when an arbitrary subgroup is finitely generated in terms of connectedness properties of unions of $\H$-classes in its $\R$-class;     
\item an exact characterisation of when an arbitrary subgroup is finitely presented in terms of the (Vietoris--)Rips complexes built on the graph arising in the previous point;     
\item a sufficient condition for finite generation and presentability of an arbitrary subgroup given in terms of a geometric finiteness property, called boundary width, 
which says something about the way the graph in the previous bullet points embeds in the \Sch graph; 
\item a result relating the geometry of the Cayley graph of the right units with the \Sch graph $S\Gamma(1)$ of the identity, more precisely we show these 
spaces are quasi-isometric. 
\end{itemize}

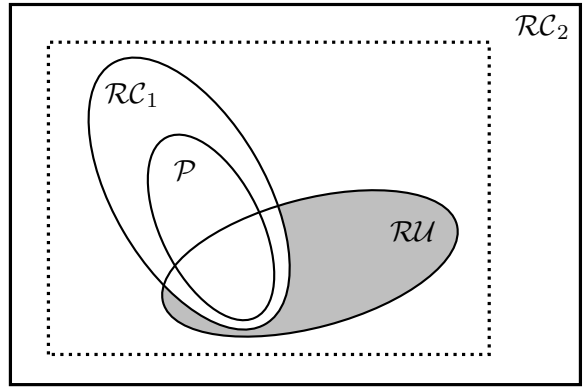
\begin{figure}[tb] 
\begin{center}
\begin{tikzpicture} 
\draw[very thick] (-0.5,-0.54) rectangle (7,4.5);
\draw[very thick, dotted] (0,-0.14) rectangle (5.8,4); 
\draw node at (6.5,4.2) {$\mathcal{RC}_2$};
\fill[gray!50, rotate=15](3.6,0.14) ellipse (2 and 0.85);
\draw node at (4.8,1.5) {$\mathcal{RU}$};
\filldraw[fill=white, thick, rotate=30](2.6,0.8) ellipse (1 and 2);
\draw node at (1.1,3.3) {$\mathcal{RC}_1$};
\filldraw[fill=white, thick, rotate=26](2.6,0.45) ellipse (0.66 and 1.33);
\draw node at (1.8,2.3) {$\mathcal{P}$};
\draw[thick, rotate=15](3.6,0.14) ellipse (2 and 0.85);
\end{tikzpicture} 
\caption{A Venn diagram of the classes $\mathcal{P}$ (Definition~\ref{defn:prefix}), $\mathcal{RU}$ (Definition~\ref{defn:ru}), and the classes $\mathcal{RC}_1$, $\mathcal{RC}_2$ of finitely generated right cancellative monoids from Definition~\ref{defn:rc12}. It is not known if the shaded area within the oval representing $\mathcal{RU}$ is empty or not, and the same applies to space between $\mathcal{RC}_1$ and $\mathcal{RC}_2$. The white areas within the ovals representing $\mathcal{P}$ and $\mathcal{RC}_1$ are positively non-empty. The dotted line represents the class considered in Theorem~\ref{thm:dense}: the finitely generated submonoids of members of $\mathcal{RU}$ containing their group of units.} \label{fig1}
\end{center}
\end{figure}

The results described above show that the $\mathcal{RU}$ is not equal to the class $\mathcal{RC}_1$ of all finitely generated submonoids of finitely RC-presented monoids.
(See Figure \ref{fig1} for an illustration of the relationships between these various
classes.)
They also show that given a monoid $N$ in $\mathcal{RC}_1$ in general there is no finite set $C$ such that $N \ast C^* \in \mathcal{RU}$. 
However, in Section \ref{sec:gk} we show that there is still a very close connection between the classes $\mathcal{RC}_1$ and $\mathcal{RU}$ so that the former
is ``dense'' in the latter, in the following sense: we shall prove that every finitely generated submonoid of a finitely RC-presented monoid is isomorphic to a 
submonoid $N$ of a finitely presented special inverse monoid $M$ such that (a) $N$ is a submonoid of the right units of $M$, and (b) $N$ contains the group of units of $M$. 
To achieve this, we generalise a key construction from \cite{GK} of a finitely presented special inverse monoid by taking a pair $(S,T)$ consisting of a finitely
RC-presented (right cancellative) monoid $S$ and its finitely generated submonoid $T$ as input data, instead of a pair $(G,H)$ consisting of a finitely presented 
group $G$ and its finitely generated subgroup $H$. Then, we proceed to find an explicit RC-presentation for their RU-monoids. In this way, we are able to show in 
Section \ref{sec:fp} that $\mathcal{RU}$ does contain monoids that are not finitely RC-presented. Also, in this final section we determine a presentation for the monoid 
of right units of special inverse monoids $M_{Q,W}$ introduced in \cite[Section 6]{GR}, which starts from a finitely presented group $K_Q$ and its finitely generated 
submonoid $T_W$. It turns out that the right unit monoid in question is precisely the greatest right cancellative image of the HNN-like Otto-Pride extension \cite{PO,GS} 
with respect to the embedding $T_W\to K_Q$. As a consequence, somewhat unexpectedly we show that there is a finitely presented $E$-unitary special inverse monoid with 
submonoid of right units that is finitely RC-presented but not finitely presented as a monoid.


\section{Boundary width: a sufficient condition for finite presentability} \label{sec:boundary}

In this section we give a sufficient condition that, when satisfied by a maximal subgroup $H$, suffices to imply that $H$ is finitely presented.  
The only results in the literature along these lines are \cite{ruvskuc1999presentations} (see also \cite{steinberg2003topological} for a topological proof) where it is shown 
that if the $\gr$-class of $H$ has only finitely many $\gh$-classes then $H$ will be finitely presented. However, that condition is too strong when studying special inverse monoids: 
by standard semigroup theory \cite{CP,How} if 
the $\gr$-class $R_1$ in $M=\Inv\pre{A}{w_1=1,\dots,w_k=1}$ has finitely many $\gh$-classes then it must have just one $\gh$-class, 
in which case (see \cite[Proposition 2.5]{gray2023subgroups}) $M$ is the free product of a group and a free inverse monoid. In other words, in any interesting example, the group 
of units of a special inverse monoid will never satisfy the finitely many $\gh$-classes in its $\gr$-class finiteness condition. 

So, to mend this situation, we introduce a finiteness condition on group $\gh$-classes that is better suited for the study of special inverse monoids. The notion we introduce 
here will properly generalise the finitely many $\gh$-classes condition in the sense that any group $\gh$-class that has finitely many $\gh$-classes in its $\gr$-class will satisfy it. 
Also, the theory presented here is valid for arbitrary inverse monoids, not just those defined by special presentations. 

In the next section we will see how the general theory we develop here can be applied to prove results about the possible structure of RU-monoids.  

For an undirected graph $\Gamma$ and vertices $x,y \in V(\Gamma)$ we use $d(x,y)$ to denote the distance between $x$ and $y$ in $\Gamma$ viewed as a metric space in the usual way.     
We view edges in a graph $\Gamma$ as coming in pairs (in the sense of Serre \cite{Serre}), 
paths in graphs are compatible sequences of edges, and we use $\iota$ and $\tau$ for the initial and terminal vertices of edges and paths.    
Usually our notation for graphs will be $\Gamma = \Gamma(V,E,\iota,\tau)$ where $V$ is the set of vertices, $E$ is the set of edges and $\iota$ and $\tau$ are maps $E\to V$.   

\begin{defn}[Boundary width (for undirected graphs)]
Let $X \subseteq V(\Gamma)$ be a set of vertices in a graph $\Gamma$. We call $(x,y)\in X\times X$ a \emph{boundary pair} of $X$ if there is a path $\pi = e_1e_2\dots e_m$ with 
the following properties: $\iota e_1=x$, $\tau e_m=y$, and $\iota e_2, \iota e_3,\dots, \iota e_m$ all belong to $V(\Gamma)\setminus X$.  In this case we call $\pi$ a \emph{boundary path} 
associated to the boundary pair $(x,y)$. Define 
\[
\beta(X) = \sup \{ d(x,y):\ (x,y) \mbox{ is a boundary pair} \}
\]
where $d(x,y)$ is the distance between $x$ and $y$ in the graph $\Gamma$. We call $\beta(X)$ the \emph{boundary width} of $X$. We say that $X$ has the \emph{finite boundary width} 
if $\beta(X)$ is finite. 
\end{defn}

Note that if $X \subseteq V(\Gamma)$ has finite boundary width this means that if $\pi$ is any path that begins in $X$, ends in $X$, and all intermediate vertices are outside of $X$, 
then the point that $\pi$ reenters $X$ must be close (i.e.\ globally bounded by $\beta(X)$) to the point where it left.        

\begin{defn}[Finite ball covers]
Let $\Gamma = \Gamma(V,E,\iota,\tau)$ be a graph and let $\Delta$ be a subset of $V$. 
Let 
\[
\Delta_r = \bigcup_{v \in \Delta} \mathcal{B}_r(v)
\]
where $\mathcal{B}_r(v) = \{ u \in V: d(u,v) \leq r \} $ denotes the ball of radius $r$ around $v$. We say that $\Delta$ has a \emph{finite ball cover with finite boundary width} if there exists $r \geq 0$ 
such that $\Delta_r$ has finite boundary width. 
We say that $\Delta$ has a \emph{connected finite ball cover} if there exists $r \geq 0$ such that the subgraph of $\Delta$ induced by $\Delta_r$ is a connected graph.  
\end{defn}

The notion of finite ball cover is needed since there are examples where the subset $\Delta$ has infinite boundary width, but admits a finite ball cover with finite boundary width. 
Due to this, it will be necessary to pass to finite ball covers in our applications. 

Before proving the key result Lemma~\ref{lem_QSI} below, we first record two preliminary lemmas that will be employed in the proof of that lemma.  

\begin{lem}\label{lem:Fact1} 
Let $\Gamma$ be a graph and let $\Delta \subseteq V(\Gamma)$. 
\begin{enumerate} 
\item If $\Delta$ has finite boundary width then any finite ball cover of $\Delta$ will also have finite boundary width.   
More generally, if $\Omega \subseteq \Delta \subseteq \Gamma$ and $\Delta$ has finite boundary width then for any finite ball cover $\overline{\Omega}$ of $\Omega$, if 
$\Delta \subseteq \overline{\Omega}$, then $\overline{\Omega}$ has finite boundary width.  
\item If $\Delta$ has a connected finite ball cover then any finite ball cover of $\Delta$ will also have a connected finite ball cover.  
\end{enumerate}  
\end{lem}

\begin{proof}
(1) 
Suppose that $\Delta$ has finite boundary width and let $\overline{\Delta}$ be the cover of $\Delta$ with balls of radius $r$.     
Let $\pi$ be a boundary path for the set $\overline{\Delta}$ with endpoints $x$ and $y$ in $\overline{\Delta}$.         
Fix $x',y' \in \Delta$ with $d(x,x') \leq r$ and $d(y,y') \leq r$. Let $\sigma$ be a path from $x'$ to $x$ of length at most $r$, and let $\tau$ be a path from $y$ to $y'$ 
of length at most $r$. The composition $\sigma \pi \tau$ is a path from $x' \in \Delta$ to $y' \in \Delta$ and since all vertices of $\pi$ apart from the endpoints belong 
to $\Gamma \setminus \overline{\Delta}$ it follows that that these vertices of $\pi$ also belong to $\Gamma \setminus \Delta$. It follows that there are subpaths $\sigma_1$ 
of $\sigma$ and $\tau_1$ or $\tau$ such that $\sigma_1 \pi \tau_1$ is a subpath of $\sigma \pi \tau$ where $\sigma_1 \pi \tau_1$ is a boundary path for $\Delta$. Since 
$\Delta$ has finite boundary width it follows that there is a global bound $K$ such that the distance in $\Gamma$ between the initial and terminal vertices of 
$\sigma_1 \pi \tau_1$ is at most $K$. It then follows that 
\[
d(x,y) \leq 2r + K. 
\]     
Since $\pi$ was an arbitrary boundary path of $\overline{\Delta}$ it follows that $\overline{\Delta}$ has finite boundary width with upper bound $2r + K$. 

The more general statement can be proved using exactly the same argument. Indeed, suppose that 
$\Omega \subseteq \Delta \subseteq \Gamma$ and $\Delta$ has finite boundary width, and let $\overline{\Omega}$ be a finite ball cover $\overline{\Omega}$ of $\Omega$, 
with balls of radius $r$ satisfying $\Delta \subseteq \overline{\Omega}$. Let $\pi$ be a boundary path for the set $\overline{\Omega}$ with endpoints $x$ and $y$ in 
$\overline{\Omega}$. Fix $x',y' \in \Omega \subseteq \Delta$ with $d(x,x') \leq r$ and $d(y,y') \leq r$. Let $\sigma$ be a path from $x'$ to $x$ of length at most $r$, 
and let $\tau$ be a path from $y$ to $y'$ of length at most $r$. The composition $\sigma \pi \tau$ is a path from $x' \in \Delta$ to $y' \in \Delta$ and since all 
vertices of $\pi$ apart from the endpoints belong to $\Gamma \setminus \overline{\Omega}$ it follows that since $\Delta \subseteq \overline{\Omega}$ these vertices of 
$\pi$ also belong to $\Gamma \setminus \Delta$. It follows that there are subpaths $\sigma_1$ of $\sigma$ and $\tau_1$ or $\tau$ such that $\sigma_1 \pi \tau_1$ is 
a subpath of $\sigma \pi \tau$ where $\sigma_1 \pi \tau_1$ is a boundary path for $\Delta$. Since $\Delta$ has finite boundary width it follows that there is 
a global bound $K$ such that the distance in $\Gamma$ between the initial and terminal vertices of $\sigma_1 \pi \tau_1$ is at most $K$. It then follows that 
\[
d(x,y) \leq 2r + K. 
\]     
Since $\pi$ was an arbitrary boundary path of $\overline{\Omega}$ it follows that $\overline{\Omega}$ has finite boundary width with upper bound $2r + K$. 

(2) 
Let $\overline{\Delta}$ be a connected finite ball cover of $\Delta$ by balls of radius $r$. Let $\Omega$ be the cover of $\Delta$ by balls of radius $n$.   
If we choose $m$ such that $n+m > r$ and take the cover $Q$ of $\Omega$ by balls of radius $m$, then we claim that $Q$ is a connected finite ball cover of $\Omega$. 
Indeed, by definition $Q$ is a finite ball cover, so all that remains is to show that it is connected. 
For any $x \in Q$ there is a path in $Q$ of length at most $m$ from $x$ to a vertex $x'$ in $\Omega$.     
Then from $x'$ there is a path in $\Omega$, which is also a path in $Q$, of length at most $n$ to a vertex $x''$ in $\Delta$. So there is a path in $Q$ of length 
at most $n+m$ from $x$ to $x'' \in \Delta$. Furthermore for any two vertices $y,z \in \Delta$ by definition there is a path in $\overline{\Delta}$ from $y$ to $z$. 
But since $n+m > r$ it follows that $\overline{\Delta}$ is contained in $Q$. It follows that for any two vertices $x_1, x_2 \in Q$ there is a path from $x_1$ to 
$x_1'' \in \Delta$, a path from $x_2$ to $x_2'' \in \Delta$ and then a path between $x_1''$ and $x_2''$ in $Q$. Thus $Q$ is connected, as required.
\end{proof}

The next result says that increasing the radius of ball covers cannot destroy either of the good properties that they have.  

\begin{lem}\label{lem:Fact2} 
Let $\Gamma$ be a graph and let $\Delta \subseteq V(\Gamma)$. 
\begin{enumerate} 
\item If $\Delta$ has a finite ball cover with finite boundary width then any cover of $\Delta$ by larger finite radius balls also has finite boundary width.    
\item If $\Delta$ has a connected finite ball cover then any cover of $\Delta$ by larger finite radius balls is also a connected finite ball cover. 
\end{enumerate}  
\end{lem}

\begin{proof} 
(1) 
Suppose that the cover $\Delta_1$ by balls of radius $r$ has finite boundary width and let $\Delta_2$ be the ball cover by balls of radius $r+s$.
We need to show that $\Delta_2$ has finite boundary width.
 
Let $\Delta_3$ be the ball cover of $\Delta$ by balls of radius $r+s$. We claim that $\Delta_3 = \Delta_2$. 

$(\Delta_3 \subseteq \Delta_2)$ If $x \in \Delta_3$ then there is a path from $x$ to a vertex in $\Delta$ of length at most $r+s$. We can then split this path 
into two, giving a path of length at most $s$ from $x$ to a vertex in $\Delta_1$ followed by a path of length at most $r$ from that vertex in $\Delta_1$ to 
a vertex in $\Delta$. Hence $x \in \Delta_2$.        
 
$(\Delta_2 \subseteq \Delta_3)$ If $x \in \Delta_2$ then there is a path of length at most $s$ from $x$ to a vertex $y \in \Delta_1$. Then there is a path of length 
at most $r$ from $y$ to a vertex $z \in \Delta$. Composing these we conclude that $x \in \Delta_3$.         

Since $\Delta_2 = \Delta_3$ it follows that $\Delta_2$ is a finite ball cover of $\Delta$. Hence by   Lemma~\ref{lem:Fact1} it follows that $\Delta_2$ has finite 
boundary width, as required.   

(2)
Let $\Delta_1$ be a connected finite ball cover of $\Delta$ by balls of radius $r$, and let $\Delta_2$ be a cover of $\Delta$ by balls of radius $r+s$. Then given 
any two vertices $x,y$ in $\Delta_2$ there is a path from $x$ to a vertex $x_1 \in \Delta$, and a path from $y$ to a vertex $y_1 \in \Delta$. Then there is a path 
in $\Delta_1$, and thus also in $\Delta_2$ since $r+s \geq r$, from $x_1$ to $y_1$. Composing these paths we see there is a path from $x$ to $y$ in $\Delta_2$, as required.
\end{proof}

For the next result we need some terminology and notation for quasi-isometries of graphs and metric spaces. We follow terminology and notation from \cite{ghysinfinite}. 

By a \emph{$(\lambda, \epsilon, \mu)$-quasi-isometry} $f: (X,d) \rightarrow (X',d')$  with $\lambda \geq 1$ and $\epsilon, \mu \geq 0$ we mean one that satisfies 
\[
\frac{1}{\lambda} d(x,y) - \epsilon 
\leq
d'(f(x),f(y)) 
\leq
\lambda d(x,y) + \epsilon 
\] 
and $f(X) \subseteq X'$ is $\mu$-quasi-dense with constant $\mu$.    

Following \cite[Exercise 10.6]{ghysinfinite} we note that if two spaces $X$ and $X'$ are quasi-isometric then there is a pair of quasi-isometries $f:X \rightarrow X'$ 
and $g:X' \rightarrow X$ both satisfying the inequalities above, and also such that $d(g(f(x)),x) \leq \mu$ and $d'(f(g(x')),x') \leq \mu$. So $f$ and $g$ are 
quasi-inverses of each other. Note the neither $f$ nor $g$ need be injective nor surjective.   

\begin{lem}\label{lem:newPreImage}
Let $f:\Gamma \rightarrow \Gamma'$ and $g: \Gamma' \rightarrow \Gamma$ be $(\lambda, \epsilon, \mu)$ quasi-isometries that are quasi-inverses of each other satisfying    
$d(g(f(x)),x) \leq \mu$ and $d'(f(g(x')),x') \leq \mu$.
If $X \subseteq V(\Gamma)$ has finite boundary width then $Y = g^{-1}(X) \subseteq V(\Gamma')$ has a finite ball cover $\overline{Y}$ with finite boundary width.     
\end{lem}

\begin{proof}
Let $\overline{Y} \subseteq V(\Gamma')$ be a ball cover of $Y$ with balls of radius $M = \lambda(\lambda + 2\epsilon)+1$. We claim that $\overline{Y}$ has finite boundary width. 
Let $(x_1, x_2, \ldots, x_k)$ be the sequence of vertices in a boundary path for $\overline{Y}$ with $x_0, x_k \in \overline{Y}$. We need to bound $d'(x_0,x_k)$. For each $i$, 
let $p_i$ be a path from $g(x_i)$ to $g(x_{i+1})$ with length $|p_i| \leq \lambda + \epsilon$. This is possible since $d'(x_i,x_{i+1})=1$ which implies 
$d(g(x_i),g(x_{i+1})) \leq \lambda \cdot d'(x_i,x_{i_1}) + \epsilon = \lambda + \epsilon$. 
For every vertex $y \in Y$ and for all $i$ we have $d'(x_i,y) \geq M$. Indeed each of $x_1, \ldots, x_{k-1}$ is distance greater than $M$ from $y$ which then implies that 
$x_0$ and $x_k$ are each distance at least $M$ from $y$. Applying $g$ this implies 
\[ 
d(g(x_i),g(y)) \geq \frac{M}{\lambda} - \epsilon > \frac{\lambda(\lambda + 2 \epsilon)}{\lambda} - \epsilon = \lambda + \epsilon.  
\]                          
Since $X = g(Y)$ it follows that for all $i$ and all $x \in X$ we have $d(g(x_i),x) > \lambda + \epsilon$. That is $d(g(x_i),X) > \lambda + \epsilon$. Since 
$|p_i| = \lambda + \epsilon$ it follows that for all $i$ the path $p_i$ has empty intersection with $X$. Choose $z \in Y$ with $d'(z,x_0) \leq M$ and choose $t \in Y$ with 
$d'(t,x_k) \leq M$; this is possible since $x_0, x_k \in \overline{Y}$. Now 
\[
d(g(x_0),g(z)) \leq \lambda \cdot M + \epsilon =\lambda(\lambda(\lambda+2\epsilon)+1)+\epsilon 
\] 
and 
\[
d(g(x_k),g(t)) \leq \lambda \cdot M + \epsilon =\lambda(\lambda(\lambda+2\epsilon)+1)+\epsilon.
\]
It follows that there exist vertices $\alpha, \beta \in X$ and paths $\pi$ from $\alpha$ to $g(x_0)$ and $\sigma$ from $g(x_k)$ to $\beta$ such that each of $\pi$ and $\sigma$ 
has length at most $\lambda(\lambda(\lambda+2\epsilon)+1)+\epsilon$ and the composition $\pi \circ p_1 \circ p_2 \circ \ldots \circ p_k \circ \sigma$ is a boundary path of $X$ 
in $\Gamma$ from $\alpha \in $ to $\beta \in X$. Since $X$ has finite boundary width with upper bound $\kappa$, say, it follows that $\delta(\alpha,\beta) \leq \kappa$. But this 
implies 
\begin{align*} 
d(g(x_0),g(x_k)) 
& \leq |\pi| + \kappa + |\sigma| \\
& \leq 
2\cdot(\lambda(\lambda(\lambda+2\epsilon)+1)+\epsilon) + \kappa.
  \end{align*}
This can then be used to bound $d'(x_0,x_k)$ as follows: 
\[
2\cdot(\lambda(\lambda(\lambda+2\epsilon)+1)+\epsilon) + \kappa
\geq d(g(x_0),g(x_k)) \geq
\frac{1}{\lambda} d(x_0,x_k) - \epsilon 
\]   
which implies 
\[
d(x_0,x_k) \leq 
\lambda \cdot (2\cdot(\lambda(\lambda(\lambda+2\epsilon)+1)+\epsilon) + \kappa) + \lambda \epsilon = K.
\]
Since $(x_0, x_1, \ldots, x_k)$ was a sequence of vertices in an arbitrary boundary path for $\overline{Y}$ we conclude that $\overline{Y}$ has finite boundary width with 
fixed upper bound $K$ given in the displayed equation above. This completes the proof of the lemma.         
\end{proof}

\begin{lem}\label{lem_QSI}
Let $f: \Gamma \rightarrow \Gamma'$ be a quasi-isometry between graphs. Let $\Delta\subseteq V(\Gamma)$. Then the following hold:  
\begin{enumerate}
\item If the graph $\Delta$ has a finite ball cover with finite boundary width in $\Gamma$ then $f(\Delta)$ has a finite ball cover with finite boundary width in $\Gamma'$.
\item If the graph $\Delta$ has a connected finite ball cover in $\Gamma$ then $f(\Delta) \subseteq \Gamma'$ has a connected finite ball cover in $\Gamma'$.
\end{enumerate}
\end{lem}

\begin{proof} 
Let $f:\Gamma \rightarrow \Gamma'$ and $g: \Gamma' \rightarrow \Gamma$ be $(\lambda, \epsilon, \mu)$ quasi-isometries that are quasi-inverses of each other satisfying    
$d(g(f(x)),x) \leq \mu$ and $d'(f(g(x')),x') \leq \mu$.

(1)
Let $\overline{\Delta}$ be a finite ball cover of $\Delta$ such that $\overline{\Delta}$ has finite boundary width. 
Let $\overline{\overline{\Delta}}$ be a finite ball cover of $\Delta$ with radius chosen large enough such that $g(f(\overline{\Delta})) \subseteq \overline{\overline{\Delta}}$.  
This is possible due to the quasi-density constant $\mu$.  
Then Lemma~\ref{lem:Fact2} (1) implies that $\overline{\overline{\Delta}}$ has finite boundary width.
Let $\Omega = g^{-1}(\overline{\overline{\Delta}}) \subseteq \Gamma'$, and observe that $F(\Delta) \subseteq \Omega$.   

Since $\overline{\overline{\Delta}}$ has finite boundary width, applying Lemma~\ref{lem:newPreImage}, it follows that $\Omega = g^{-1}(\overline{\overline{\Delta}}) 
\subseteq \Gamma'$ has a finite cover $\overline{\Omega}$ finite boundary width. 

To complete the proof of part (1) of this lemma choose a finite ball cover $\overline{f(\Delta)}$ of $f(\Delta)$ such that $\overline{\Omega} \subseteq \overline{f(\Delta)}$. 
Such a choice is possible by the definitions and the fact that $f$ and $g$ are quasi-isometries that are quasi-inverses of each other.  
Since $\overline{\Omega}$ had finite boundary width it then follows from the general statement in Lemma~\ref{lem:Fact1}(1) using the sets $f(\Delta) \subseteq \overline{\Omega} 
\subseteq \overline{f(\Delta)}$ that $\overline{f(\Delta)}$ has finite boundary width, as required.   

(2) 
Let $\overline{\Delta}$ be a connected finite ball cover of $\Delta$. 
Let $\Omega$ be the cover of $f(\overline{\Delta})$ by balls of radius $\lambda + \epsilon$ and let $\overline{f(\Delta)}$ be a finite ball cover of $f(\Delta)$ such that 
$\Omega \subseteq \overline{f(\Delta)}$. This is possible by the definitions and quasi-isometry properties of $f$ and $g$.    
For any edge $e$ in $\overline{\Delta}$ with endvertices $x$ and $y$, we have $d'(f(x),f(y)) \leq \lambda d(x,y) + \epsilon = \lambda + \epsilon$. Hence there is a path 
in $\Omega$ from $f(x)$ to $f(y)$. 

It follows that for every $u,v \in f(\Delta) \subseteq f(\overline{\Delta})$ there is a path in $\Omega$ from $u$ to $v$.     
By definition for every $z \in \overline{f(\Delta)} $ there is a path in $\overline{f(\Delta)}$ from $z$ to a vertex in $f(\Delta)$.
Combining these facts we conclude that $\overline{f(\Delta)}$ is connected and thus is a connected finite ball cover of $f(\Delta)$, as required. 
\end{proof}

The fact that the finite boundary width property behaves well with respect to quasi-isometry will be important when passing between properties of the Cayley graph of the 
right unit monoid and those of the \Sch graph $S\Gamma(1)$ which are non-isomorphic, yet quasi-isometric (this will be proved below).  

We now introduce some definitions that will allow the notions above to be applied to inverse monoids.  

In the following definition $H$ is an arbitrary subgroup of $M$ (not just a group $\gh$-class), and a \emph{coset} is defined as in \cite{ruvskuc1999presentations}: 
this is a right translate $Hs$ of $H$ that is also a subset of the $\gr$-class of the monoid that contains $H$. In the case that $H$ is a group $\gh$-class the cosets are exactly 
the $\gh$-classes in the $\gr$-class that contains $H$.

Also, it is convenient at this point to recall (e.g.\ from \cite{IMM,GK}) the notions of the \emph{(right) Cayley graph} of an inverse monoid $M$ with a specified generating set $A$, 
and of the \emph{(right) \Sch graph} of an element $m\in M$. Namely, the right Cayley graph $\Gamma_A(M)$ has vertex set $M$ and a directed edge labelled by $a \in \ol{A}= A \cup A^{-1}$ 
from $s$ to $sa$ for all $s \in M$. Then the (right) \Sch graph $S\Gamma_A(m)$ of an element $m\in M$  is the strongly connected component of the right Cayley graph $\Gamma_A(M)$ 
containing $m$. It follows readily that $S\Gamma_A(m_1)=S\Gamma_A(m_2)$ (meaning that their underlying sets of vertices are equal as subsets of $M$) if and only if $m_1\,\R\,m_2$ 
and in this sense we may speak about the \Sch graph $S\Gamma_A(R)$ of an $\R$-class $R$. If $R$ is an $\gr$-class of a monoid $M$ we shall write $S\Gamma_A(R)$ to denote the 
\Sch graph of the $\R$-class $R$. In particular, the vertices of $S\Gamma_A(1)$ are precisely the right units of $S$. The \Sch graph $S\Gamma_A(m)$ has a natural structure of 
a metric space where the distance between $x,y \in S\Gamma_A(m)$ is defined to be the shortest length of a word $w$ in $(A \cup A^{-1})^*$ such that $xw = y$. This does define 
a metric space since as $x \gr y$ we have $x = xww^{-1} = yw^{-1}$.    

\begin{defn}\label{def_bouncy}
Let $M$ be an inverse monoid generated by a finite set $A$. Let $H$ be a subgroup of $M$. 
If $\{H_i:\ i\in F\}$ is a finite set of right cosets of $H$ such that $H$ is a subset of $\Delta = \bigcup_{i \in F} H_i $ then we call $\Delta$ a \emph{finite cover} 
of the subgroup $H$. 
We say that $\Delta$ has \emph{finite boundary width} if it has this property as a subset of the \Sch graph $S\Gamma_A(R)$ of the $\gr$-class $R$ of $M$ that contains $H$.   
Similarly, we say that $\Delta$ is \emph{connected} if the subgraph of $S\Gamma_A(R)$ induced on the set $\Delta$ is a connected graph.    
\end{defn}

\begin{rmk}
In Definition~\ref{def_bouncy}, whether or not a finite cover $\Delta$ is connected can depend on the choice of finite generating set $A$ for the inverse monoid $M$. Similarly, 
whether or not $\Delta$ has finite boundary width can also depend on the choice of finite generating set $A$. In contrast, we shall see in Remark~\ref{rmk:bobmark} and in the proof 
of Theorem~\ref{thm:SMApplication} below that whether $H$ admits some finite cover $\Delta$ with finite boundary width is independent of the choice of finitely generating set for $M$, 
and similarly whether $H$ admits a finite cover $\Delta$ that is connected is also independent of the choice of finite generating set for $M$.            
\end{rmk}

\begin{lem}\label{lem:CosetsGettingBigger}
Let $M$ be an inverse monoid generated by a finite set $A$, let $H$ be a subgroup of $M$, and suppose that $\Delta = \bigcup_{i \in J} H_i $ is a finite cover of $H$ with finite 
boundary width. Then
\begin{enumerate}
\item For any finite set $J' \supseteq J$, if the union of cosets $\Delta' = \bigcup_{i \in J'} H_i $ is connected then $\Delta'$ has finite boundary width. 
\item For every finite set $J' \supseteq J$ there is a finite set $J'' \supseteq J'$  such that the union of cosets $\Delta'' = \bigcup_{i \in J''} H_i $ is connected.
\end{enumerate}
\end{lem}

\begin{proof} 
(1) 
Let $\Gamma = \Gamma_A(R)$ be the \Sch graph of the $\gr$-class $R$ that contains the group $H$. From the assumptions it follows that there is a constant $K$ such that for every vertex 
$v \in \Delta'$ we have $d(v,\Delta) \leq K$. Now every boundary path $\tau = (x_0, x_1, \ldots, x_k)$ of $\Delta'$ can be extended using paths $\pi, \sigma$ in $\Delta'$ 
with $|\pi| \leq K$ and $|\sigma| \leq K$ such that $\pi \circ \tau \circ \sigma$ is a boundary path of $\Delta$. Since $\Delta$ has finite boundary width, $\kappa$ say, 
it follows that $d(x_0,x_k)$ is at most $\kappa + 2K$. Hence $\Delta'$ has finite boundary width with upper bound $\kappa + 2K$.   

(2)
Let $\Gamma = \Gamma_A(R)$ be the \Sch graph of the $\gr$-class $R$ that contains the group $H$. Suppose that the boundary width of $\Delta$ is bounded above by $\kappa$. Let $\Delta_\kappa$ 
be the cover of $\Delta$ by balls of radius $\kappa$. Let $\overline{J} \supseteq J$ be the finite collection of cosets that the set $\Delta_\kappa$ intersects and define  
$\overline{\Delta} = \bigcup_{i \in \overline{J}} H_i $. We claim that $\overline{\Delta}$ induces a connected subgraph of $\Gamma$. To see this observe that 
$\Delta \subseteq \Delta_\kappa \subseteq \overline{\Delta}$. Given any two vertices $x,y \in \Delta$ we can take any path $\pi$ in $\Gamma$ from $x$ to $y$ and decompose it 
into $\pi = \pi_1 \ldots \pi_m$ where each $\pi_i$ is either a path in $\Delta$ or $\pi_i$ is a boundary path of $\Delta$. 
This just involves tracing the path $\pi$ and recording the points where we leave the set $\Delta$, and the points where we re-enter $\Delta$.     
Each boundary path $\pi_i$ can be replaced by a path in $\Gamma$ between two vertices of $\Delta$ of length at most $\kappa$. Hence that entire path is in $\Delta_\kappa$. 
Hence doing this for each $\pi_i$ we conclude that there is a path entirely in $\Delta_\kappa$ from $x$ to $y$. Finally it follows from the definitions that for every vertex 
$z \in \overline{\Delta}$ there is a path in $\overline{\Delta}$ from $z$ to a vertex in $\Delta$. This completes the proof that $\overline{\Delta}$ is connected.             

Now take a $\Delta_\upsilon$ be the cover of $\Delta$ be balls of radius $\upsilon$ where $\upsilon > \kappa$ is chosen large enough that $\Delta_\upsilon$ intersects 
all of the costs $\overline{J} \cup J'$ and then let $J''$ be the finite set of cosets that intersect with $\Delta_\upsilon$.         
Then arguing as in the previous paragraph we conclude that $\Delta'' = \bigcup_{i \in J''} H_i $ is connected.
\end{proof}

\begin{cor}\label{cor:connectedCover}
Let $M$ be a finitely generated inverse monoid and let $H$ be a subgroup of $M$. If $H$ has a finite cover with finite boundary width then every finite cover of $H$ itself 
has a finite connected cover with finite boundary width. 
\end{cor}

\begin{proof} 
Let $\Delta = \bigcup_{i \in J} H_i $ be a finite cover of $H$ with finite boundary width.
Let $\Delta_1 = \bigcup_{i \in J_1} H_i $ be some other finite cover of $H$.
Then applying Lemma~\ref{lem:CosetsGettingBigger} with $J' = J \cup J_1$ we conclude there is a finite set $J'' \supseteq J'$  such that the union of cosets 
$\Delta'' = \bigcup_{i \in J''} H_i $ is connected. It then follows from Lemma~\ref{lem:CosetsGettingBigger}(1) that $\Delta''$ is a finite connected cover of 
$\Delta \cup \Delta_1$ with finite boundary width, by which we mean $\Delta''$ is a finite connected cover of $H$ with finite boundary width and 
$\Delta \cup \Delta_1 \subseteq \Delta''$.       
\end{proof}

Putting the results above together gives us the following key proposition.
 
\begin{pro}\label{prop_SC}
Let $M$ be a finitely generated inverse monoid and let $H$ be a subgroup of $M$. Then the following are equivalent: 
\begin{enumerate}
\item $H$ (as a subset of its Sch\"{u}tzenberger graph) has a finite ball cover with finite boundary width;
\item $H$ (as a subset of its Sch\"{u}tzenberger graph) has a connected finite ball cover with finite boundary width;
\item $H$ has a finite cover with finite boundary width;
\item $H$ has a connected finite cover with finite boundary width.
\end{enumerate}
\end{pro}

\begin{proof} 
(3) $\Rightarrow$ (4): This follows from  Corollary~\ref{cor:connectedCover}. 

(4) $\Rightarrow$ (2): If $\Delta = \bigcup_{i \in J} H_i $ is a connected finite cover with finite boundary width then we can take a finite ball cover $\Omega$ 
of $H$ such that $\Omega \supseteq \Delta$ then $\Omega$ is a connected finite ball cover of $H$ and $\Omega$ has finite boundary width by Lemma~\ref{lem:Fact1}(1). 

(2) $\Rightarrow$ (1): This is immediate from the definitions. 

(1) $\Rightarrow$ (3): If $\Omega \subseteq \Gamma$, where $\Gamma$ is the Sch\"{u}tzenberger graph, is a finite ball cover of $H$ with finite boundary width, then if we let 
$\Delta = \bigcup_{i \in J} H_i $ be the finite collection of cosets of $H$ that $\Omega$ intersects then it may be verified that $\Delta$ is a finite cover of $H$ with finite boundary width.              
\end{proof}

\begin{rmk}\label{rmk:bobmark}
Whether or not one (and hence all) of the four equivalent conditions given in the statement of Proposition~\ref{prop_SC} hold for a subgroup $H$ is independent of the choice of finite generating 
set for the containing inverse monoid $M$ in the following sense. If $A$ and $B$ are two finite generating sets for $M$ then by \cite[Proposition 5]{gray2013groups} if $R$ is the $\gr$-class containing 
$H$ then $S\Gamma_A(R)$ is quasi-isometric to $S\Gamma_B(R)$ where the identity map on $R$ is a quasi-isometry. Then since the identity map fixes the subset $H$, by Lemma~\ref{lem_QSI}(1) it follows 
that condition (1) of Proposition~\ref{prop_SC} holds with respect to the generating set $A$ if and only if it holds with respect to the generating set $B$. This means that the in the equivalent 
conditions of proposition is one of abstract subgroups of finitely generated inverse monoids, and we can refer to it without mentioning any particular finite generating set for the inverse monoid.  
\end{rmk}

Before looking at the issue of finite presentability we deal with the question of finite generation. We state the result for inverse monoids, but it may also be shown to hold for 
monoids in general using a similar proof.   

The following result is of particular interest since it has recently been proved \cite{GK} that a maximal subgroup of a special inverse monoid need not be 
finitely generated. 

\begin{thm}\label{thm:SMApplication} 
Let $M$ be a finitely generated inverse monoid and let $H$ be a subgroup of $M$. Let $\Gamma$ be the \Sch graph containing $H$. Then $H$ is finitely generated 
if and only if $H$ admits a finite connected cover, that is, there is a finite cover $H \subseteq \Delta = \bigcup_{i \in F} H_i $ such that the subgraph of $\Gamma$ 
induced by $\Delta$ is connected. In this case the graph induced by $\Delta$ is quasi-isometric to the Cayley graph of the group $H$.       
\end{thm}

\begin{proof} 
($\Rightarrow$) 
Suppose that $H$ is finitely generated. Then we can choose a generating set $A$ for $M$ that contains a subset $X \subseteq A$ such that $X$ is a finite generating set 
for $H$. Then with respect to the generating set $X$ the set $H$ itself is a finite cover of $H$. 
As we have already mentioned in Remark~\ref{rmk:bobmark} the results of 
 \cite{gray2013groups} imply that for any generating set 
$B$ of $M$ the Sch\"{u}tzenberger graph $\Gamma_B$ containing $H$ taken with respect to $B$ is quasi-isometric to that same graph taken $\Gamma_A$  with respect to $A$, 
where the identity map of $M$ induces a quasi-isometry. It then follows from Lemma~\ref{lem_QSI}(2) that with respect to 
any finite generating set $B$ the group $H$ admits a finite connected cover in $\Gamma_B$ which combined with Proposition~\ref{prop_SC} completes the proof.        

($\Leftarrow$) 
This is proved using an argument exactly analogous to the proofs of \cite[Theorem~2 and Theorem~3]{gray2013groups}. 
Indeed, suppose that $H \subseteq \Delta = \bigcup_{i \in F} H_i $ 
is a finite connected cover of $H$. Then $H$ acts on the graph $\Delta$ by left multiplication and since the cover is finite it follows that this is a proper co-compact action 
by isometries. It then follows from the \v Svarc--Milnor Lemma \cite{BridsonHaefliger} that $H$ is finitely generated and is quasi-isometric to $\Delta$.      
\end{proof}

\begin{cor}
Let $H$ be a subgroup an inverse monoid $M$ generated by a finite set $A$. 
If $H$ is finitely generated then $S\Gamma_A(H)$ contains a quasi-isometric copy of the Cayley graph of $H$ as an induced subgraph.    
\end{cor}

\begin{rmk} 
Since the group of units of a special inverse monoid is finitely generated this shows that for any finitely presented special inverse monoid there is a quasi-isometric copy 
of the Cayley graph of its group of units in $S\Gamma(1)$.  
\end{rmk}

Since finite presentability is a geometric property (see \cite[Proposition 10.18]{ghysinfinite}) we can extend Theorem~\ref{thm:SMApplication} to a statement about 
finite presentability. 

\begin{thm}\label{thm:SMApplication:FP} 
Let $M$ be a finitely generated inverse monoid and let $H$ be a subgroup of $M$. Let $\Gamma$ be the \Sch graph containing $H$. 
Then $H$ is finitely generated if and only if $H$ admits a finite connected cover, that is, there is a finite cover $H \subseteq \Delta = \bigcup_{i \in F} H_i $ such that 
the subgraph of $\Gamma$ induced by $\Delta$ is connected. In this case the graph induced by $\Delta$ is quasi-isometric to the Cayley graph of the group $H$.       
Moreover, the group $H$ is finitely presented if and only if the graph $\Gamma$ has the property that: 
\[
L(k)\text{ generates the fundamental group }\pi_1(\Delta)\text{ for $k$ large enough,}
\]  
where $L(k)$ is the set of all loops in $\Delta$ based at the idempotent $e \in H$ of the form $p q p^{-1}$ such that $q$ is a loop of length at most $k$.      
\end{thm}

\begin{rmk} 
The condition involving $L(k)$ given in the previous theorem can alternatively be expressed as saying that the Rips complex $\mathrm{Rips}_r(\Delta)$ of $\Delta$ is 
simply connected for $r$ large enough; see \cite[Chapter~4]{de2000topics}.
\end{rmk}

We now prove the key result showing that if one (and hence all) of the equivalent conditions in Proposition~\ref{prop_SC} is satisfied this will suffice to prove that 
the subgroup inherits the properties of being finitely generated or presented from the monoid $M$.    

\begin{thm}\label{thm_monhybrid_inv}
Let $M$ be a finite generated inverse monoid and let $H$ be a subgroup of $M$ which has a finite cover with finite boundary width. Then $H$ is finitely generated. 
Moreover, if $M$ is finitely presented (as an inverse monoid) then the group $H$ is finitely presented. 
\end{thm}

The rest of this section will be devoted to the proof of Theorem~\ref{thm_monhybrid_inv}.
Let $M = \Inv\pre{A}{\mathfrak{R}}$ where $A$ is a finite set, and if $M$ is finitely presented we also take $\mathfrak{R}$ to be a finite set. Then as a monoid $M$ is defined by the following infinite monoid presentation: 
\[ \Mon\pre{A, A^{-1}}{\mathfrak{R}, 
\; \alpha \alpha^{-1} \alpha = \alpha,  
\; \alpha \alpha^{-1} \beta \beta^{-1} = \beta \beta^{-1}  \alpha \alpha^{-1} 
\; (\alpha, \beta \in (A \cup A^{-1})^*) }. \]

Let $\{H_i:\ i \in I\}$ be the collection of all cosets of $H$ in the $\gr$-class $R$ containing $H$.
By Proposition~\ref{prop_SC} and the assumptions in the theorem there is a finite collection of right cosets $\{H_j:\ j \in J\}$, where $J\subseteq I$ is a finite set, with $1 \in J$ and $H_1 = H$, such that the union $X = \bigcup_{j \in J} H_j$ has finite boundary width, say $\kappa$, and also the subgraph of the \Sch graph induced on $X$ is a connected graph. 
(We note here that we have $1 \in I$ and write $H = H_1$ but this does not mean that $H$ is the entire $\gh$-class of the identity element of $M$, that is, $H$ is not necessarily the group of units of $M$ here.)

As is not difficult to show, the right multiplicative action of $M$ on $R \cup \{ 0 \}$ induces an action of $M$ on the set $I \cup \{0\}$ which we denote by $(i,m) \mapsto i \cdot m \in I \cup \{0\}$. It is important to note that $J \cup \{0\}$ need not be a subact of this act since typically we can right multiply a coset from $X$ to get a coset belonging to $R$ but not to $X$. 

For each $i$ we choose some word $r_i \in (A \cup A^{-1})^*$ such that $H r_i = H_i$; in particular we choose $r_1 = 1$ for $H r_1 = H_1 r_1 = H_1 = H$. Also set $r_i' = r_i^{-1} \in (A \cup A^{-1})^*$.
Note that for any $h\in H$, $i\in I$ and $s\in M$ such that $i\cdot s\in I$ we have
$h\,\R\,hr_i\,\R\,hr_is$, implying $hh^{-1}=hr_i(hr_i)^{-1}=hr_is(hr_is)^{-1}$ and so
\begin{align*}
hr_iss^{-1} &= hr_i(hr_i)^{-1}hr_iss^{-1} = hr_iss^{-1}r_i^{-1}h^{-1}hr_i\\
&= hr_is(hr_is)^{-1}hr_i = hh^{-1}hr_i = hr_i
\end{align*}
holds in $M$. 
In particular, $hr_ir_i'=h$ holds for all $h\in H$, $i\in I$.
The elements $r_i, r_i'$ are called \emph{coset representatives} of $H$.   
Note that it also follows that for all $k \in H_i$ we have $k r_i' r_i = k$. Indeed, we can write $k = h r_i$ for some $h \in H = H_1$ and then 
$$
k r_i' r_i = h r_i r_i' r_i = h r_i = k. 
$$

Our next aim will be to define a set of words over $A\cup A^{-1}$,
$$
W = \{w_{x,y}:\ x,y\in X,\ (x,y)\text{ is a boundary pair }\},
$$
such that $w_{x,y}$ labels a path of minimal possible length from $x$ to $y$ in the \Sch graph $S\Gamma(R)$. By our assumptions, all words from $W$ will have a length $\leq\kappa$, and therefore the set $W$ is necessarily finite. We are going to show that we can assume that the set $W$ enjoys some further nice properties which will greatly facilitate further work. 

\begin{lem}\label{lem:wprop}
The set of words $W$ may be chosen so that it has the following two properties:
\begin{enumerate}
\item whenever $w_{x_1,y_1}$ and $w_{x_2,y_2}$ are in $W$ with $x_1, x_2 \in H_j$ and $y_1,y_2 \in H_k$ for some $j,k \in J$ and there exists an element $s \in M$ such that $x_1s=y_1$ and $x_2s=y_2$ (so that $j\cdot s=k$), then  
$$er_j w_{x_1,y_1} = er_j w_{x_2,y_2}$$
holds in $M$, where $e$ is the idempotent of $M$ which is the identity element of the subgroup $H$;
\item $w_{y,x} \equiv w_{x,y}^{-1}$ holds for all boundary pairs $(x,y)$.
\end{enumerate}
\end{lem}

\begin{proof}
We begin by choosing, for any pair $(j,s)$, ($j\in J$, $s\in M$) such that $j\cdot s\in J$, a word $\beta_{j,s}\in (A\cup A^{-1})^*$ labelling a shortest path in $S\Gamma(R)$ from $er_j\leadsto er_js$. Throughout the proof, these words will be fixed.

Next, we define a linear order $<$ on the set $X$ by fixing a linear order $\prec$ in the set $J$, linear orders $<_j$ on the cosets $H_j$, $j\in J$, and then taking the ordered sum of ordered sets $(H_j,<_j)$ with respect to $(J,\prec)$. 

Now assume that $(x,y)\in X\times X$ is a boundary pair. If $x=y$ then we set $w_{x,x}=1$, the empty word. Otherwise, assume first that $x<y$. Since $x\,\R\,y$, we have $y=xs$ for some $s\in M$, so if $x\in H_j$ then $y\in H_{j\cdot s}$. In this case we set $w_{x,y}\equiv \beta_{j,s}$. If, however, $x>y$ then we define $w_{x,y}$ to be the word $w_{y,x}^{-1}$.
Since $ys^{-1}=xss^{-1}=hr_jss^{-1}=hr_j=x$, the latter word is just $\beta_{j\cdot s,s^{-1}}^{-1}$.
We need to prove that this definition is logically correct, i.e.\ that it does not depend on the particular choice of $s\in M$, and also that it satisfies the required property (1) (the property (2) is immediate from the definition).

Indeed, assume first that we also have $y=xt$ for some $t\in M$ (so $j\cdot t=j\cdot s$). Since we assume that $x\in H_j$, it follows that $x=hr_j$ for some $h\in H = H_1$. Hence, $y=xs=xt$ implies that $hr_jt=hr_js$, so $er_jt=h^{-1}hr_jt=h^{-1}hr_js=er_js$. Thus we have $\beta_{j,t}\equiv\beta_{j,s}$. 
As we have already observed, $ys^{-1}=xss^{-1}=hr_jss^{-1}=hr_j=x$, and similarly we have $x=yt^{-1}$, so we can also conclude that $\beta_{j\cdot t,t^{-1}}\equiv\beta_{j\cdot s,s^{-1}}$ and thus
$\beta_{j\cdot t,t^{-1}}^{-1}\equiv\beta_{j\cdot s,s^{-1}}^{-1}$. Therefore, the definition of the word $w_{x,y}$ does not depend on the ``intervening'' element $s$.

Also, $w_{x,y}$ indeed labels a path from $x$ to $y$ in $S\Gamma(R)$ because, by definition, $er_j\beta_{j,s} = er_js$ holds in $M$ and so, if $x<y$, $xw_{x,y}=hr_j\beta_{j,s} = her_j\beta_{j,s}=her_js=xs=y$. On the other hand, if $x>y$, then, upon writing $y=kr_{j\cdot s}$ for some $k\in H$ we have
$$
xw_{x,y} = ys^{-1}\beta_{j\cdot s,s^{-1}}^{-1} = ker_{j\cdot s}s^{-1}\beta_{j\cdot s,s^{-1}}^{-1} = ker_{j\cdot s}\beta_{j\cdot s,s^{-1}}\beta_{j\cdot s,s^{-1}}^{-1} = kr_{j\cdot s} = y.
$$
It is straightforward to see that in both cases the path labelled by $w_{x,y}$ must be one of the shortest paths between $x$ and $y$, for otherwise there would be a path $er_j\leadsto er_js$ shorter than the length of $\beta_{j,s}$ (or a path $er_{j\cdot s}\leadsto er_{j\cdot s}s^{-1}$ shorter than the length of $\beta_{j\cdot s,s^{-1}}^{-1}$).

Now assume that $x_1,x_2,y_1,y_2$ are as described in the property (1). If $j\prec k$, or $j=k$ and $x_1<_jy_1$, $x_2<_jy_2$, then both words $w_{x_1,y_1}$ and $w_{x_2,y_2}$ are equal to $\beta_{j,s}$, and if $k\prec j$, or $j=k$ and $y_1<_jx_1$, $y_2<_jx_2$, they are both equal to $\beta_{j\cdot s,s^{-1}}^{-1}$. So, let $j=k$ with $x_1<_jy_1$, $y_2<_jx_2$, or $y_1<_jx_1$, $x_2<_jy_2$. Then $\{w_{x_1,y_1},w_{x_2,y_2}\} = \{\beta_{j,s},\beta_{j,s^{-1}}^{-1}\}$.  
Since by assumption $er_js=er_j\beta_{j,s}\in H_j$ we have $er_js=hr_j$ for some $h\in H$, thus
$$
er_js\beta_{j,s^{-1}} = her_j\beta_{j,s^{-1}} = her_js^{-1} = er_jss^{-1}=er_j.
$$
Hence,
$$
er_j\beta_{j,s^{-1}}^{-1} = her_j\beta_{j,s^{-1}}\beta_{j,s^{-1}}^{-1} = her_j = er_js = er_j\beta_{j,s}
$$
holds in $M$, which implies that in $M$ we have $er_jw_{x_1,y_1} = er_jw_{x_2,y_2}$ in all cases, as required.
\end{proof}

In the next lemma we identify a natural finite monoid generating set $Y$ for the group $H$. This deals with the finite generation part in the statement of Theorem~\ref{thm_monhybrid_inv}. 

\begin{lem}\label{lem:fgBoundary} 
Let 
$M = \Inv\pre{A}{\mathfrak{R}}$ where $A$ is a finite set and let $H$ be a subgroup of $M$ with idempotent $e \in H$. 
Suppose that $H$ has a finite cover with finite boundary width. Then with the above definitions and notation      
\[
Y = \{ 
e r_j w r_{j \cdot w}' : w \in W, j \in J \text{ and } j \cdot w \in J
\}
\]
is a finite monoid generating set for the group $H$. In particular, the group $H$ is finitely generated.   
\end{lem}

\begin{proof} 
We shall now prove that $Y$ is a finite generating set for the subgroup $H$. The proof that $Y$ generates $H$ will also reveal how the rewriting mapping for the finite presentabtility proof will need to be defined.  

Let $h \in H$. Write $h = \gamma \in (A \cup A^{-1})^+$ so $h = e\gamma$ in $M$. Observe that for every decomposition $\gamma \equiv \gamma' \gamma''$ of $\gamma$ we have $e \gamma' \in X$ if and only if $1 \cdot \gamma' \in J$. Now decompose the word 
$
\gamma \equiv \gamma_1 \ldots \gamma_k
$
uniquely such that each $\gamma_i$ is non empty and $e \gamma_1 \ldots \gamma_n \in X$ for all $1 \leq n \leq k $ while $e\gamma' \not\in X$ for every other non-empty prefix $\gamma'$ of $\gamma$.         

Next note that every pair $(e, e\gamma_1)$ and $(e \gamma_1 \ldots \gamma_m, e \gamma_1 \ldots \gamma_{m+1})$ is a boundary pair.    
It then follows from the definitions that $(er_1, er_1 \gamma_1) = (e, e \gamma_1)$ is a boundary pair and also 
$(er_{1\cdot \gamma_1 \ldots \gamma_m}, er_{1\cdot \gamma_1 \ldots \gamma_{m}} \gamma_{m+1})$ is a boundary pair for all $m$. 

Now for each $m$ we define 
\[
\overline{\gamma_{m+1}} = 
w_{
er_{1\cdot \gamma_1 \ldots \gamma_m}, er_{1\cdot \gamma_1 \ldots \gamma_{m}} \gamma_{m+1}},
\]
and 
\[
\overline{\gamma_{1}} = 
w_{er_1, er_1 \gamma_1},
\]
Then it follows from the definitions that 
\[
h = \gamma = e \gamma = 
e \gamma_1 \ldots \gamma_k
\]
and 
\begin{eqnarray*}
 e {\gamma_1} \ldots {\gamma_k}  & = &
  e r_1 {\gamma_1} r_{1 \cdot {\gamma_1}}' r_{1 \cdot {\gamma_1}} {\gamma_2} \ldots {\gamma_k} \\
& = &  \left(e r_1 {\gamma_1} r_{1 \cdot {\gamma_1}}'\right) e r_{1 \cdot {\gamma_1}} {\gamma_2} \ldots {\gamma_k} \\
& = & \ldots \\  
& = &  
\left(e r_1 {\gamma_1} r_{1 \cdot {\gamma_1}}'\right) 
\left(e r_{1 \cdot {\gamma_1}} {\gamma_2} r_{1 \cdot {\gamma_1} {\gamma_2}}'\right) 
\ldots 
\left(e r_{1 \cdot {\gamma_1} \ldots {\gamma_{k-1}}} 
{\gamma_k} 
r_{1 \cdot {\gamma_1} \ldots {\gamma_k}}'\right)  \\
& = &  
\left(e r_1 \overline{\gamma_1} r_{1 \cdot {\gamma_1}}'\right) 
\left(e r_{1 \cdot {\gamma_1}} \overline{\gamma_2} r_{1 \cdot {\gamma_1} {\gamma_2}}'\right) 
\ldots 
\left(e r_{1 \cdot {\gamma_1} \ldots {\gamma_{k-1}}} 
\overline{\gamma_k} 
r_{1 \cdot {\gamma_1} \ldots {\gamma_k}}'\right) \in \lb Y \rb,
\end{eqnarray*}
as required. This completes the proof that $Y$ is a finite generating set for $H$.   
\end{proof}

Now we move on to the finite presentability proof in the statement of Theorem~\ref{thm_monhybrid_inv}. Again, following the proof in \cite{ruvskuc1999presentations} for this we need to introduce a new alphabet $B$ corresponding to the generators $Y$ and then set up the so-called rewriting and representation mappings in order to be able to apply semigroup Reidemeister--Schrieier rewriting theory.      

For this we define
\[
B = \{ 
[j,w] :
w \in W, j \in J  \text{ and } j \cdot w \in J
\}
\]
and define the \emph{rewriting mapping}
\[
\phi: \{ 
(j,u): j \in J, u \in (A \cup A^{-1})^* \text{ and } j \cdot u \in J   
\} \rightarrow (B \cup B^{-1})^*
\]
inductively in the following way. When $u = \epsilon$ the empty word we define $\phi(j,\epsilon) = 1$.  When $u$ is not the empty word we decompose 
$u \equiv u'u''$ where $u'$ is the shortest non-empty prefix of $u$ with the property that $j \cdot u' \in J$.  
(Note that $u'$ is just a prefix of $u$ here and the dash notation has no relation to the $r_i'$ notation that was used above when defining coset representatives.)    
Then we inductively define 
\[ \phi(j,u) = \phi(j,u'u'') = 
\phi(j, u') \phi(j \cdot u', u'') = 
[j,\overline{u'}] \phi(j \cdot u', u'') \]
Here the bar map is defined, following the finite generation proof above, as follows. It follows from the definitions that $(er_j, er_ju')$ is a boundary pair so we define
\[\overline{u'} = w_{er_j, er_ju'} \in W. \]

Note that in the special case that $e r_j = er_ju'$ in $M$ it follows from the definitions that  
\[\overline{u'} = w_{er_j, er_ju'} = 1, \]
the empty word. Note in particular that the bar map sends $u'$ to a word $\overline{u'} \in W$ that satisfies 
\[
e r_j u' = e r_j \overline{u'}
\]
in the inverse monoid $M$. 
Note that this definition of bar matches exactly the definition of the bar map in the proof of Lemma~\ref{lem:fgBoundary}.

We want to use Theorem~2.1 from the paper \cite{CRRT} next, and for that we also need to define the, so-called, \emph{representation mapping}. 
We define the representation mapping 
\[
\psi: B^* \rightarrow A^*  
\]
to be the unique homomorphism extending the map
\[
[j,w] \mapsto e r_j w r_{j \cdot w}' 
\]
for each $[j,w] \in B$, where $e$ now denotes some fixed word over $A \cup A^{-1}$ that represents the unique idempotent in the $\gh$-class containing $H$.      

Next we state a helpful formula that follows from the definitions and shows that the map $\phi$ behaves like a homomorphism. 
Let $j \in J$ and let $w_1, w_2 \in (A \cup A^{-1})^*$ such that $j \cdot w_1 \in J$ and $j \cdot w_1w_2 \in J$. Then 
$j \cdot w_1 \in J$ and $(j \cdot w_1) \cdot w_2 \in J$ so $\phi(j \cdot w_1, w_2) \in B^*$ is defined and we have 
\begin{equation}\label{eqn:useful:phi:hom}
    \phi(j, w_1w_2) \equiv  \phi(j,w_1)\phi(j \cdot w_1, w_2) 
\end{equation}
in $B^*$. 

Let $\mathcal{L}(A,H)$ denote the set of all words from $(A\cup A^{-1})^*$ representing an element of $H$.
We define a map $\phi:\mathcal{L}(A,H) \to B^* $ by setting $\phi(w) = \phi(1,w)$. It is then straightforward to check that $\phi$ and $\psi$ are rewriting and representation mappings in the sense defined in \cite{CRRT}. 
Namely, for all $\gamma \in \mathcal{L}(A,H)$ we have that $\psi \phi(\gamma) = \gamma$ holds in $M$, which may be proved by checking that $\psi \phi(\gamma)$ equals exactly the expression appearing in the last line of the final displayed equation in the proof of Lemma~\ref{lem:fgBoundary}.    
Thus, it follows we can use \cite[Theorem~2.1]{CRRT} to write down an infinite monoid presentation for group $H$. After writing down that infinite presentation the remainder of the proof will be concerned with showing that there is a finite set of relations holding in $H$ from which all the defining relations in this infinite presentation are consequences.  

By \cite[Theorem~2.1]{CRRT} the group $H$ is defined by the monoid presentation 
with generating set $B$ and the following infinite families of defining relations:    
\begin{enumerate}
\item $b = \phi(\psi(b))$  for all $b \in B$, and $\phi(e) = 1$;      
\item $\phi(w_1 u w_2) = \phi(w_1 v w_2)$ where $w_1, w_2 \in (A \cup A^{-1})^*$, 
$(u=v) \in \mathfrak{R}$, $w_1 u w_2 \in \mathcal{L}(A,H)$;     
\item $\phi(w_1 \alpha \alpha^{-1} \alpha  w_2) = \phi(w_1 \alpha w_2)$ 
where $w_1, w_2, \alpha \in(A \cup A^{-1})^* $ and $w_1 \alpha w_2 \in \mathcal{L}(A,H)$;
\item $\phi(w_1 \alpha \alpha^{-1} \beta \beta^{-1}  w_2) = \phi(w_1 \beta \beta^{-1} \alpha \alpha^{-1}  w_2)$ 
where $w_1, w_2, \alpha \in(A \cup A^{-1})^* $ and we have $w_1 \alpha \alpha^{-1} \beta \beta^{-1}  w_2 \in \mathcal{L}(A,H)$;
\item $\phi(w_1 w_2) = \phi(w_1) \phi(w_2)$ where $w_1, w_2 \in \mathcal{L}(A,H)$.
\end{enumerate}
We shall now go through each of these five families of relations in turn proving that in each case they can be replaced and then derived as consequences of another finite set of defining relations.  
Before doing this we now state and prove two lemmas that will be useful for the proof.  

\begin{lem}\label{lem:claim2} 
Let $j \in J$ and $\alpha \in (A \cup A^{-1})^*$ such that $j \cdot \alpha \in J$. Then both  $\phi(j, \alpha) \in B^*$ and $\phi(j \cdot \alpha, \alpha^{-1}) \in B^*$ are defined and  
\[
\phi(j, \alpha)\phi(j \cdot \alpha, \alpha^{-1}) = 1.
\]
holds in $H$.  
Furthermore all of these relations are consequences of the following finite set of relations holding in $H$: 
$b_1=b_2$ for $b_1,b_2 \in B$ such that $\psi(b_1)=\psi(b_2)$ in $M$, and 
$b_1b_2=1$ for $b_1,b_2 \in B$ such that $\psi(b_1b_2)=1$ in $M$. 
\end{lem}
\begin{proof}
This lemma is proved by induction on the length of the word $\phi(j,\alpha) \in B^*$ using the inductive definition of the map $\phi$. 
The base case is when $j \cdot \alpha' \not\in J$ for any proper non-empty prefix $\alpha'$ of $\alpha$.
This implies that $(j \cdot \alpha) \cdot \gamma\not\in J$ for any proper non-empty prefix $\gamma$ of $\alpha^{-1}$.  
In this case $\phi(j,\alpha) = [j, w_{er_j,er_j\alpha}]$ and   
\[\phi(j \cdot \alpha, \alpha^{-1}) = 
[j \cdot \alpha, w_{er_{j \cdot \alpha}, er_{j \cdot \alpha}\alpha^{-1}}]  
= [j \cdot \alpha, w_{er_{j \cdot \alpha}\alpha^{-1},er_{j \cdot \alpha}}^{-1}]  
\] 
by definition of the words $w_{k,l}$ and Lemma \ref{lem:wprop}. Under the representation mapping we have 
\[
\psi([j, w_{er_j,er_j\alpha}]) = e r_j w_{er_j,er_j\alpha} r_{j \cdot \alpha}^{-1}   
\]
and
\[
\psi([j \cdot \alpha, w_{er_{j \cdot \alpha}\alpha^{-1},er_{j \cdot \alpha}}^{-1}]) = 
e r_{j \cdot \alpha} w_{er_{j \cdot \alpha}\alpha^{-1},er_{j \cdot \alpha}}^{-1} r_j^{-1}
\]
Since $er_j\alpha\alpha^{-1}=er_j$ and $er_{j \cdot \alpha}\alpha^{-1}\alpha = er_{j \cdot \alpha}$ holds in $M$, by Lemma \ref{lem:wprop} we have
$$
er_{j\cdot\alpha}w_{er_{j \cdot \alpha}\alpha^{-1},er_{j \cdot \alpha}}^{-1} = er_{j\cdot\alpha}w_{er_j\alpha\alpha^{-1},er_j\alpha}^{-1} = er_{j\cdot\alpha}w_{er_j,er_j\alpha}^{-1}
$$
in $M$. 
Multiplying the right-hand side of the expression for $\psi([j, w_{er_j,er_j\alpha}])$, the right-hand side of the previous equation, and $r_j^{-1}$, gives the following equalities in the inverse monoid $M$:
\begin{align*}
e r_j w_{er_j,er_j\alpha} r_{j \cdot \alpha}^{-1} e r_{j \cdot \alpha} w_{er_j,er_j\alpha}^{-1} r_j^{-1}
&=  
e r_j w_{er_j,er_j\alpha} r_{j \cdot \alpha}^{-1} r_{j \cdot \alpha} w_{er_j,er_j\alpha}^{-1} r_j^{-1} \\
&= 
e r_j w_{er_j,er_j\alpha} r_{j \cdot \alpha}^{-1} r_{j \cdot \alpha} w_{er_j,er_j\alpha}^{-1} r_j^{-1}
= e.
\end{align*}
The lemma then follows by a straightforward inductive argument using the definition of the mapping $\phi$. 
The last sentence in the lemma follows by observing that the  
two families of relations 
\[
[j \cdot \alpha, w_{er_{j \cdot \alpha}\alpha^{-1},er_{j \cdot \alpha}}^{-1}] = [j \cdot \alpha, w_{er_j,er_j\alpha}^{-1}]
\]
and
\[
[j, w_{er_j,er_j\alpha}] [j \cdot \alpha, w_{er_j,er_j\alpha}^{-1}]=1
\]
are in the finite set of relations given in the statement of the lemma. 
This completes the proof. 
\end{proof}

\begin{lem}\label{lem:claim1}
Let $j \in J$ and $\alpha \in (A \cup A^{-1})^*$ such that for every proper non-empty prefix $\alpha'$ of $\alpha$ we have $j \cdot \alpha' \not\in J$. Furthermore suppose that $er_j \alpha = er_j$ 
Then $\phi(j,\alpha) = 1$ holds in the group $H$. Furthermore, there is a finite set of relations holding in $H$ from which all these relations are consequences.         
\end{lem}

\begin{proof}
It follows from the definitions that
\[
\phi(j,\alpha) = [j,w_{er_j,er_j\alpha}] = [j,1].   
\]
Then in the defining relations we have, by applying \eqref{eqn:useful:phi:hom} and noting that $1\cdot e=1$,
\[
[j,1] = \phi\psi([j,1]) = \phi(e r_j 1 r_j') = 
\phi(1,e r_j  r_j^{-1}) = 
\phi(1,e)\phi(1,r_j) \phi(j,r_j^{-1}).     
\]
The fact that 
$\phi(1,r_j) \phi(j,r_j^{-1})=1$ holds and can be deduced from a finite set of relations follows from Lemma~\ref{lem:claim2}. Furthermore $\phi(1,e)=1$ is a defining relation in the presentation.  
This completes the proof of the lemma. 
\end{proof}

We now have everything in place to prove Theorem~\ref{thm_monhybrid_inv}.

\begin{proof}[Proof of Theorem~\ref{thm_monhybrid_inv}]
From the discussion above, and maintaining that notation, the group $H$ is defined by an infinite monoid presentation with finite generating set $B$ and the five families of relations (1)--(5) listed above.    
To prove the theorem it will suffice to work through each of theses five families of relations in turn and show that for each of them there is a finite set of relations over $B$ holding in $H$ from which they can all be derived as consequences.   

\smallskip

(1) Since $B$ is a finite, this is a finite set of defining relations. 

\smallskip

(2) Consider the relation $\phi(w_1 u w_2) = \phi(w_1 v w_2)$ where   $w_1, w_2 \in (A \cup A^{-1})^*$,  $(u=v) \in \mathfrak{R}$, $w_1 u w_2 \in \mathcal{L}(A,H)$.
Decompose the word $w_1 u w_2$ as 
\[ w_1 u w_2 \equiv w_1' w_1'' u w_2' w_2''    \]
such that 
$w_1'$ is the longest prefix of $w_1$ such that $1 \cdot w_1' \in J$, and
$w_2'$ is the shortest prefix of $w_2$ 
such that   
$1 \cdot w_1 u w_2' \in J$.
Since $u=v$ in $M$ implies $w_1 u = w_1 v$ in $M$, it follows that when we decompose the word 
$w_1 v w_2$ in the same way we actually obtain  
\[ w_1 v w_2 \equiv w_1' w_1'' v w_2' w_2'' \] 
with the same decompositions of $w_1 \equiv w_1'w_1''$ and $w_2 \equiv w_2' w_2''$.  
From this and \eqref{eqn:useful:phi:hom} we conclude that 
\[
\phi(w_1uw_2) \equiv 
\phi(1,w_1') \phi(1\cdot w_1', w_1'' u w_2') \phi(1\cdot w_1uw_2', w_2'')
\]
and 
\[
\phi(w_1vw_2) \equiv 
\phi(1,w_1') \phi(1\cdot w_1', w_1'' v w_2') \phi(1\cdot w_1vw_2', w_2'').
\]
But $u=v$ in $M$ implies $1\cdot w_1uw_2' = 1\cdot w_1vw_2'$ and thus 
$\phi(1\cdot w_1uw_2', w_2'') \equiv \phi(1\cdot w_1vw_2', w_2'')$. 
Thus the relation  
$\phi(1\cdot w_1', w_1'' u w_2') = \phi(1\cdot w_1', w_1'' v w_2')$ 
holds in $H$ and the relation   
$\phi(w_1uw_2) = \phi(w_1vw_2)$ is a consequence of it.  
Furthermore it follows from the definition of $w_1'$ and $w_2'$ that   
$|\phi(1\cdot w_1', w_1'' u w_2')| \leq |u|$ and  
$|\phi(1\cdot w_1', w_1'' v w_2')| \leq |v|$.
This therefore gives a finite set of relations over $B$ that hold in $H$ from which all the relations in the family (2) are consequences. 

\smallskip

(5) Jumping to the set of relations (5) for a moment, these are in fact all trivial relations (that is, equalities as words) since from Equation~\eqref{eqn:useful:phi:hom} we have   \[
\phi(w_1w_2) \equiv \phi(1, w_1)\phi(1 \cdot w_1,w_2) 
\equiv \phi(1,w_1)\phi(1,w_2) \equiv \phi(w_1)\phi(w_2).
\]
Thus we can remove all of the relations (5) from the presentation.

\smallskip

This just leaves the two families of relations (3) and (4) to consider.

\smallskip

(3) $\phi(w_1 \alpha \alpha^{-1} \alpha  w_2) = \phi(w_1 \alpha w_2)$ where $w_1, w_2, \alpha \in(A \cup A^{-1})^* $ and $w_1 \alpha w_2 \in \mathcal{L}(A,H)$.

\smallskip

We decompose   
\[
w_1 \alpha w_2
\equiv 
w_1'w_1'' \alpha w_2'w_2''
\] 
such that
$w_1 \equiv w_1'w_1''$, $w_2 \equiv w_2' w_2''$;     
$w_1'$ is the longest prefix of $w_1$ such that $ew_1' \in X$ (or equivalently $1 \cdot w_1' \in J$); $w_2'$ is the shortest prefix of $w_2$ such that $ew_1\alpha w_2' \in X$. 

First suppose that $e w_1\alpha' \not\in X$ for every prefix $\alpha'$ of $\alpha$.
Then  
\begin{align*} 
\phi(w_1 \alpha w_2) & \equiv  
\phi(1,w_1')\phi(1\cdot w_1', w_1'' \alpha w_2')\phi(1\cdot w_1 \alpha w_2', w_2''),   
\end{align*}
while 
\begin{align*} 
\phi(w_1 \alpha\alpha^{-1}\alpha  w_2) 
& \equiv   \phi(1,w_1')\phi(1\cdot w_1', w_1'' \alpha\alpha^{-1}\alpha w_2')\phi(1\cdot w_1 \alpha\alpha^{-1}\alpha w_2', w_2'') \\   
& \equiv   \phi(1,w_1')\phi(1\cdot w_1', w_1'' \alpha\alpha^{-1}\alpha w_2')\phi(1\cdot w_1 \alpha w_2', w_2''),    
\end{align*}
where 
$\phi(1\cdot w_1', w_1'' \alpha w_2') \in B$
and 
$\phi(1\cdot w_1', w_1'' \alpha\alpha^{-1}\alpha w_2') \in B$. 
So in this case the relation is a consequence of the finite set of relations holding in $H$ given in the
statement of Lemma~\ref{lem:claim2}.

Otherwise, we decompose 
\[
\alpha \equiv \alpha_1 \alpha_2 \alpha_3 
\]
such that 
$\alpha_1$ is the shortest prefix of $\alpha$ such that $e w_1\alpha_1 \in X$ and      
$\alpha_1\alpha_2$ is the longest prefix of $\alpha$ such that $ew_1\alpha_1\alpha_2 \in X$.   
Note that now $w_2'$ is also the shortest prefix of $w_2$ such that $ew_1\alpha\alpha^{-1} \alpha w_2' \in X$. 
We can now use this decomposition to compute $\phi(w_1 \alpha \alpha^{-1} \alpha  w_2)$ and $\phi(w_1 \alpha w_2)$ and compare them. Computing the former, we have
\begin{align*} 
\phi(w_1 \alpha \alpha^{-1} \alpha  w_2) & \equiv 
 \phi(1, w_1') \phi(1 \cdot w_1', w_1''\alpha_1)\phi(1 \cdot w_1 \alpha_1, \alpha_2) \\ 
& \quad\quad\, \phi(1 \cdot w_1 \alpha_1\alpha_2, \alpha_3\alpha_3^{-1}) \phi(1 \cdot w_1 \alpha \alpha_3^{-1}, \alpha_2^{-1}) \\
& \quad\quad\quad \phi(1 \cdot w_1 \alpha \alpha_3^{-1}\alpha_2^{-1}, \alpha_1^{-1} \alpha_1)
  \phi(1 \cdot w_1 \alpha \alpha^{-1} \alpha_1, \alpha_2) \\
& \quad\quad\quad\quad  \phi(1 \cdot w_1 \alpha \alpha^{-1} \alpha_1 \alpha_2, \alpha_3 w_2') 
\phi(1 \cdot w_1 \alpha \alpha^{-1} \alpha w_2', w_2''). 
\end{align*}
Next observe that from the definitions it follows that Lemma~\ref{lem:claim1} may be applied to show that $\phi(1 \cdot w_1 \alpha_1\alpha_2, \alpha_3\alpha_3^{-1})=1$ and $\phi(1 \cdot w_1 \alpha \alpha_3^{-1}  \alpha_2^{-1}, \alpha_1^{-1} \alpha_1)=1$ and both of these relations are consequences of the finite set of relations holding in $H$ given in the statement of Lemma~\ref{lem:claim1}.
Furthermore, it follows from the definitions that Lemma~\ref{lem:claim2} may be applied to deduce that 
\[
\phi(1 \cdot w_1 \alpha_1, \alpha_2) \phi(1 \cdot w_1 \alpha \alpha_3^{-1}, \alpha_2^{-1}) = 1
\]
holds in the $H$ and it is a consequence of the finite set of relations holding in $H$ identified in the statement of Lemma~\ref{lem:claim2}.  
We also have
\[\phi(1 \cdot w_1 \alpha \alpha^{-1} \alpha_1, \alpha_2) \equiv \phi(1 \cdot w_1 \alpha_1, \alpha_2), \]   
\[\phi(1 \cdot w_1 \alpha \alpha^{-1} \alpha_1 \alpha_2, \alpha_3 w_2') \equiv \phi(1 \cdot w_1 \alpha_1 \alpha_2, \alpha_3 w_2') \]
and
\[ \phi(1 \cdot w_1 \alpha \alpha^{-1} \alpha w_2', w_2'') \equiv \phi(1 \cdot w_1 \alpha w_2', w_2''). \]
Now from all the observations and claims above it follows that after including the finitely many relations from Lemmas~\ref{lem:claim1} and \ref{lem:claim2} in our presentation, 
all of the relations of type (3) will follow from that finite collection, since from the observations above we can use those relations to deduce:  
\begin{align*} 
\phi(w_1 \alpha \alpha^{-1} \alpha  w_2) &= \phi(1, w_1') \phi(1 \cdot w_1', w_1''\alpha_1)
 \phi(1 \cdot w_1 \alpha_1, \alpha_2) \\
& \quad\quad\quad
\phi(1 \cdot w_1 \alpha_1 \alpha_2, \alpha_3 w_2') 
\phi(1 \cdot w_1 \alpha w_2', w_2'')  \\
& \equiv \phi(w_1 \alpha  w_2).  
\end{align*}
This deals with the family of relations (3). 

\smallskip

(4) $\phi(w_1 \alpha \alpha^{-1} \beta \beta^{-1}  w_2) = \phi(w_1 \beta \beta^{-1} \alpha \alpha^{-1}  w_2)$ where $w_1, w_2, \alpha \in(A \cup A^{-1})^* $ and we have $w_1 \alpha \alpha^{-1} \beta \beta^{-1}  w_2 \in \mathcal{L}(A,H)$. 

\smallskip

For each such relation we shall show that there is a relation over $B$ between a pair of words each of length at most three, from which the original relation is a consequence when taken together with the finite families of relations from Lemmas~\ref{lem:claim1} and \ref{lem:claim2}. Then since $B$ is finite this allows us to replace this infinite family of relations by a finite family of relations. 
There are several cases depending on 
whether or not $\alpha$ has a prefix $\alpha'$ satisfying $ew_1\alpha' \in X$ and 
whether $\beta$ has a prefix $\beta'$ satisfying $ew_1\beta' \in X$. We will give the proof for the case that such prefixes $\alpha'$ and $\beta'$ both exist. The other cases can be handled similarly as in the proof of (3) above.  

Consider the word $w_1 \alpha \alpha^{-1} \beta \beta^{-1} w_2$. We decompose this word as  
\[w_1' w_1'' \alpha_1 \alpha_2 \alpha_3 \alpha_3^{-1} \alpha_2^{-1} \alpha_1^{-1} \beta_1 \beta_2 \beta_3 \beta_3^{-1} \beta_2^{-1} \beta_1^{-1} w_2' w_2'' \]
where
$w_1 \equiv w_1'w_1''$, $w_2 \equiv w_2' w_2''$, $\alpha \equiv \alpha_1 \alpha_2 \alpha_3$ and $\beta \equiv \beta_1 \beta_2 \beta_3$;     
$w_1'$ is the longest prefix of $w_1$ such that $ew_1' \in X$ (or equivalently $1 \cdot w_1' \in J$);      
$w_2'$ is the shortest prefix of $w_2$ such that $ew_1\alpha\alpha^{-1} \beta \beta^{-1} w_2' \in X$;
$\alpha_1$ is the shortest prefix of $\alpha$ such that $e w_1\alpha_1 \in X$;      
$\alpha_1\alpha_2$ is the longest prefix of $\alpha$ such that $ew_1\alpha_1\alpha_2 \in X$;   
$\beta_1$ is the shortest prefix of $\beta$ such that $e w_1\alpha \alpha^{-1} \beta_1 \in X$; 
$\beta_1\beta_2$ is the longest prefix of $\beta$ such that $e w_1\alpha \alpha^{-1} \beta_1 \beta_2 \in X$. 

It follows from the definitions and assumptions that $\beta_1$ is also the shortest prefix of $\beta$ such that $e w_1 \beta_1 \in X$ since $e w_1 \alpha \alpha^{-1} = ew_1$ in the monoid since by assumption $w_1 \alpha \alpha^{-1} \beta \beta^{-1}  w_2 \in \mathcal{L}(A,H)$ and so all of prefixes $p$ of the word $w_1 \alpha \alpha^{-1} \beta \beta^{-1}  w_2$ have the property that $ep\, \R \, e$. 
It also follows from the definitions and assumptions that $\beta_1\beta_2$ is the longest prefix of $\beta$ such that $e w_1 \beta_1\beta_2 \in X$ since $e w_1 \alpha \alpha^{-1} = ew_1$ in the monoid. 
Finally, note that it follows from the definitions and assumptions that $w_2'$ is also the shortest prefix of $w_2$ such that $w_1w_2' \in X$ and also the shortest prefix of $w_2$ such that $w_1\beta \beta^{-1} \alpha \alpha^{-1} w_2' \in X$.
We can now use this decomposition to compute $\phi(w_1 \alpha \alpha^{-1} \beta \beta^{-1}  w_2)$ and $\phi(w_1 \beta \beta^{-1} \alpha \alpha^{-1}  w_2)$ and compare them. We have
\begin{align*} 
\phi(w_1 \alpha \alpha^{-1} \beta \beta^{-1}  w_2) & \equiv 
 \phi(1, w_1') \phi(1 \cdot w_1', w_1''\alpha_1)\phi(1 \cdot w_1 \alpha_1, \alpha_2) \\ 
& \quad\quad\, \phi(1 \cdot w_1 \alpha_1\alpha_2, \alpha_3\alpha_3^{-1})\phi(1 \cdot w_1 \alpha \alpha_3^{-1}, \alpha_2^{-1}) \\
& \quad\quad\quad
\phi(1 \cdot w_1 \alpha \alpha_3^{-1}\alpha_2^{-1}, \alpha_1^{-1} \beta_1) \phi(1 \cdot w_1 \alpha \alpha^{-1} \beta_1, \beta_2) \\
& \quad\quad\quad\quad
\phi(1 \cdot w_1 \alpha \alpha^{-1} \beta_1 \beta_2, \beta_3 \beta_3^{-1}) 
\phi(1 \cdot w_1 \alpha \alpha^{-1} \beta \beta_3^{-1}, \beta_2^{-1})\\
& \quad\quad\quad\quad\quad \phi(1 \cdot w_1 \alpha \alpha^{-1} \beta \beta_3^{-1}\beta_2^{-1}, \beta_1^{-1}w_2')\\
& \quad\quad\quad\quad\quad\quad \phi(1 \cdot w_1 \alpha \alpha^{-1} \beta \beta^{-1} w_2', w_2'').
\end{align*}

Next observe that from the definitions it follows that Lemma~\ref{lem:claim1} may be applied to show that 
$\phi(1 \cdot w_1 \alpha_1\alpha_2, \alpha_3\alpha_3^{-1}) = 1$
 and 
$\phi(1 \cdot w_1 \alpha \alpha^{-1} \beta_1 \beta_2, \beta_3 \beta_3^{-1})=1$
 and both of these relations are consequences of the finite set of relations holding in $H$ given in the statement of Lemma~\ref{lem:claim1}.
Furthermore, it follows from the definitions that Lemma~\ref{lem:claim2} may be applied to deduce that 
\[
\phi(1 \cdot w_1 \alpha_1, \alpha_2)\phi(1 \cdot w_1 \alpha \alpha_3^{-1}, \alpha_2^{-1})=1 
\]
and also 
\[
\phi(1 \cdot w_1 \alpha \alpha^{-1} \beta_1, \beta_2)\phi(1 \cdot w_1 \alpha \alpha^{-1} \beta \beta_3^{-1}, \beta_2^{-1})=1
\]
hold in the $H$ and they are consequences of the finite set of relations holding in $H$ identified in the statement of Lemma~\ref{lem:claim2}.  

It follows that the equality  
\begin{align*} 
 \phi(w_1 \alpha \alpha^{-1} \beta \beta^{-1}  w_2) & =
\phi(1, w_1') \phi(1 \cdot w_1', w_1''\alpha_1) 
\phi(1 \cdot w_1 \alpha \alpha_3^{-1}\alpha_2^{-1}, \alpha_1^{-1} \beta_1) \\
& \quad \quad \phi(1 \cdot w_1 \alpha \alpha^{-1} \beta \beta_3^{-1}\beta_2^{-1}, \beta_1^{-1}w_2') 
\phi(1 \cdot w_1 \alpha \alpha^{-1} \beta \beta^{-1} w_2', w_2'')
\end{align*}
holds in $H$ and is a consequence of the  finitely many relations from Lemmas~\ref{lem:claim1} and \ref{lem:claim2} in our presentation.  

It then follows from the definitions and observations above that if carry out the same process on the word $w_1 \beta \beta^{-1} \alpha \alpha^{-1} w_2$ we obtain that  
\begin{align*} 
\phi(w_1 \beta \beta^{-1} \alpha \alpha^{-1}  w_2) & = 
\phi(1, w_1') \phi(1 \cdot w_1', w_1''\beta_1)
 \phi(1 \cdot w_1 \beta \beta_3^{-1}\beta_2^{-1}, \beta_1^{-1} \alpha_1)  \\
& \quad \quad \phi(1 \cdot w_1 \beta \beta^{-1} \alpha \alpha_3^{-1}\alpha_2^{-1}, \alpha_1^{-1}w_2') 
\phi(1 \cdot w_1 \beta \beta^{-1} \alpha \alpha^{-1} w_2', w_2'') 
\end{align*}
holds in $H$ and is a consequence of the  finitely many relations from Lemmas~\ref{lem:claim1} and \ref{lem:claim2} in our presentation.  
From the definitions it follows that
\[
\phi(1 \cdot w_1 \alpha \alpha^{-1} \beta \beta^{-1} w_2', w_2'') \equiv \phi(1 \cdot w_1 \beta \beta^{-1} \alpha \alpha^{-1} w_2', w_2'') 
\]
in $B^*$. Hence we have that the relation   
\begin{align*} 
& \phi(1 \cdot w_1', w_1''\alpha_1) \phi(1 \cdot w_1 \alpha \alpha_3^{-1}\alpha_2^{-1}, \alpha_1^{-1} \beta_1) \phi(1 \cdot w_1 \alpha \alpha^{-1} \beta \beta_3^{-1}\beta_2^{-1}, \beta_1^{-1}w_2') \\ 
& \quad = \phi(1 \cdot w_1', w_1''\beta_1) \phi(1 \cdot w_1 \beta \beta_3^{-1}\beta_2^{-1}, \beta_1^{-1} \alpha_1) \phi(1 \cdot w_1 \alpha \beta^{-1} \alpha \alpha_3^{-1}\alpha_2^{-1}, \alpha_1^{-1}w_2') 
  \end{align*}
holds in the group $H$ and the relation $\phi(w_1 \alpha \alpha^{-1} \beta \beta^{-1}  w_2) = \phi(w_1 \beta \beta^{-1} \alpha \alpha^{-1}  w_2)$ is a consequence of it. But it follows from the definitions that both sides of this relation have length at most three as words over the alphabet $B$. Since $B$ is finite this means that by taking all relations that hold in $H$ and have length bounded above by six (together with the finite collection of relations arising from Lemmas~\ref{lem:claim1} and \ref{lem:claim2}) 
we will have a finite set of relations holding in $H$ from which all the relations in the set (4) are consequences.

\smallskip

This covers all the relations in the infinite presentation for the group $H$ and thus completes the proof that $H$ is finitely presented as a monoid, and so completes the proof of the theorem that $H$ is a finitely presented group.    
\end{proof}

\begin{rmk} 
There is an analogue of Theorem~\ref{thm_monhybrid_inv} where $S$ is a finitely presented monoid that can be proved using similar methods. We do not include that 
result here since the focus of this paper is on inverse monoid presentations. Note that since a finitely presented inverse monoid need not be finitely presented as 
a monoid, the two results require separate proofs like the corresponding pair of results in \cite{ruvskuc1999presentations}.
It turns out in fact that for any finitely presented special monoid $M$ if $H$ is the group of units of $M$ then $H$ has finite boundary width. This gives an alternative proof of a result of Makanin \cite{Mak} that the group of units of a finitely presented special monoid is a finitely presented group.         
One of the key differences between the theory of special inverse monoids and special monoids is that in general the group of units will not have finite boundary width, since in general it need not be finitely presented.   
\end{rmk}

\begin{rmk} 
Theorem~\ref{thm_monhybrid_inv} generalises \cite[Corollary 4.7]{ruvskuc1999presentations} since if a subgroup has finite index in the sense defined there then one can take the its complete set of right cosets as the finite cover with finite boundary width. 
\end{rmk}


\section{Finite presentability of the group of units in the case when the right units admit a free product decomposition} \label{sec:freeprod}

We now apply the general results from the previous section to prove the following result announced earlier.  

\begin{thm} \label{thm:nofreeprod}
Let $M$ be a finitely presented special inverse monoid with group of units $U$. Suppose that $\mathrm{RU}(M)\cong U\ast T$ for some finitely generated
monoid $T$ with a trivial group of units. Then $U$ must be finitely presented. 
\end{thm}

\begin{cor}
Let $M$ be a finitely presented special inverse monoid with group of units $U$. Suppose that $\mathrm{RU}(M)\cong U\ast X^*$ for some finite set $X$.
Then $U$ (and so $\mathrm{RU}(M)$) must be finitely presented.
\end{cor}

As we have mentioned in the previous section, given a special inverse monoid 
$$
M=\Inv\pre{A}{w_1=1,\dots,w_k=1}
$$
there are two graphs that one can naturally associate with the right units $\mathrm{RU}(M)$ (coinciding with $R_1$, the $\R$-class of $1$ in $M$). 
Firstly there is $S\Gamma_A(1)$, the right \Sch graph of $1$ (labelled by the generators of $M$ and their inverses). 
Secondly since, as explained in the paragraph after Definition~\ref{defn:ru}, the right units $\mathrm{RU}(M)$ are finitely generated by the set of all elements represented by 
prefixes of the relator words, there is the Cayley graph of the monoid $\mathrm{RU}(M)$ with respect to this finite generating set. 
Here, the right Cayley graph of a monoid $T$ with respect to a finite generating set $B$ has vertex set $T$ and directed edges from $t$ to $tb$ for all $t \in T$ and $b \in B$.   
In fact we only want to work with the underlying undirected graphs of each of these directed labelled graphs.

We want to prove the following key result which shows a close relationship between the geometries of these two graphs.

\begin{lem}\label{lem:QSIR1} 
Let $M$ be a finitely presented special inverse monoid  
\[
M = \Inv\pre{A}{w_1=1,\dots,w_k=1}.
\]
Let $\Gamma$ be the underlying undirected graph of $S\Gamma_A(1)$ and let $\Delta$ be the underlying undirected graph of the right Cayley graph of the submonoid 
$R_1=\mathrm{RU}(M)$ of right units with respect to the finite generating set $X$ of all prefixes of all defining relations $w_i$, $1 \leq i \leq k$.
Then the identity map $\iota: R_1 \rightarrow R_1$ defines a quasi-isometry $\iota: \Gamma \rightarrow \Delta$.    
\end{lem}

\begin{proof} 
Since $\iota$ is a bijection it is immediate that the map $\iota:\Gamma \rightarrow \Delta$ is quasi-dense with quasi-density constant $\mu=0$.   

Next, consider an adjacent pair of vertices in $\Delta$. This pair has the form $\{ x, xp \}$ where $x \in R_1$ and $p$ is a prefix of one of the defining relations. 
Hence $d_{\Gamma}(x,xp) \leq \lambda$ where $\lambda$ is the maximum length of a defining relator in the finite presentation for $M$. Then by induction for any pair of 
vertices $x,y \in R_1$ we have 
\[d_{\Gamma}(x,y) \leq \lambda\cdot d_{\Delta}(x,y).\]  

For the inequality in the other direction it is best to consider the construction of $S\Gamma(1)$ via Stephen's procedure. Let $T$ be the directed labelled graph obtained 
by starting with a single vertex, adjoining loops labelled by $w_i$ at that vertex (for each $1 \leq i \leq k$) and then repeating this iteratively for every newly created 
vertex. Then it follows from \cite{Stephen} that $S\Gamma(1)$ is the graph obtained by bi-determinising the graph $T$. This means that for every edge $e$ in the graph 
$S\Gamma(1)$ we can choose some preimage edge $e'$ in $T$  such that $e'$ maps to $e$ after bi-determinising $T$. This is because no new edges are created when 
bi-determinising $T$ so every edge in $S\Gamma(1)$ has at least one preimage in $T$. Now we can read a product of prefixes of defining relations in the tree of balls 
that traverses $e$ as its last edge. It follows that given any edge $e = \{ x, xa \}$ in $\Gamma = S\Gamma(1)$ we can write 
\begin{align*} 
x  &=_M p_1 p_2 \ldots p_k \\
xa &=_M p_1 p_2 \ldots (p_ka) 
\end{align*}
where all $p_i$ and $p_ka$ are prefixes of defining relators. From this it follows that 
\[
d_{\Delta}(x,xa) \leq 2.
\]
Since in $\Delta$ we have edges 
\[
p_1 \ldots p_{k-1} p_k
\xleftarrow{p_k}
p_1 \ldots p_{k-1} 
\xrightarrow{p_k a}
p_1 \ldots p_{k-1} p_k a 
\]  
where $p_k, p_ka \in X$ are prefixes of defining relators. It follows by induction that 
\[d_{\Delta}(x,y) \leq 2\cdot d_{\Gamma}(x,y).\]  

Putting the two inequalities together along with the quasi-density observation this completes the proof that the identity map defines a quasi-isometry between $\Gamma$ and $\Delta$.  
\end{proof}

\begin{lem}\label{lem:ChangeGenSet}
Let $M$ be a monoid generated by finite generating sets $A$ and $B$. Then the identity map on $M$ defines a quasi-isometry between the undirected right Cayley graph of $M$ with 
respect to $A$ and the undirected right Cayley graph of $M$ with respect to $B$.  
\end{lem}

\begin{proof}
This is straightforward to prove using an argument analogous to the proof of \cite[Proposition~4]{gray2013groups}.       
\end{proof}

The remaining step is as follows.

\begin{lem}\label{lem:free:prod:boundary} 
Let $S = U \ast V$ be a free product of two finitely generated monoids $U$ and $V$ and let $\Gamma$ be the undirected Cayley graph of $S$ with respect to a generating set $X \cup Y$ 
where $X \subseteq U$ generates $U$ and $Y \subseteq V$ generates $V$. Then $U \subseteq \Gamma$ has finite boundary width.       
\end{lem}

\begin{proof}
It follows from the normal forms in the free product $S = U \ast V$ that if $(x_0, \ldots, x_k)$ is a boundary path for $U$ in $\Gamma$ with $x_0, x_k \in U$ then we must have 
$x_0 = x_k$. Hence $U$ has finite boundary width $0$ in this case.         
\end{proof}

\begin{proof}[Proof of Theorem \ref{thm:nofreeprod}]
Let $R = \mathrm{RU}(M)$ and fix an isomorphism $\phi: U \ast T \rightarrow R$.  Since by assumption $T$ has a trivial group of units, it follows that the group of units of $U \ast T$ is $U$. 
Hence $\phi$ maps the first factor $U$ of $U \ast T$ isomorphically to the group of units $U$ of $R$. 
Let $\Omega$ by the underlying undirected graph of the right Cayley graph of $R$ with respect to some finite generating set of the form $\phi(Y) \cup \phi(X)$ where $Y$ is a finite generating 
set for $U$ and $X$ is a finite generating set for $T$ where $1 \not\in X$. Let $\Gamma = S\Gamma_A(1)$ be the \Sch graph of $R$.  
Combining Lemma~\ref{lem:QSIR1} and Lemma~\ref{lem:ChangeGenSet} the identity map $\iota$ on $R$ defines a quasi-isometry between $\Omega$ and $\Gamma$.  Now $\iota$ maps $U=H_1 \subseteq R_1$ 
to the factor $U$ of  $U \ast T$, and since in $\Omega$ the factor $U$ has finite boundary width (by Lemma~\ref{lem:free:prod:boundary} due to the free product structure $U \ast T$) it follows 
from  Lemma~\ref{lem_QSI} and Proposition~\ref{prop_SC} that $U=H_1 \subseteq \Gamma$ has a finite cover with finite boundary width. It then follows from Theorem~\ref{thm_monhybrid_inv}  
that $U$ is finitely presented, as required. 
\end{proof} 

Since any finitely presented group has also a finite monoid presentation, we get the result announced in the introduction.

\begin{cor}
If $H$ is a finitely generated but not finitely presented group then $H\ast C^*\not\in\mathcal{RU}$ for any finite set $C$.
\end{cor}


\section{The right unit monoids in the Gray-Kambites construction} \label{sec:gk}

As we have announced earlier, the principal aim in this section is to prove the following result which generalises \cite[Theorem~4.1]{GK}.

\begin{thm}\label{thm:dense}
Every finitely generated submonoid of a finitely RC-presented mo\-noid is isomorphic to a 
submonoid $N$ of a finitely presented special inverse monoid $M$ such that 
\begin{itemize}
    \item[(a)] $N$ is a submonoid of the right units of $M$, and 
    \item[(b)] $N$ contains the group of units of $M$. 
\end{itemize}
\end{thm}

Note in the case that $N$ is a group properties (a) and (b) in Theorem~\ref{thm:dense} imply that $N$ is the group of units of the inverse monoid $M$, and then since every finitely presented group 
is clearly a finitely RC-presented monoid, Theorem~\ref{thm:dense} recovers the result \cite[Theorem~4.1]{GK}.  

\begin{rmk}
Note that the conditions in the statement of Theorem~\ref{thm:dense} imply that the submonoid $N$ that we construct is a union of $\gh$-classes of $M$ that are themselves all contained in the 
$\gr$-class of $1$ in $M$. 
\end{rmk}
 
\begin{rmk}
Note that if $M$ is a finitely presented special inverse monoid and $N$ is a finitely generated submonoid of the right units of $M$ such that $N$ contains the 
group of units of $M$, then $N \in \mathcal{RC}_2$. Thus, if $\mathcal{RC}_1 = \mathcal{RC}_2$ then it would follow from Theorem~\ref{thm:dense} that the class 
of finitely generated submonoids of right units of finitely presented special inverse monoids containing the group of units is in fact equal to the class of recursively RC-presented monoids.  
\end{rmk}
        
We begin by recalling the original construction from the paper \cite{GK}, based on a finite monoid presentation of a group and a finitely generated subgroup.
Let $G$ a finitely presented group and $H \leq G$ a finitely generated subgroup. One can assume that $G$ is given by a finite monoid presentation 
$$G=\Mon\pre{A}{r_1=1,\dots,r_k=1},$$ 
that $H$ is generated by the subset $B\subseteq A$, and, moreover, that for any letter $a\in A$ we have at our disposal a formal inverse $\ol{a}\in A$ with $\ol{\ol{a}}=a$ 
while the relation $a\ol{a}=1$ is included in the set of defining relations of $G$. With these assumptions the authors in \cite{GK} introduce a special inverse monoid $M_{G,H}$
generated by $A$ and $p_0,p_1,\dots,p_k,z,d$ and defined by the relations
\begin{align*}
& p_i a p_i^{-1} p_i a^{-1} p_i^{-1} = 1 & (a\in A,\ i=0,1,\dots,k) \\
& p_i r_i d^{-1} p_i^{-1} = 1 & (i=0,1,\dots,k) \\
& p_0dp_0^{-1} = 1 & \\
& zbz^{-1}zb^{-1}z^{-1} = 1 & (b\in B) \\
& z \left(\prod_{i=0}^k p_i^{-1}p_i\right) z^{-1} = 1.
\end{align*}
It was shown in \cite{GK} that the group of units of $M_{G,H}$ is isomorphic to $H$ and that it consists precisely of elements represented by words of the form 
$zwz^{-1}$, $w\in B^*$, with 
$$(zuz^{-1})(zvz^{-1})=zuvz^{-1}$$ 
holding in $M_{G,H}$ for all $u, v \in B^*$. This was the key step in establishing the result that the class of groups of units
of finitely presented special inverse monoids coincides with the finitely generated, recursively presented groups.

Furthermore, a directed deterministic graph $\Omega$ edge-labelled by the generating symbols of $M_{G,H}$ was constructed in the same paper that turns out to be
a copy of the \Sch graph  $S\Gamma(1)$ of $M_{G,H}$ (Actually, it was remarked in \cite{GK} that a bit more work is needed to prove this labelled graph isomorphism, 
even though many pertinent properties of $S\Gamma(1)$ were actually proved in the paper.) 

Now we present a generalisation of the previous construction which takes a finitely RC-presented right cancellative monoid and a finitely generated submonoid. In this more general setup, 
we determine an RC-presentation for the corresponding monoid of right units and show that in general it need not be finite.

\begin{defn} \label{def:mST}
Let $S$ be a finitely RC-presented monoid and let $T$ be a finitely generated submonoid of $S$. Let 
$$
S = \MonRC\pre{A}{u_i=v_i\; (1\leq i\leq k)}
$$
be a finite RC-presentation for $S$ such that $T$ is generated by a finite subset $B$ of $A$.  
This can always be achieved by adding in relations of the form $b = w_b \in A^*$ holding in $S$, one for each $b \in B$, which transforms the original RC-presentation into one with the desired property. 

Then we define $M_{S,T}$ to be the finitely presented special inverse monoid generated by $\Sigma=A\cup\{p_0,p_1,\dots,p_k,z,d\}$ subject to the defining relations
\begin{align*}
& p_i a p_i^{-1} p_i a^{-1} p_i^{-1} = 1 & (a\in A,\ i=0,1,\dots,k) \\
& p_i u_i d^{-1} v_i^{-1} p_i^{-1} = 1 & (i=1,\dots,k) \\
& p_0dp_0^{-1} = 1 & \\
& zbz^{-1}zb^{-1}z^{-1} = 1 & (b\in B) \\
& z \left(\prod_{i=0}^k p_i^{-1}p_i\right) z^{-1} = 1.
\end{align*}
\end{defn}

From \cite{IMM,GR} it follows that the elements represented by all prefixes of all relator words of a special inverse monoid form a generating set for the monoid
of its right units. 
Applying \cite[Corollary 3.2]{Gr-Inv} to replace prefixes of relator words by their reduced forms together with the fact that if $uv=1$ is a defining relator then the prefix $u$ is equal to $v^{-1}$ in the monoid, it may be shown that every prefix of a relator word is equal in $M_{S,T}$ to one of the following words: 
\begin{align*}
& p_i, p_iap_i^{-1} & (a\in A,\ 0\leq i\leq k), \\
& p_i\xi_i & (\xi_i\in\pref(u_i)\cup\pref(v_i),\ 1\leq i\leq k), \\
& z, zbz^{-1} & (b\in B), \\
& zp_i^{-1} & (0\leq i\leq k).
\end{align*}
It follows that this is a set of words representing the elements of the generating set of $\mathrm{RU}(M_{S,T})$.
Note that $z$ is redundant in this generating set, as $z=(zp_0^{-1})p_0$ holds in $M_{S,T}$. 
Also note that the elements represented by words of the form $p_i\xi_i$, 
$\xi_i\in\pref(u_i)\cup\pref(v_i)$, are redundant generators as, whenever $\xi_i=a_{1,i}\dots a_{l_i,i}$, we have
$$
p_i\xi_i = (p_ia_{1,i}p_i^{-1})\dots (p_ia_{l_i,i}p_i^{-1}) p_i
$$
holding in $M_{S,T}$ and in $\mathrm{RU}(M_{S,T})$. In conclusion, the previous discussion yields the following result.

\begin{lem}\label{lem:mst-gen}
The elements represented by words $p_i, zp_i^{-1}, p_iap_i^{-1}$ ($0\leq i\leq k$, $a\in A$) and $zbz^{-1}$ ($b\in B$) in $M_{S,T}$ form a generating set for $\mathrm{RU}(M_{S,T})$.
\end{lem}

Here is our main technical result of this section. It gives a presentation for $\mathrm{RU}(M_{S,T})$ with respect to the generating set identified in the previous lemma.

\begin{thm}\label{main-GK}
Let $S$ be a finitely RC-presented monoid and let $T$ be a finitely generated submonoid of $S$ and let $M_{S,T}$ be the finitely presented special inverse monoid defined above.  
Let $Q$ be the right cancellative monoid RC-presented by generators $p_i,q_i$ ($0\leq i\leq k$), $a^{(i)}$ ($a\in A$, $0\leq i\leq k$), $b^{(z)}$ ($b\in B$), 
and relations
\begin{align*}
& q_i w^{(i)} p_i = q_0 w^{(0)} p_0 & (w\in A^*,\ 1\leq i\leq k) \\
& q_i u^{(i)} = q_i v^{(i)} & (u,v\in A^*\text{ such that }u=v\text{ holds in }S,\ 0\leq i\leq k) \\
& q_i b^{(i)} = b^{(z)} q_i & (b\in B,\ 0\leq i\leq k)
\end{align*}
where the word $w^{(i)}$ is obtained from $w\in A^*$ by replacing each letter $a\in A$ by $a^{(i)}$, and, similarly, the word $u^{(z)}$ is obtained from $u\in B^*$
by replacing each letter $b\in B$ by $b^{(z)}$. Then $\mathrm{RU}(M_{S,T})\cong Q$ and 
the elements $b^{(z)}$, $b\in B$, generate a submonoid of $Q$ isomorphic to $T$.
In particular, $T$ embeds into $\mathrm{RU}(M_{S,T})$.
\end{thm}

First of all, we make a general observation that parallels a well-known fact from the theory of ordinary monoid presentations.

\begin{lem}\label{lem:rc}
Let $M=\MonRC\pre{A}{u_i=v_i\; (i\in I)}$ and $\xi:A^*\to N$ a monoid homomorphism, where $N$ is a right cancellative monoid, such that $u_i\xi = v_i\xi$ holds for all $i\in I$. Then there is a homomorphism $\xi':M\to N$ such that $\xi=\nu\xi'$, where $\nu$ if the canonical homomorphism $A^*\to M$.
\end{lem}

\begin{proof}
Let $\ol{M}=\Mon\pre{A}{u_i=v_i\; (i\in I)}$. Then $M$ is a homomorphic image of $\ol{M}$. In fact, the one-sided, right cancellative version of \cite[Exercise 2.6.12]{How} implies that all relations $\alpha=\beta$, $\alpha,\beta\in A^*$, with the property that $\alpha\gamma = \beta\gamma$ holds in $\ol{M}$ for some $\gamma\in A^*$ provide a \emph{monoid} presentation of $M$. In particular, this includes all relations of the form $u_i=v_i$, $i\in I$.

Now assume that $(\alpha,\beta)$ is a pair of words from $A^*$ with the property that $\alpha\gamma = \beta\gamma$ holds in $\ol{M}$ for some $\gamma\in A^*$. Then
$$
(\alpha\xi)(\gamma\xi) = (\alpha\gamma)\xi = (\beta\gamma)\xi =  (\beta\xi)(\gamma\xi)
$$
holds in $N$, and, since it is right cancellative, we conclude that $\alpha\xi = \beta\xi$. The lemma now follows immediately e.g.\ from \cite[Proposition 1.2.1]{RuskucPhD}.
\end{proof}

As a first immediate application, we have the following conclusion.

\begin{pro}\label{pro:psi}
The map $\psi_0$ from the set of generating symbols of $Q$ to the generators of $\mathrm{RU}(M_{S,T})$ from Lemma \ref{lem:mst-gen}, given by  
\begin{equation}\label{subst}
p_i\mapsto p_i,\quad q_i\mapsto zp_i^{-1}, \quad a^{(i)}\mapsto p_iap_i^{-1}, \quad b^{(z)}\mapsto zbz^{-1}
\end{equation}
($0\leq i\leq k$, $a\in A$, $b\in B$), extends to a surjective homomorphism $\psi:Q\to\mathrm{RU}(M_{S,T})$.
\end{pro}

\begin{proof}
It is straightforward to verify that if $\psi_0'$ is the homomorphism 
$$\Sigma_0^*\to \mathrm{RU}(M_{S,T})$$
extending $\psi_0$, where 
$$
\Sigma_0 = 
\{ p_i,q_i\, (0\leq i\leq k), \; a^{(i)}\, (a\in A, 0\leq i\leq k), \; b^{(z)}\, (b\in B) \},
$$
then each defining relation of $Q$ turns into a true equality in $M_{S,T}$ under $\psi_0'$. (We make repeated use of \cite[Corollary 3.2]{Gr-Inv} to see this.) The proposition now follows from the previous lemma.
\end{proof}

In the remainder of the section, the homomorphism $\psi:Q\to\mathrm{RU}(M_{S,T})$ extending $\psi_0$ will be fixed, as well as the alphabet $\Sigma_0$. To prove the first part of the theorem, we need to show that $\psi$ is in fact injective. For improved clarity of
the exposition, the proof of injectivity of $\psi$ will be done over the next five subsections; then the remaining two subsections of this section will deal, respectively, with the second part of the theorem showing that $T$ embeds into $\mathrm{RU}(M_{S,T})$, and the completion of the proof of Theorem \ref{thm:dense}.

\subsection{The graph $\Omega$}

In the proof of \cite[Theorem 4.1]{GK}, a $\Sigma$-edge-labelled graph $\Omega$ was defined in a recursive manner, which turns out to be a morphic image of $S\Gamma(1)$, the \Sch graph of right units of  the special inverse monoid defined in that theorem  (but which, in reality, coincides with the \Sch graph in question). Here we define a similar infinite graph, also denoted by $\Omega$, in a recursive fashion, as follows.

The graph $\Omega$ has a root vertex from which there is edge labelled $x$ to a copy of the graph $\Omega_x$ defined below, for each $x\in\{z,p_0,\dots,p_k\}$.
For each $0\leq i\leq k$, a \emph{$p_i$-zone} $\Omega_{p_i}$ consists of the Cayley graph of the free monoid $A^*$ (with respect to the generating set $A$) whose
root vertex is the empty word along with:
\begin{itemize}
\item at each vertex and for each $x\in\{z,p_0,\dots,p_k\}$ an edge labelled $x$ going out with a copy of $\Omega_x$ at the far end;
\item at each vertex distinct from the root an edge labelled $p_i$ coming in, with a copy of $\Omega_{p_i^{-1}}$ at the far end;
\item if $i=0$ there is a loop labelled $d$ at each vertex; and
\item if $1\leq i\leq k$ for each pair of vertices $W',W''$ 
(in general, in this section we denote graph vertices by capital letters to distinguish them from words) 
such that there is a vertex $W$ in this $p_i$-zone with the property that $W',W''$ are endpoints of the paths starting from $W$ labelled by $u_i$ and $v_i$, 
respectively, an edge from $W''$ to $W'$ labelled $d$.
\end{itemize}
Let $\Gamma_S$ denote the directed labelled right Cayley graph $\Gamma_S$ of $S$ with with respect to the generating set $A$.  
A \emph{$z$-zone} $\Omega_z$ consists of a copy of $\Gamma_S$ together with
\begin{itemize}
\item at each vertex and for each $x\in\{z,p_0,\dots,p_k\}$ an edge labelled $x$ going out with a copy of $\Omega_x$ at the far end;
\item at each vertex and for each $x\in\{p_0,\dots,p_k\}$ an edge labelled $x$ coming in with a copy of $\Omega_{x^{-1}}$ at the far end;
\item at the each vertex of the copy of $\Gamma_S$ corresponding to an element of $T$ an edge labelled $z$ coming in with a copy of $\Omega_{z^{-1}}$ 
at the far end; and
\item at each vertex a loop labelled $d$.
\end{itemize}
Finally, for $x\in\{z,p_0,\dots,p_k\}$ an \emph{$x^{-1}$-zone} $\Omega_{x^{-1}}$ consists of a single root vertex with 
\begin{itemize}
\item for each $y\in\{z,p_0,\dots,p_k\}$ such that $y\neq x$ an edge labelled $y$ going out with a copy of $\Omega_y$ at the far end.
\end{itemize}
See Figure \ref{fig_omega} for an illustration of a part of the graph $\Omega$.

%
\begin{figure}[tb]
\begin{center}
\begin{tikzpicture}[scale=.3,  
TRectangle/.style ={draw, rectangle, thick, minimum height=8em, minimum width=8em}]
\filldraw (-4,0) circle (7pt);
\draw[thick, rounded corners] (8.5,8.5) -- (-0.2,8.5) -- (-0.2,-0.2) -- (8.5,-0.2)--cycle;
\draw[thick, rounded corners] (-10,2) -- (-10-4, 2+8) -- (-10+4, 2+8) -- (-10,2);
\draw[thick, rounded corners, fill=gray!20] (-0.2,-0.2) -- (8.5,-0.2) -- (-0.2,8.5) --cycle;
\node at (3.5,7.1) {\scriptsize $\Gamma_{S} \setminus \Gamma_{T}$};
\node at (-8.7,13) {\scriptsize $p_0$};
\node at (-10.1,7) {\scriptsize $A^*$};
\node at (-5.4,13) {\scriptsize $p_k$};
\node at (-8,-0.5) {\scriptsize $p_0$};
\node at (-7.1,13.1) {\scriptsize $\ldots$};
\node at (-10.3,13) {\scriptsize $z$};
\node at (-3.4,8.6) {\scriptsize $p_i$};
\node at (-8.1,2.3) {\scriptsize $p_i$};
\node at (-6.2,2.6) {\scriptsize $\iddots$};
\node at (2,4.1) {\scriptsize $\Gamma_{T}$};
\node at (-9,1.2) {\scriptsize $\vdots$};
\node at (-6.2,4.3) {\scriptsize $p_k$};
\arrowdraw{0.5}{12,4}{6,6};
\arrowdraw{0.6}{2,2}{4,-4};
\arrowdraw{0.6}{-8,8}{-10,14};
\arrowdraw{0.6}{6,6}{6,12};
\arrowdraw{0.6}{6,6}{8,12};
\arrowdraw{0.5}{12,6}{6,6};
\arrowdraw{0.6}{12,6}{18,4};
\arrowdraw{0.6}{-8,8}{-8,14};
\arrowdraw{0.6}{2,2}{2,-4};
\arrowdraw{0.6}{2,2}{0,-4};
\arrowdraw{0.5}{12,8}{6,6};
\arrowdraw{0.5}{-4,2}{2,2};
\arrowdraw{0.6}{6,6}{4,12};
\arrowdraw{0.5}{-4,4}{2,2};
\arrowdraw{0.5}{-4,0}{2,2};
\arrowdraw{0.6}{-4,0}{-10,0};
\arrowdraw{0.6}{12,6}{18,6};
\arrowdraw{0.6}{-4,0}{-10,2};
\arrowdraw{0.6}{-4,0}{-6,6};
\arrowdraw{0.6}{-8,8}{-6,14};
\arrowdraw{0.5}{-2,8}{-8,8};
\arrowdraw{0.6}{12,6}{18,8};
\arrowdraw{0.6}{12,6}{14,0};
\node at (-3,3.2) {\scriptsize $\vdots$};
\node at (5.2,11) {\scriptsize $p_0$};
\node at (3.8,11) {\scriptsize $z$};
\node at (10.5,5.6) {\scriptsize $\vdots$};
\node at (6.9,11) {\scriptsize $\ldots$};
\node at (4.2+4.3,11) {\scriptsize $p_k$};
\node at (-0.2,-3) {\scriptsize $z$};
\node at (-2,1.4) {\scriptsize $p_0$};
\node at (1.4,-3) {\scriptsize $p_0$};
\node at (2.8,-3) {\scriptsize $\ldots$};
\node at (0.3+4.2,-3) {\scriptsize $p_k$};
\node at (-2,4) {\scriptsize $p_k$};
\node at (16.5,6.5) {\scriptsize $p_0$};
\node at (16.5,5.5) {\scriptsize $\vdots$};
\node at (-2,3.8-3.6) {\scriptsize $z$};
\node at (10.5,8.2) {\scriptsize $p_0$};
\node at (10.5,7.1) {\scriptsize $\vdots$};
\node at (15,3) {\scriptsize $\iddots$};
\node at (11.4,6.5) {\scriptsize $p_i$};
\node at (10.5,3.9) {\scriptsize $p_k$};
\node at (16.5,8) {\scriptsize $z$};
\node at (19.2,3.5) {\scriptsize $p_j \; \; \; \; (j \neq i)$};
\node at (12.9,1) {\scriptsize $p_k$};
\end{tikzpicture}
\end{center}
\caption{An illustration of a typical part of the graph $\Omega$ constructed in this subsection, including the root vertex marked in black. The square represents a $z$-zone (of the root vertex, in this instance) with the copy of the Cayley graph $\Gamma_S$ of $S$ partitioned so that the shaded triangle corresponds to the submonoid $T$. The triangle on the left represents a $p_i$-zone. The $d$-labelled edges are omitted from this sketch.} \label{fig_omega}
\end{figure}
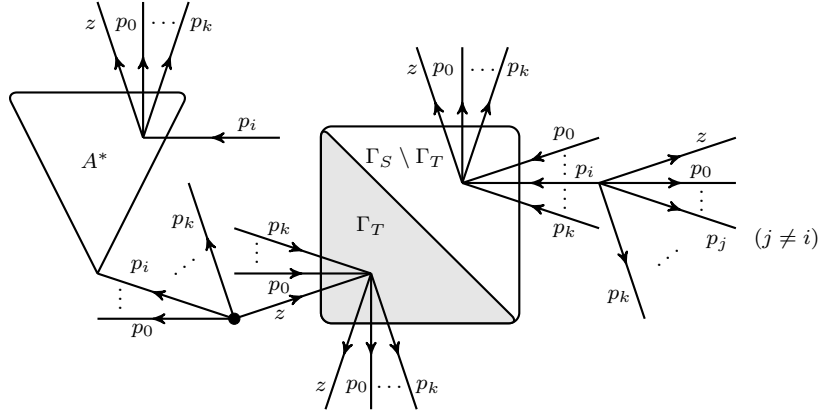
%

\begin{lem}\label{lem:prop-om}
The graph $\Omega$ constructed above has the following properties:
\begin{itemize}
\item[(i)] $\Omega$ is bi-deterministic, meaning that no vertex has two edges going out labelled by the same letter or two edges coming in labelled by the same letter. 
\item[(ii)] Every vertex of $\Omega$ apart from its root belongs either to a $p_i$-zone for some $0\leq i\leq k$, or to some $z$-zone, or to an $x^{-1}$-zone for some $x\in\{z,p_0,\dots,p_k\}$.
\item[(iii)] All zones in the definition of $\Omega$ are disjoint.
\item[(iv)] A vertex belongs to some $z$-zone if and only if it has enges coming in labelled by $p_i$ and $p_j$ for some indices $i\neq j$.
\end{itemize}
\end{lem}

\begin{proof}
The property (i) is a consequence of the right cancellativity of $S$, while properties (ii)--(iv) follow from a direct inspection of the definition of $\Omega$.
\end{proof}

Much in the same way as in the proof of \cite[Theorem 4.1]{GK}, 
we argue that every relator defining $M_{S,T}$ can be read around a closed path in $\Omega$ at every vertex. 
For all the relations not involving the letter $d$ this is routine to verify.  
For the second family of relations in the presentation of $M_{S,T}$, if following an edge labelled $p_i$ from 
the given vertex takes us into a $z$-zone, then we arrive within a copy of the Cayley graph of $S$, and so reading  $u_i$ and $v_i$ from the latter vertex takes us 
to the same vertex, which, in addition, has a loop labelled $d$. In other words, there is a closed path starting with the vertex at which we enter the $z$-zone around 
which one can read $u_id^{-1}v_i^{-1}$; reading $p_i^{-1}$ takes us back to the initial vertex. On the other hand, if by traversing a certain edge 
labelled $p_i$ we land in a $p_i$-zone, say at the vertex $W$, then the paths in $\Omega_{p_i}$ starting at $W$ and labelled by $u_i$ and $v_i$ take us to some vertices 
$W'$ and $W''$, respectively, and the definition of $\Omega_{p_i}$ postulates that there is an edge labelled $d$ from $W''$ to $W'$. So, similarly as above, there is 
a cycle starting with $W$ around which we read $u_id^{-1}v_i^{-1}$ and reading $p_i^{-1}$ takes us back to the vertex from which we have started our journey. 
The relator of the form $p_0 d p_0^{-1}$ can similarly be shown to be readable in $\Omega$ around a cycle from every vertex. 
Hence, as a direct consequence of \cite[Lemma 3.3]{GK} we have the following conclusion.

\begin{pro}\label{pro:mst-om}
For every vertex $U$ of $\Omega$ there is a labelled graph morphism from the \Sch graph $S\Gamma(1)$ of $M_{S,T}$ to $\Omega$ taking the root vertex of $S\Gamma(1)$ to $U$.
\end{pro}

\subsection{Tweaking the Cayley graph of $Q$}

Let $\Gamma$ be the right Cayley graph for the given presentation of $Q$ in the statement of Theorem~\ref{main-GK}. This is a $\Sigma_0$-labelled graph which is bi-deterministic as a consequence of the right cancellativity of $Q$. 

Our ultimate aim is to prove that $\mathrm{RU}(M_{S,T})$ is isomorphic to $Q$. Our approach to proving this will be to first compare the right Cayley graph $\Gamma$ of $Q$ with the Sch\"{u}tzenberger graph $S\Gamma(1)$. This will be done in this subsection where we give a process of modifying the graph $\Gamma$ by adding and removing edges to obtain a $\Sigma$-labelled graph denoted $\Gamma'$ that we shall later go on to prove is isomorphic to $S\Gamma(1)$.               
To describe this process we first begin by defining several types of vertices in the graph $\Gamma$.
\begin{itemize}
\item 
$U\in V(\Gamma)$ \emph{is of type-$z$} if there is a walk in $\Gamma$ from its root to $U$ labelled by a word of the form $uq_iw^{(i)}p_i$ for some $w\in A^*$ and some $u\in\Sigma_0^*$
(and any $0\leq i\leq k$, bearing in mind that $q_iw^{(i)}p_i=q_0w^{(0)}p_0$ holds in $Q$). Let $\mathcal{Z}$ be the set of all vertices of this type.
\item 
$U\in V(\Gamma)$ \emph{is of type-$p_i$ ($0\leq i\leq k$)} if there is a walk in $\Gamma$ from its root to $U$ labelled by a word of the form $v w^{(i)}p_i$ for some $w\in A^*$ and some
$v\in\Sigma_0^*$ which is either the empty word or ends with a letter $\alpha\not\in\{q_i\}\cup \{a^{(i)}:\ a\in A\}$. Let $\Pi_i$ be the set of all vertices of this type.
\end{itemize}

The main idea behind this typing of certain vertices of $\Gamma$ is to ``simulate'' the $z$-zones and the $p_i$-zones of the graph $\Omega$ within $\Gamma$.
In order to be able to define the graph $\Gamma'$, i.e.\ to coherently describe the process of turning the graph $\Gamma$ into $\Gamma'$, we need to verify several properties of the sets of vertices just introduced, most importantly that they are all pairwise disjoint. We begin by an easy observation.

\begin{lem}\label{lem-inc}
For any $0 \leq i \leq k$ if a vertex $W$ of $\Gamma$ has an incoming edge labelled by $p_i$ if and only if $W \in \mathcal{Z} \cup \Pi_i$.     
\end{lem}

\begin{proof}
The converse implication is immediate from the definitions of $\mathcal{Z}$ and $\Pi_i$, so assume that we have an edge $U\xrightarrow{p_i}W$. Choose an arbitrary walk from the root vertex of $\Gamma$ to $U$ and let $u_0\in\Sigma_0^*$ be the label of that walk. If $u_0$ is empty then clearly $U\in\Pi_i$. Otherwise, consider the rightmost letter $\alpha$ of $u_0$ not belonging to the set $\{a^{(i)}:\ a\in A\}$. If $\alpha=q_i$ then we can write $u_0=uq_iw^{(i)}$ for some words $u\in\Sigma_0^*$ and $w\in A^*$, whence $W\in\mathcal{Z}$. If this is not the case then $W\in\Pi_i$.
\end{proof}

In order to be able to coherently define the graph $\Gamma'$, it will be necessary to show that the classes $\mathcal{Z},\Pi_0,\Pi_1,\dots,\Pi_k$ of vertices of $\Gamma$ are pairwise disjoint. For this purpose, we prove two auxiliary results.

\begin{lem}\label{lem:calc}
Let $w\in A^*$, $u,v\in\Sigma_0$ and $0\leq i\leq k$ be arbitrary. Then:
\begin{enumerate}
\item $w^{(i)}\psi = p_i w p_i^{-1}$,
\item $(uq_iw^{(i)}p_i)\psi = (u\psi)zw$,
\item $(vw^{(i)}p_i)\psi = (v\psi)p_iw$
\end{enumerate}
holds in $\mathrm{RU}(M_{S,T})$.
\end{lem}

\begin{proof}
(1) If the word $w$ is empty then both the left- and the right-hand side evaluate to $1=p_ip_i^{-1}$. Otherwise, assume that $w=a_1\dots a_m$. Then $w^{(i)} = a_1^{(i)}\dots a_m^{(i)}$. Therefore,
$$
w^{(i)}\psi = (p_i a_1 p_i^{-1})\dots (p_i a_m p_i^{-1}),
$$
which is by \cite[Corollary 3.2]{Gr-Inv} equal in $\mathrm{RU}(M_{S,T})$ to $p_ia_1\dots a_mp_i^{-1} = p_iwp_i^{-1}$.

(2) Bearing in mind (1) and \cite[Corollary 3.2]{Gr-Inv}, we have
$$
(uq_iw^{(i)}p_i)\psi = (u\psi)(q_iw^{(i)}p_i)\psi = (u\psi)zp_i^{-1}p_iwp_i^{-1}p_i = (u\psi)zw,
$$
since $zp_i^{-1}$, $p_iwp_i^{-1}$ and $p_i$ are right units of $M_{S,T}$. 

(3) We have
$$
(vw^{(i)}p_i)\psi = (v\psi)(w^{(i)}p_i)\psi = (v\psi)p_iwp_i^{-1}p_i = (v\psi)p_iw,
$$
with the last step holding for a similar reason as in (2).
\end{proof}

\begin{lem}\label{lem-qi}
Let $\xi:S\Gamma(1)\to\Omega$ be a graph morphism that takes the root of the \Sch graph $S\Gamma(1)$ of $M_{S,T}$ to the root of $\Omega$, let $0\leq i\leq k$, and let $w_0\in\Sigma_0^*$ and $v\in A^*$ be words such that the vertex $(w_0v^{(i)}p_i)\psi\xi$ (where $(w_0v^{(i)}p_i)\psi\in\mathrm{RU}(M_{S,T})$ is identified with the corresponding vertex of $S\Gamma(1)$) belongs to a $z$-zone of $\Omega$. If the last letter of $w_0$ does not belong to $\{a^{(i)}:\ a\in A\}$ then $w_0$ is not empty and its last letter must be $q_i$.
\end{lem}

\begin{proof}
First of all, note that by the item (3) of the previous lemma we have that
$$
(w_0v^{(i)}p_i)\psi = (w_0\psi)p_iv
$$
holds in $\mathrm{RU}(M_{S,T})$. It follows that in $\Omega$ the walk from the vertex $U=(w_0\psi)\xi$ labelled by the word $p_iv$ leads to a vertex $V$ belonging to some $z$-zone $\Omega_z$.

Seeking a contradiction, assume first that $w_0$ is empty. Then there is a walk in $\Omega$ from its root to $V$ labelled by $p_iv$. But then it readily follows that the vertex $V$ belongs to the $p_i$-zone of the root vertex, a contradiction with Lemma \ref{lem:prop-om}(iii). Hence, $w_0$ cannot be the empty word, and we proceed by analysing the possible cases regarding its last letter.

\emph{Case 1: The last letter of $w_0$ is $a^{(t)}$ for some $a\in A$ and $t\neq i$.} This means that $(w_0\psi)p_iv = x\cdot p_tap_t^{-1}\cdot p_iv$ holds in $M_{S,T}$ for some $x\in\mathrm{RU}(M_{S,T})$, and so there is a walk 
$$
U_1\xrightarrow{p_t} U_2\xrightarrow{a} U_3\xleftarrow{p_t} U
$$
in $\Omega$, where $U_1=(x\psi)\xi$. It follows by the construction of $\Omega$ that if $U$ belongs to a $p_j$-zone for some $0\leq j\leq k$ or to a $z$-zone or to a $p_j^{-1}$-zone for some $j\neq t$ or is the root of $\Omega$, then $U_3$ is the root of the $p_t$-zone at the far end of the edge $U\xrightarrow{p_t} U_3$. Otherwise, $U$ is the single vertex in the $p_t^{-1}$-zone of $U_3$. The former case is impossible, since then $U_3$ is the root vertex of a copy of the Cayley graph of $A^*$ and so it cannot have an incoming edge labelled by $a$. In the latter case, the edge labelled by $p_i$ coming out of $U$ takes us to a $p_i$-zone because $i\neq t$, and the walk labelled by the word $v$ keeps us in the same $p_i$-zone. Thus $V$ belongs to that same $p_i$-zone, a contradiction with Lemma \ref{lem:prop-om}(iii).

\emph{Case 2: The last letter of $w_0$ is $b^{(z)}$ for some $b\in B$.} Then in $M_{S,T}$ we have $(w_0'\psi)p_iv = x\cdot zbz^{-1}\cdot p_iv$ for some $x\in\mathrm{RU}(M_{S,T})$, implying the existence of a walk
$$
U_1\xrightarrow{z} U_2\xrightarrow{b} U_3\xleftarrow{z} U
$$
in $\Omega$, where $U_1=(x\psi)\xi$. Hence, $U_2,U_3$ are in the same $z$-zone (and in fact in the ``shaded part'' a copy of the subgraph $\Gamma_T$, with incoming $z$-edges from $U_1$ and $U$). Again by the construction of $\Omega$, this setup implies that there is a $p_i$-zone at the far end of the $p_i$-edge coming out of $U$ and the walk from the root of that zone labelled by $v$ keeps us within that zone, forcing $V$ to be in the same zone, a contradiction with Lemma \ref{lem:prop-om}(iii).

\emph{Case 3: The last letter of $w_0$ is $p_t$ for some $0\leq t\leq k$.} In this case  $(w_0'\psi)p_iv = x\cdot p_t\cdot p_iv$ holds in $M_{S,T}$ for some $x\in\mathrm{RU}(M_{S,T})$. Hence, $U$ has an incoming edge labelled by $p_t$, and now in a similar fashion as in the previous cases a simple inspection of the definition of $\Omega$ puts the vertex $V$ in some $p_i$-zone.

\emph{Case 4: The last letter of $w_0$ is $q_t$ for some $t\neq i$.} Now we have $(w_0'\psi)p_iv = x\cdot zp_t^{-1}\cdot p_iv$ in $M_{S,T}$ for some $x\in\mathrm{RU}(M_{S,T})$ resulting in the existence of the walk
$$
U_1\xrightarrow{z} U_2\xleftarrow{p_t} U
$$
in $\Omega$, where $U_1=(x\psi)\xi$. So, $U_2$ belongs to a $z$-zone (and again to its ``shaded part'') and $U$ then must be the single vertex in the $p_t^{-1}$-zone of $U_2$. Once again, because $i\neq t$, we must conclude that the walk $U\leadsto V$ labelled by $p_iv$ ends in a $p_i$-zone, a contradiction. 

Thus we are left with no other possibility but the last letter of $w_0$ being $q_i$, and the lemma is proved. 
\end{proof}

We are now in a position to state and prove the first part of the announced result about the  disjointness of the considered classes.

\begin{lem}\label{lem-disj1}
For any $0 \leq i \leq k$, $\mathcal{Z}$ and $\Pi_i$ are disjoint.
\end{lem}

\begin{proof}
Seeking a contradiction, assume that there is a vertex of $\Gamma$ that is is at the same time type-$z$ and type-$p_i$ for some $0\leq i\leq k$. Then we have $uq_iw_1^{(i)}p_i = vw_2^{(i)}p_i$ in $Q$ for some words $w_1,w_2\in A^*$ and $u,v\in\Sigma_0^*$ such that $v$ is either empty or ends with a letter other than $q_i$ or $a^{(i)}$ ($a\in A$). 
However, this implies that $(uq_iw_1^{(i)}p_i)\psi = (vw_2^{(i)}p_i)\psi$ holds in $M_{S,T}$, which, by Lemma \ref{lem:calc}, takes the following form:
$$
(u\psi)zw_1 = (v\psi)p_iw_2.
$$
Upon applying the graph morphism $\xi: S\Gamma(1)\to\Omega$ from the previous lemma, we obtain two vertices of $\Omega$, namely $U_1=(u\psi)\xi$ and $U_2=(v\psi)\xi$, such that the walks in $\Omega$ starting from $U_1,U_2$ labelled by $zw_1$ and $p_iw_2$, respectively, end in the same vertex $V$. Since the walk $U_1\leadsto V$ starts with a $z$-edge and proceeds by edges labelled by letters from $A$, the vertex $V=(vw_2^{(i)}p_i)\psi\xi$ must belong to some $z$-zone of $\Omega$. But then the previous lemma immediately implies that the word $v$ must end with the letter $q_i$. This is a contradiction with the assumed form of $v$.
\end{proof}

Our next goal is to prove that $\Pi_i\cap\Pi_j=\es$ whenever $i\neq j$. However, first we are going to prove a slightly stronger result about the Cayley graph $\Gamma$ of the right cancellative monoid $Q$, because this result will be of key importance at a later stage of our arguments, too (specifically, in Subsection \ref{subsec:prop-ii}). 

\begin{lem}\label{lem-det}
Let $0\leq i\neq j\leq k$, $v\in A^*$ and $U_1,U_2$ two vertices of $\Gamma$ such that the walks from $U_1,U_2$ labelled by $v^{(i)}p_i$ and $v^{(j)}p_j$, respectively, end in 
the same vertex $V$ of $\Gamma$. Then there exists a vertex $W$ of $\Gamma$ and a word $u\in A^*$ such that the walk from $W$ labelled by $q_iu^{(i)}$ ends in $U_1$ and the walk 
from $W$  labelled by $q_ju^{(j)}$ ends in $U_2$. In particular, each vertex of $\Gamma$ that has incoming edges labelled by $p_i$ and $p_j$ for some $i\neq j$ belongs to $\mathcal{Z}$.
\end{lem}

\begin{proof}
By the conditions of the lemma, there are words $w',w''\in \Sigma_0^*$ such that $w'v^{(i)}p_i = w''v^{(j)}p_j$ holds in $Q$, as well as walks in $\Gamma$ from its root vertex to $U_1,U_2$ labelled by $w',w''$, respectively. 

By Lemma \ref{lem:calc}(3), we have $(w'v^{(i)}p_i)\psi = (w'\psi)p_iv$ and $(w''v^{(j)}p_j)\psi = (w''\psi)p_jv$, and thus
$$
(w'\psi)p_iv = (w''\psi)p_jv
$$
holds in $M_{S,T}$. Therefore, it follows that in the \Sch graph $S\Gamma(1)$ of $M_{S,T}$ the walks from vertices $w'\psi$ and $w''\psi$ labelled by $p_iv$ and $p_jv$, respectively, lead to the same vertex. But then, by Proposition \ref{pro:mst-om}, we have that if $U'=(w'\psi)\xi$ and  $U''=(w''\psi)\xi$ (where $\xi:S\Gamma(1)\to\Omega$ is the graph morphism from Lemma \ref{lem-qi}) the walks in $\Omega$ from $U',U''$ labelled by $p_iv$ and $p_jv$, respectively, lead to the same vertex $V_1$. However, $\Omega$ is bi-deterministic, so we conclude that there is a vertex $V_0'$ of $\Omega$ with edges 
$$
U'\xrightarrow{p_i} V_0' \xleftarrow{p_j} U''
$$
and a walk $V_0'\leadsto V_1$ labelled by the word $v$. In particular, by Lemma \ref{lem:prop-om}(iv), the vertex $V_0'$ belongs to some $z$-zone $\Omega_z$ and all vertices along the considered walk $V_0'\leadsto V_1$ belong to the same $z$-zone.

Let $u^{(i)}$ now be the longest suffix of $w'$ consisting entirely of letters from $\{a^{(i)}:\ a\in A\}$, so that we can write
$$
w' = w_0' u^{(i)}.
$$
Here, $w_0'\in\Sigma_0^*$ is either empty or its last letter is not of the form $a^{(i)}$ ($a\in A$). Similarly as above, we have
$$
(w'v^{(i)}p_i)\psi = (w_0'u^{(i)}v^{(i)}p_i)\psi = (w_0'\psi)p_iuv.
$$
This implies that in $\Omega$ we have a walk $V_0=((w_0'\psi) p_i)\xi\leadsto V_0'$ labelled by the word $u\in A^*$; however, this is possible only if this latter walk, too, belongs to the same $z$-zone as $V_0'\leadsto V_1$ does. The combined effect is a walk $V_0\leadsto V_1$ labelled by $uv$ that belongs to a single $z$-zone. In particular, the vertex
$$
V_1 = (w_0'u^{(i)}v^{(i)}p_i)\psi\xi,
$$
belongs to a $z$-zone of $\Omega$. Then, by the previous lemma, the word $w_0'$ is not empty and its last letter is $q_i$.

Therefore, we arrive at the conclusion that 
$$
w'v^{(i)}p_i=w_0 q_i u^{(i)} v^{(i)} p_i = w_0 q_j u^{(j)} v^{(j)} p_j
$$
holds in $Q$ for some $w_0\in\Sigma_0^*$. Since $w'v^{(i)}p_i = w''v^{(j)}p_j$ holds in $Q$ and $Q$ is right cancellative, we must have that $w''=w_0 q_j u^{(j)}$ holds in $Q$. Upon taking $W$ of $\Gamma$ to be the terminal vertex of the walk in $\Gamma$ starting at its root vertex labelled by $w_0$, the lemma is proved. The last statement is just a special case when $v$ is the empty word.
\end{proof}

\begin{lem}\label{lem-disj2}
For any $0 \leq i\neq j \leq k$, $\Pi_i$ and $\Pi_j$ are disjoint.   
\end{lem}

\begin{proof}
Assume that $V\in\Pi_i\cap\Pi_j$. Then $V$ has incoming edges both labelled by $p_i$ and $p_j$, so by the previous result $V\in\mathcal{Z}$. However, this contradicts Lemma \ref{lem-disj1}.
\end{proof}

\begin{cor}
\begin{enumerate} 
\item 
A vertex $W$ of $\Gamma$ belongs to $\mathcal{Z}$ if and only if it has an incoming edge labelled by $p_i$ for all $0\leq i\leq k$.
\item 
A vertex $W$ of $\Gamma$ belongs to $\Pi_i$ for some $0\leq i\leq k$ if and only if it has an incoming edge labelled by $p_i$ but no incoming edges labelled by $p_j$ for any $0\leq j\leq k$, $j\neq i$.
\end{enumerate} 
\end{cor}

So now, here is how we transform the graph $\Gamma$. 
\begin{itemize}
\item For any three $U_1,U_2,U_3\in Q$ such that there is an edge $U_1\to U_2$ labelled $q_0$ and an edge $U_2\to U_3$ labelled $p_0$ (by applying the first relation in the presentation of $Q$ in the case $w$ is the empty word, for any $1\leq i\leq k$ there are walks from $U_1$ to $U_3$ labelled by $q_ip_i$), add an edge $U_1\to U_3$ labelled $z$.
After adding these $z$-edges note that every vertex will have a unique $z$-edge coming out of it because from every vertex in $\Gamma$ there is a unique path labelled $q_0p_0$.
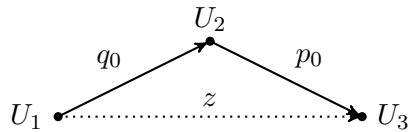
\begin{center}
\begin{tikzpicture} 
\filldraw[color=black, fill=black] (-2,0) circle (1.5pt) node[anchor=east]{$U_1$\,};
\filldraw[color=black, fill=black] (0,1) circle (1.5pt) node[anchor=south]{$U_2$};
\filldraw[color=black, fill=black] (2,0) circle (1.5pt) node[anchor=west]{\,$U_3$};
\draw[thick,->] (-2,0) -- (-0.03,0.985) node[midway,above left]{$q_0$};
\draw[thick,->] (0,1) -- (1.97,0.015) node[midway,above right]{$p_0$};
\draw[thick,dotted,->] (-2,0) -- (1.985,0) node[midway,above]{$z$};
\end{tikzpicture}
\end{center}
\item For any $U_1,U_2\in Q$ such that there is an edge $U_1\to U_2$ labelled $a^{(i)}$ (for some $a\in A$, $0\leq i\leq k$), add an edge $V_1\to V_2$ labelled $a$, where 
$V_1,V_2$ are the endpoints of the unique edges labelled $p_i$ originating from $U_1,U_2$, respectively. 
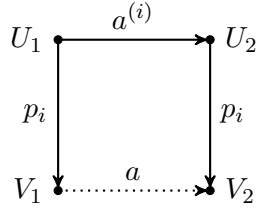
\begin{center}
\begin{tikzpicture} 
\filldraw[color=black, fill=black] (-1,0) circle (1.5pt) node[anchor=east]{$U_1$\,};
\filldraw[color=black, fill=black] (1,0) circle (1.5pt) node[anchor=west]{\,$U_2$};
\draw[thick,->] (-1,0) -- (0.97,0) node[midway,above]{$a^{(i)}$};
\filldraw[color=black, fill=black] (-1,-2) circle (1.5pt) node[anchor=east]{$V_1$\,};
\draw[thick,->] (-1,0) -- (-1,-1.97) node[midway,left]{$p_i$};
\filldraw[color=black, fill=black] (1,-2) circle (1.5pt) node[anchor=west]{\,$V_2$};
\draw[thick,->] (1,0) -- (1,-1.97) node[midway,right]{$p_i$};
\draw[thick,dotted,->] (-1,-2) -- (0.97,-2) node[midway,above]{$a$};
\end{tikzpicture}
\end{center}
\item For any $U_1,U_2\in Q$ such that there is an edge $U_1\to U_2$ labelled $b^{(z)}$ (for some $b\in B$), add an edge $V_1\to V_2$ labelled $b$, where $V_1,V_2$ are 
the endpoints of 
the unique edges labelled $z$ originating from $U_1,U_2$, respectively.
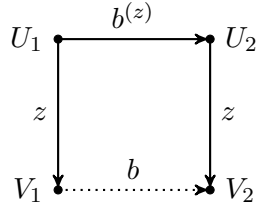
\begin{center}
\begin{tikzpicture} 
\filldraw[color=black, fill=black] (-1,0) circle (1.5pt) node[anchor=east]{$U_1$\,};
\filldraw[color=black, fill=black] (1,0) circle (1.5pt) node[anchor=west]{\,$U_2$};
\draw[thick,->] (-1,0) -- (0.97,0) node[midway,above]{$b^{(z)}$};
\filldraw[color=black, fill=black] (-1,-2) circle (1.5pt) node[anchor=east]{$V_1$\,};
\draw[thick,->] (-1,0) -- (-1,-1.97) node[midway,left]{$z$};
\filldraw[color=black, fill=black] (1,-2) circle (1.5pt) node[anchor=west]{\,$V_2$};
\draw[thick,->] (1,0) -- (1,-1.97) node[midway,right]{$z$};
\draw[thick,dotted,->] (-1,-2) -- (0.97,-2) node[midway,above]{$b$};
\end{tikzpicture}
\end{center}
\item Once the previous steps are completed, remove all edges labelled by any of the letters $q_i$ ($0\leq i\leq k$), $a^{(i)}$ ($a\in A$, $0\leq i\leq k$), $b^{(z)}$ ($b\in B$).
This results in a graph whose edges are labelled by $\Sigma\setminus\{d\}$.
\item Finish off the modification by adding certain $d$-labelled edges. First of all, add a $d$-labelled loop at any vertex $U\in \mathcal{Z}\cup \Pi_0$.
Furthermore, for any $U,V,W\in\Pi_i$ ($1\leq i\leq k$) such that there is a walk from $W$ to $U$ labelled by $u_i$ and a walk from $W$ to $V$ labelled by $v_i$, where $u_i=v_i$ is the defining relation of $S$ labelled by $i$, add a $d$-edge 
$V\to U$.
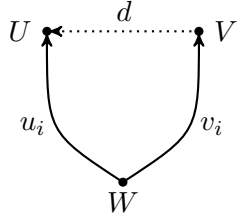
\begin{center}
\begin{tikzpicture} 
\filldraw[color=black, fill=black] (0,0) circle (1.5pt) node[anchor=north]{$W$};
\filldraw[color=black, fill=black] (-1,2) circle (1.5pt) node[anchor=east]{$U$\,};
\filldraw[color=black, fill=black] (1,2) circle (1.5pt) node[anchor=west]{\,$V$};
\draw[thick,->] (0,0) .. controls (-1,0.66) .. (-1,1.97) node[midway,left]{$u_i$};
\draw[thick,->] (0,0) .. controls (1,0.66) .. (1,1.97) node[midway,right]{$v_i$};
\draw[thick,dotted,->] (1,2) -- (-0.97,2) node[midway,above]{$d$};
\end{tikzpicture}
\end{center}
\end{itemize}
The result of this process is the graph $\Gamma'$.

\subsection{The central idea}\label{idea}

Here is the gist of the plan of our proof. Assume that we were able to show the following two properties of the graph $\Gamma'$:
\begin{itemize}
\item[(i)] for any word $r\in\overline{\Sigma}^*$ that occurs as a relator word in the presentation of $M_{S,T}$ and any $U\in Q=V(\Gamma')$ there is a closed walk labelled by $r$ 
originating at $U$;
\item[(ii)] $\Gamma'$ is bi-deterministic.
\end{itemize}
We claim that these two properties imply that the homomorphism $\psi:Q\to\mathrm{RU}(M_{S,T})$ from Proposition \ref{pro:psi} must be injective. In the rest of this subsection we explain why this is indeed the case.

First of all, \cite[Lemma 3.3]{GK} implies that for any $W\in Q$ there is a $\Sigma$-labelled graph morphism $S\Gamma(1)\to\Gamma'$ (where $S\Gamma(1)$ is taken with respect 
to $M_{S,T}$) that takes the root vertex of $S\Gamma(1)$ to $W$. In particular, let $\phi$ be one such morphism that takes the root of $S\Gamma(1)$ to the vertex of $\Gamma'$ 
representing the identity element of $Q$.

So, assume now that $U\psi=V\psi=X\in\mathrm{RU}(M_{S,T})$ for some $U,V\in Q$. Then there are words $u,v\in\Sigma_0^*$ such that in the Cayley graph $\Gamma$ there are walks 
from the vertex representing $1$ to $U$ and $V$ labelled by $u$ and $v$, respectively. But then in $S\Gamma(1)$ there are two (identical or not) walks from the root to $X$ labelled 
by the words $u\ol\psi,v\ol\psi\in\Sigma^*$, respectively, where $\ol\psi:\Sigma_0^*\to\Sigma^*$ is a morphism of free monoids induced by the same substitution \eqref{subst} that 
defines $\psi$. Consequently, in $\Gamma'$, there are walks from $1$ to $Y=X\phi$ labelled by $u\ol\psi$ and $v\ol\psi$. Bearing in mind the form of the mapping $\ol\psi$, 
the construction of the graph $\Gamma'$ and its bi-determinism, it follows that in $\Gamma'$ there are walks from $1$ to $Y$ labelled by $u$ and $v$. Since $\Gamma'$ is assumed to be bi-deterministic, it follows that $U=Y=V$, showing that $\psi$ is injective.

Therefore, to prove Theorem \ref{main-GK} it suffices to verify that the graph $\Gamma'$ indeed has the properties (i) and (ii).

\subsection{The proof of property (i)} 

Throughout, let $U$ be an arbitrary vertex of $\Gamma'$. We analyse five types of relators of $M_{S,T}$.

\begin{itemize}
\item[(1)] $r=p_i a p_i^{-1} p_i a^{-1} p_i^{-1}$ ($0\leq i\leq k$) \\
Clearly, there is a unique edge of the form $U\to V$ labelled by $a^{(i)}$ in $\Gamma$. If $U',V'$ are, respectively,
the endpoints of the edges 
in $\Gamma$ 
originating from $U,V$ labelled by $p_i$, then by construction $\Gamma'$ contains an edge $U'\to V'$ labelled by $a$. But then
$$
U\xrightarrow{p_i} U'\xrightarrow{a} V'\xleftarrow{p_i} V \xrightarrow{p_i} V'\xleftarrow{a} U'\xleftarrow{p_i} U
$$
is a closed walk in $\Gamma'$ based at $U$ and labelled by $r$.

\item[(2)] $r=p_i u_i d^{-1} v_i^{-1} p_i^{-1}$ ($1\leq i\leq k$) \\ 
First of all, in $\Gamma$ there are unique walks $U\to U_1\to\dots\to U_q$ and $U\to V_1\to\dots\to V_s$ labelled by $u_i^{(i)}$ and $v_i^{(i)}$, respectively. 
Let $W$, $U_1',\dots,U_q'$, $V_1',\dots,V_s'$ be the endpoints of edges in $\Gamma$ labelled by $p_i$ and originating, respectively, from $U$, $U_1,\dots,U_q$, $V_1,\dots,V_s$.
Then $W$ has an incoming edge labelled by $p_i$, which implies by Lemma~\ref{lem-inc} that $W\in\mathcal{Z}\cup\Pi_i$. Furthermore, by the construction of $\Gamma'$, there are walks $W\to U_1'\to\dots\to U_q'$ 
and $W\to V_1'\to\dots\to V_s'$ (in $\Gamma'$) labelled by $u_i$ and $v_i$, respectively. 
\begin{itemize}
    \item[(a)] Assume first that $W\in\mathcal{Z}$. Then, by definition, there is a walk in $\Gamma$ whose final vertex is $W$ and which is labelled by a word of the form $q_iw^{(i)}p_i$ for some word $w\in A^*$. Since $\Gamma$ is the Cayley graph of a right cancellative monoid and there is an edge in $\Gamma$ from $U$ to $W$ labelled by $p_i$, it follows that in $\Gamma$ there is a walk whose final vertex is $U$ labelled by $q_iw^{(i)}$. But then it immediately follows that for each $1\leq j\leq q$ there is a walk in $\Gamma$ with final vertex $U_j$ labelled by $q_i(wy_j)^{(i)}$ where $y_j$ is the prefix of $u_i$ of length $j$. We conclude that there is a walk in $\Gamma$ labelled by $q_i(wy_j)^{(i)}p_i$ with final vertex $U_j'$, showing that $U_j'\in\mathcal{Z}$. Similarly, we show that all vertices $V_1',\dots,V_s'$ are also of type-$z$. 
    Since $wu_i=wv_i$ holds in $S$ we have that $q_i(wu_i)^{(i)}=q_i(wv_i)^{(i)}$ holds in $Q$, and thus we conclude that $U_q=V_s$ and $U_q'=V_s'$, whilst this latter vertex has a $d$-loop. 
    \item[(b)] Now let $W\in\Pi_i$. Then there is a walk in $\Gamma$ from a vertex $W_0$ to $W$ labelled by a word of the form $w^{(i)}p_i$ for some word $w\in A^*$ such that $W_0$ can be reached from the root of $\Gamma$ by a word $v$ which is either empty or does not end with either $q_i$ or any of $a^{(i)}$ ($a\in A$). Then, by an analogous argument as in (a), there is a walk from $W_0$ to $U$ labelled by $w^{(i)}$ and thus there is a walk from $W_0$ to $U_j$, $1\leq j\leq q$, labelled by $(wy_j)^{(i)}$ (where $y_j$ has the same meaning as in (a)). This shows that $U_1',\dots,U_q'\in\Pi_i$ and, similarly, $V_1',\dots,V_s'\in\Pi_i$. 
    By the construction of $\Gamma'$, there is an edge $V_s'\to U_q'$ labelled by $d$. 
\end{itemize}
In both cases, we have completed a closed walk in $\Gamma'$ based at $U$ and labelled by the relator word $r$. 

\item[(3)] $r=p_0dp_0^{-1}$ \\
Let $U\to V$ be the unique edge from $U$ labelled by $p_0$. Then by Lemma~\ref{lem-inc} we have $V\in\mathcal{Z}\cup\Pi_0$ and so by construction of $\Gamma'$ there is a loop at $V$ 
labelled by $d$, so 
$$U\xrightarrow{p_0} V\xrightarrow{d} V\xleftarrow{p_0} U$$ 
is the closed walk we are looking for here. 

\item[(4)] $r=zbz^{-1}zb^{-1}z^{-1}$ ($b\in B$) \\
Consider the unique edge of the form $U\to V$ labelled by $b^{(z)}$ in $\Gamma$, and let $U\to U'\to U''$ and $V\to V'\to V''$ be the walks
in $\Gamma$ labelled by $q_0p_0$. Then $\Gamma'$ will contain edges $U\to U''$ and $V\to V''$ labelled by $z$, and, furthermore, an edge $U''\to V''$ labelled by $b$. Hence, 
$$U\xrightarrow{z} U''\xrightarrow{b} V''\xleftarrow{z} V\xrightarrow{z} V''\xleftarrow{b} U''\xleftarrow{z} U$$ 
constitutes the required closed walk in $\Gamma'$ labelled by $r$. 

\item[(5)] $r=z \left(\prod_{i=0}^k p_i^{-1}p_i\right) z^{-1}$ \\
Let $U\to U'\to U''$ be the walk in $\Gamma$ labelled by $q_0p_0$. Then there is an edge $U\to U''$ in $\Gamma'$ labelled by
$z$. However, since $q_0p_0=q_ip_i$ holds in $Q$ for all $1\leq i\leq k$, the vertex $U''$ has an incoming edge labelled by $p_i$ for all $0\leq i\leq k$. So, going from $U$ to $U''$ via
the considered edge labelled by $z$, traversing the said incoming edges back and forth, and then going back to $U$ constitutes a closed walk in $\Gamma'$ labelled by $r$.
\end{itemize}

\subsection{The proof of property (ii)} \label{subsec:prop-ii} 

Let us first notice that $\Gamma'$ is bi-deterministic with respect to the letters $p_i$, $0\leq i\leq k$: no vertex has more than one outgoing edge or more than one incoming edge labelled by each of these letters, since $\Gamma$ has this property and we have not added or deleted any edges with such labels.

The argument that $\Gamma'$ is bi-deterministic with respect to $z$ is straightforward. Indeed, assume there exist two $z$-edges $U\to V_1$ and $U\to V_2$. Then there are two 
vertices $V_1'$ and $V_2'$ such that in $\Gamma$ there are $q_0$-edges $U\to V_1'$ and $U\to V_2'$ and $p_0$-edges $V_1'\to V_1$ and $V_2'\to V_2$. The bi-determinism of $\Gamma$ first 
implies $V_1'=V_2'$ and then, consequently, $V_1=V_2$. Similarly, we cannot have two $z$-edges $U_1\to V$ and $U_2\to V$.

\begin{figure}[b]
\begin{center}
\begin{tikzpicture}[scale=1.5]
\filldraw[color=black, fill=black] (-2,2) circle (1.5pt) node[anchor=east]{$U$\,};
\filldraw[color=black, fill=black] (-2,1) circle (1.5pt) node[anchor=east]{$U'$\,};
\filldraw[color=black, fill=black] (0,2) circle (1.5pt) node[anchor=west]{\,$V_1$};
\filldraw[color=black, fill=black] (2,2) circle (1.5pt) node[anchor=west]{\,$V_2$};
\filldraw[color=black, fill=black] (0.5,1) circle (1.5pt);
\filldraw[color=black, fill=black] (2.5,1) circle (1.5pt);
\filldraw[color=black, fill=black] (-1,0) circle (1.5pt) node[anchor=east]{$W_1$};
\filldraw[color=black, fill=black] (1,0) circle (1.5pt) node[anchor=west]{$W_2$};
\filldraw[color=black, fill=black] (-1,-1) circle (1.5pt) node[anchor=north]{$W_1'$};
\filldraw[color=black, fill=black] (1,-1) circle (1.5pt) node[anchor=north]{$W_2'$};
\draw[thick,->] (-1,-1) .. controls (-2,-0.33) .. (-2,0.97) node[midway,left]{$u_i^{(i)}$};
\draw[thick,->] (-1,0) .. controls (-1,1.33) .. (-1.97,2) node[midway,left]{$u_i$};
\draw[thick,->] (1,0) .. controls (0,0) and (-1,2) .. (-1.97,2) node[midway,left]{$u_i$};
\draw[thick,->] (-1,-1) -- (-1,-0.03) node[midway,right]{{\small $p_i$}};
\draw[thick,->] (1,-1) -- (1,-0.03) node[midway,right]{{\small $p_i$}};
\draw[thick,->] (-2,1) -- (-2,1.97) node[midway,left]{{\small $p_i$}};
\draw[thick,->] (0.5,1) -- (0,1.97) node[midway,right]{{\small $p_i$}};
\draw[thick,->] (2.5,1) -- (2,1.97) node[midway,right]{{\small $p_i$}};
\draw[thick,->] (-1,0) .. controls (-0.3,1) .. (0,1.97) node[midway,right]{$v_i$};
\draw[thick,->] (1,0) .. controls (1.7,1) .. (2,1.97) node[midway,right]{$v_i$};
\draw[thick,->] (-1,-1) .. controls (0.3,0) .. (0.5,0.97) node[midway,left]{$v_i^{(i)}$};
\draw[thick,->] (1,-1) .. controls (2.3,0) .. (2.5,0.97) node[midway,right]{$v_i^{(i)}$};
\draw[thick,->] (1,-1) .. controls (-1,-1) .. (-2,0.97) node[midway,right]{\ \ \ \ $u_i^{(i)}$};
\draw[thick,dotted,->] (0,2) -- (-1.97,2) node[midway,below]{$d$};
\draw[thick,dotted,->] (2,2) .. controls (0,2.3) .. (-1.97,2) node[midway,above]{$d$};
\end{tikzpicture}
\end{center}
\caption{An illustration of the proof of bi-determinism of $\Gamma'$ with respect to $d$-edges}
\label{fig3}
\end{figure}
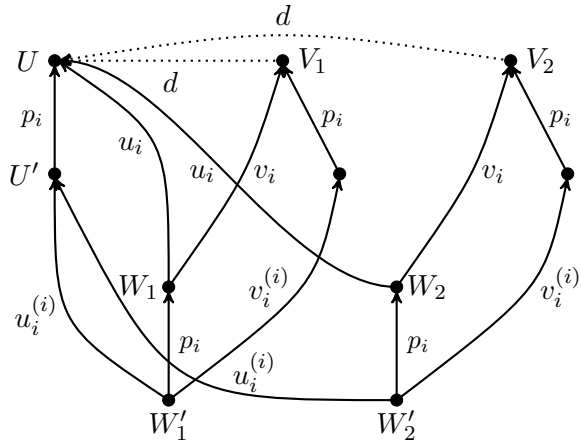

Regarding the letter $d$ note that because of the construction of $\Gamma'$ and Lemmas \ref{lem-disj1} and \ref{lem-disj2}, there are two possibilities for $\Gamma'$ not to be bi-deterministic with respect to 
this letter. One possibility is that there are distinct $d$-edges $V_1\to U$ and $V_2\to U$ for some vertices $U,V_1,V_2\in V(\Gamma)=V(\Gamma')$. 
Then we conclude that $U,V_1,V_2\in\Pi_i$ for some $1\leq i\leq k$  and that there exist $W_1,W_2\in\Pi_i$ such that there are walks in $\Gamma'$ from $W_1,W_2$ to $U$ both labelled by $u_i$.
From the bi-determinism of the $p_i$-labelled edges in $\Gamma'$ and the definition of $\Gamma'$ it follows that there are walks in $\Gamma$ labelled by $p_i^{-1}v_i^{(i)}p_i$ both from $W_1$ to $V_1$ and from $W_2$ to $V_2$. 
Similarly, there are walks in $\Gamma$ labelled by $p_i^{-1} u_i^{(i)} p_i$ from both $W_1$ and $W_2$ to $U$.
Because $\Gamma$ is bi-deterministic, them assumption that $V_1,V_2$ are distinct implies that $W_1,W_2$ must also be distinct.
Let $W_1'$, $W_2'$ and $U'$ be the unique vertices in $V(\Gamma)=V(\Gamma')$ such that $W_1' \xrightarrow{p_i} W_1$, $W_2' \xrightarrow{p_i} W_2$ and $U' \xrightarrow{p_i} U$. 
Because $W_1,W_2$ distinct, so are $W_1',W_2'$. 
Then there are walks in $\Gamma$ from $W_1',W_2'$ to $U'$ both labelled by $u_i^{(i)}$. However, the latter contradicts the bi-determinism of $\Gamma$. The other possibility, that there are $d$-edges $V\to U_1$ and $V\to U_2$ for some vertices $U_1,U_2,V\in V(\Gamma)=V(\Gamma')$, is handled analogously. 
See Figure \ref{fig3} for a visual represenatation of this argument.

So, the only remaining challenge to the bi-determinism of $\Gamma'$ comes from the edges labelled by letters $a\in A$. Bearing in mind how these edges arise in $\Gamma'$, we have 
six subcases.
\begin{itemize}
\item[(1)] \emph{There are two $a$-edges $U\to V_1$ and $U\to V_2$ in $\Gamma'$, and they arise from the existence of vertices $U_1,U_2,V_1',V_2'$ and edges $U_1\to U$, $V_1'\to V_1$ 
labelled by $p_i$, $U_1\to V_1'$ labelled by $a^{(i)}$, $U_2\to U$, $V_2'\to V_2$ labelled by $p_j$, $U_2\to V_2'$ labelled by $a^{(j)}$.} 
Clearly, if $i=j$ then it quickly follows from the bi-determinism of $\Gamma$
that $U_1=U_2$, and then $V_1'=V_2'$ and so $V_1=V_2$; thus we may assume $i\neq j$. Now Lemma \ref{lem-det}, applied to the pair of walks $U_1\xrightarrow{p_i}U$ and $U_2\xrightarrow{p_j}U$, guarantees the existence of a vertex $W$ such that there are walks 
from $W$ to $U_1,U_2$ labelled by $q_iw^{(i)}$ and $q_jw^{(j)}$, respectively, for some $w\in A^*$. But then there are walks in $\Gamma$ from $W$ to $V_1,V_2$ labelled by 
$q_iw^{(i)}a^{(i)}p_i$ and $q_jw^{(j)}a^{(j)}p_j$, respectively. However, since the latter two words represent the same element of $Q$, we conclude that $V_1=V_2$.
\item[(2)] \emph{There are two $a$-edges $U_1\to V$ and $U_2\to V$ in $\Gamma'$, and they arise from the existence of vertices $U_1',U_2',V_1,V_2$ and edges $U_1'\to U_1$, $V_1\to V$ 
labelled by $p_i$, $U_1'\to V_1$ labelled by $a^{(i)}$, $U_2'\to U_2$, $V_2\to V$ labelled by $p_j$, $U_2'\to V_2$ labelled by $a^{(j)}$.} 
Again, the case $i=j$ is clear, so let $i\neq j$. Lemma \ref{lem-det}, applied to the pair of walks 
\[U_1'\xrightarrow{a^{(i)}}V_1\xrightarrow{p_i}V \text{\ \ and\ \ } 
U_2'\xrightarrow{a^{(j)}}V_2\xrightarrow{p_j}V,\]
implies that there is a vertex $W$ together with $\Gamma$-walks towards $U_1'$ and $U_2'$ labelled by $q_iw^{(i)}$ and $q_jw^{(j)}$, respectively. 
However, then there are walks from $W$ to $U_1,U_2$ labelled by $q_iw^{(i)}p_i$ and $q_jw^{(j)}p_j$, respectively. Since in $Q$ these two words are equivalent, we conclude $U_1=U_2$. 
\item[(3)] \emph{$a=b\in B$ and there exist $b$-edges $U\to V_1$ and $U\to V_2$ in $\Gamma'$ because there are vertices $U_1,U_2,V_1',V_2'$, $z$-edges $U_1\to U$, $V_1'\to V_1$, $U_2\to U$, $V_2'\to V_2$ and $b^{(z)}$-edges $U_1\to V_1'$ and $U_2\to V_2'$.} 
First of all, this immediately implies $U_1=U_2$ (because of the already proved bi-determinism of $\Gamma'$ with respect to $z$), 
which in turn forces $V_1'=V_2'$ (because of the bi-determinism of $\Gamma$ with respect to the letter $b^{(z)}$) and, finally, $V_1=V_2$ (again because of the bi-determinism of $\Gamma'$ with respect to $z$).
\item[(4)] \emph{$a=b\in B$ and there exist $b$-edges $U_1\to V$ and $U_2\to V$ in $\Gamma'$ because there are vertices $U_1',U_2',V_1,V_2$, $z$-edges $U_1'\to U_1$, $V_1\to V$, 
$U_2'\to U_2$, $V_2\to V$ and $b^{(z)}$-edges $U_1'\to V_1$ and $U_2'\to V_2$.} 
This is analogous to the previous subcase.
\item[(5)] \emph{$a=b\in B$ and there exist $b$-edges $U\to V_1$ and $U\to V_2$ in $\Gamma'$ because there are vertices $U_1,U_2,V_1',V_2'$, $z$-edges $U_1\to U$, $V_1'\to V_1$, 
$p_i$-edges $U_2\to U$, $V_2'\to V_2$, a $b^{(z)}$-edge $U_1\to V_1'$ and a $b^{(i)}$-edge $U_2\to V_2'$.} 
Let $U_1'$ be the endpoint of the edge in $\Gamma$ labelled by $q_i$ going out of $U_1$. Since $q_ip_i=q_0p_0$ holds in $Q$ and because of the edge $U\xrightarrow{z}U_1$ and the already proved determinism with respect to $z$, the endpoint of the (unique) $p_i$-edge originating from $U_1'$ is $U$. However, because of the already proven bi-determinism of $\Gamma'$ with respect to the letter $p_i$, we conclude that $U_1'=U_2$.
Therefore, the walk $U_1\to U_2\to V_2'$ in $\Gamma$ is labelled by $q_ib^{(i)}$. 
On the other hand, let $V_1''$ be the endpoint of the unique edge in $\Gamma$ labelled by $q_i$ going out of $V_1'$. By a completely analogous argument to the one just presented, the endpoint of the (unique) $p_i$-edge going out of $V_1''$ is $V_1$. Also, the walk $U_1\to V_1'\to V_1''$ in $\Gamma$ is labelled by $b^{(z)}q_i$. Since $q_ib_i=b^{(z)}q_i$ holds in $Q$, we must have $V_1''=V_2'$. But then $V_2'$ has edges labelled by $p_i$ going out to $V_2$ and $V_1$, which, by the bi-determinism of $\Gamma'$ with respect to the letter $p_i$, implies that $V_1=V_2$.
\item[(6)] \emph{$a=b\in B$ and there exist $b$-edges $U_1\to V$ and $U_2\to V$ in $\Gamma'$ because there are vertices $U_1',U_2',V_1,V_2$, $z$-edges $U_1'\to U_1$, $V_1\to V_1$, 
$p_i$-edges $U_2'\to U_2$, $V_2\to V$, a $b^{(z)}$-edge $U_1'\to V_1$ and a $b^{(i)}$-edge $U_2'\to V_2$.} 
This is analogous to the previous case.
\end{itemize}

It is at this point that we can conclude that the homomorphism $\psi:Q\to\mathrm{RU}(M_{S,T})$ is injective and therefore an isomorphism, implying (as explained in Subsection \ref{idea}) the first part of Theorem \ref{main-GK}.

\subsection{Embedding $T$ into the right units}

\begin{lem}\label{lem-T}
The elements $b^{(z)}\psi=zbz^{-1}$, $b\in B$, generate a submonoid of $\mathrm{RU}(M_{S,T})$ isomorphic to $T$.
\end{lem}

\begin{proof}
Note that for any $u,v\in B^*$ we have $zuz^{-1}zvz^{-1} = zuvz^{-1}$ holding in $M_{S,T}$. This implies that $w\mapsto zwz^{-1}$ is a monoid homomorphism of $B^*$ into $\mathrm{RU}(M_{S,T})$. Now the statement we need to prove is as follows:
\begin{equation}\label{t-equiv}
\text{we have }zuz^{-1}=zvz^{-1}\text{ in }M_{S,T}\text{ if and only if }u=v\text{ in }T
\end{equation}
for all $u,v\in B^*$.

($\Ra$) Starting from the root vertex of $\Omega$, the edge labelled by $z$ takes us into a $z$-zone, 
which is a copy of the Cayley graph $\Gamma_S$ of $S$. More precisely, we are now in the root vertex $U_0$ of this copy of $\Gamma_S$. By reading words $u$ and $v$ respectively, 
we entirely remain within $\Gamma_T$, the Cayley graph of $T$ with respect to $B$, which appears here as a subgraph of $\Gamma_S$. If this results in two different vertices 
$U_1\neq U_2$ of $\Gamma_T$, then, by traversing the incoming edges into $U_1$ and $U_2$ labelled by $z$ backwards, we would arrive at two different vertices of $\Omega$, say $V_1$ and $V_2$, respectively, by reading $zuz^{-1}$ and $zvz^{-1}$ from its root vertex. However, since by Proposition \ref{pro:mst-om} there is a labelled graph morphism from $S\Gamma(1)$ into $\Omega$ that takes 
the root vertex of the former to the root vertex of the latter, this would imply that $zuz^{-1}\neq zvz^{-1}$ (as both are right units), contradicting our assumption. Thus we must 
have $U_1=U_2$. But this means that within $\Gamma_T$, starting from its root vertex, the endpoints of the walks labelled by the words $u$ and $v$ are the same; that is, $u=v$ holds in $T$.

($\La$) Assume that $u,v\in B^*$ are such that $u=v$ holds in $T$. Then, bearing in mind \cite[Corollary 3.2]{Gr-Inv}, $zu=zuz^{-1}\cdot z$ and $zv=zvz^{-1}\cdot z$ are right units of $M_{S,T}$. Applying \cite[Corollary 3.2]{Gr-Inv} again, as well as Lemma \ref{lem:calc}(i), we have that
$$
zu = zp_0^{-1} \cdot p_0up_0^{-1}\cdot p_0 = q_0\psi\cdot u^{(0)}\psi\cdot p_0\psi = 
(q_0u^{(0)}p_0)\psi
$$
holds in $\mathrm{RU}(M_{S,T})$. Analogously, $zv=(q_0v^{(0)}p_0)\psi$. However, by the given condition, $q_0u^{(0)}=q_0v^{(0)}$ is a defining relation in the presentation of $Q$. 
Therefore, $zu=zv$ holds in $M_{S,T}$ and so $zuz^{-1}=zvz^{-1}$.
\end{proof}

This lemma completes the proof of Theorem \ref{main-GK}.

\subsection{The proof of Theorem \ref{thm:dense}}

\begin{lem}\label{lem-U}
An element is a unit of $M_{S,T}$ if and only if it is represented by a word $zuz^{-1}$ where $u\in B^*$ represents a unit of $T$.
\end{lem}

\begin{proof}
($\La$) Assume that the word $u\in B^*$ represents a unit of $T$. Hence, there is a word $v\in B^*$ such that $uv=vu=1$ holds in $T$. But then, by \cite[Corollary 3.2]{Gr-Inv} and the converse implication in \eqref{t-equiv},
$$
(zuz^{-1})(zvz^{-1}) = zuvz^{-1}=zz^{-1}=1
$$
holds in $M_{S,T}$, as well as $(zvz^{-1})(zuz^{-1})=1$, showing that $zuz^{-1}$ represents a unit in $M_{S,T}$. 

($\Ra$)
By \cite[Theorem 1.3]{GR} and \cite[Proposition 4.2]{IMM}, for every special inverse monoid (and so, in particular, for $M_{S,T}$) there exists a factorisation of its relator words such that the factors represent units in the inverse monoid in question and generate its entire group of units.
So, to achieve our aim, it suffices to show that if a relator word can be written as $v_1wv_2$ where $v_1,w,v_2$ represent units of $M_{S,T}$ then $w=zuz^{-1}$ holds in $M_{S,T}$ for some word $u\in B^*$ representing a unit of $T$.

To analyse the possibilities for such factorisations, just as in \cite[Theorem 4.1]{GK}, we consider the special inverse monoid $N$ generated by $p_0,\dots,p_k,z$ subject to the relations
$$
p_ip_i^{-1}=1\ (0\leq i\leq k), \quad zz^{-1}=1, \quad z \left(\prod_{i=0}^k p_i^{-1}p_i\right) z^{-1} = 1.
$$
It was shown in the proof of \cite[Theorem 4.1]{GK} that the group of units of $N$ is trivial and that none of its generators represents a unit of $N$. 
In addition, it is immediate to see (by identifying $d$ and all generators from $A$ with $1$) that there is a natural homomorphism from $M_{S,T}$ onto $N$. Since units of $M_{S,T}$ must clearly map to units of $N$, it follows that all units of $M_{S,T}$ map to $1$ under such a homomorphism.

Hence, it suffices to show that for any factorisation $v_1wv_2$ of a relator word of $M_{S,T}$ (with $v_1,v_2$ possibly empty) such that $v_1,w,v_2$ all map to the identity element of $N$, the element represented by $w$ is either not a unit of $M_{S,T}$ or $w=zuz^{-1}$ holds for some $u\in B^*$ representing a unit in $T$. 
A direct inspection of the defining relations of $M_{S,T}$ show that the only such factorisations are $(p_iap_i^{-1})(p_ia^{-1}p_i^{-1})=1$ and $(zbz^{-1})(zb^{-1}z^{-1})=1$, stemming from the first and penultimate type of relations of $M_{S,T}$, respectively.

We claim that $p_iap_i^{-1}$ cannot represent a unit of $M_{S,T}$.
Note that in the graph $\Omega$ the word $ap_i^{-1}$ cannot be read into root vertex of $\Omega$ since the root vertex of any $p_i$-zone has no incoming edges labelled by any letter from $A$.
Since $\Omega$ is a morphic image of $S\Gamma(1)$, the same is true in the latter graph. 
Therefore, the element represented by $ap_i^{-1}$ cannot be left invertible, thus preventing $p_iap_i^{-1}$ from being a unit of $M_{S,T}$.

On the other hand, assume that $zbz^{-1}$ represents a unit of $M_{S,T}$ for some $b\in B$.
Then, in particular, $zbz^{-1}$ represents a left invertible element, implying that there is a walk in $S\Gamma(1)$ labelled by this word terminating at its root. 
However, because of the morphism $S\Gamma(1)\to\Omega$ taking the root of the former graph to the root of the latter (whose existence is guaranteed by Proposition \ref{pro:mst-om}), there is a walk $U_3\xrightarrow{z}U_2\xrightarrow{b}U_1\xleftarrow{z}U_0$ in $\Omega$, where $U_0$ is the root of $\Omega$. 
Since $U_2$ has an incoming edge labelled by $z$, it belongs to the $z$-zone whose root vertex is $U_1$; in fact, it belongs to the copy of the Cayley graph $\Gamma_T$ of $T$ within that $z$-zone. 
Therefore, there is a walk $U_1\leadsto U_2$ within this copy of $\Gamma_T$ labelled by a word $x\in B^*$, say. 
This walk, taken together with the edge $U_2\xrightarrow{b}U_1$, forms a closed walk in $\Gamma_T$ starting at $U_1$ and labelled by $xb$. Thus, $xb=1$ holds in $T$.
By the converse implication of \eqref{t-equiv}, we have that
$$
(zxz^{-1})(zbz^{-1}) = zxbz^{-1} = zz^{-1} = 1
$$
holds in $M_{S,T}$. Upon multiplying the previous equation by $zb^{-1}z^{-1}$ from the right, we obtain that $zb^{-1}z^{-1}=zxz^{-1}$ holds in $M_{S,T}$.
So, we deduce
$$
z\cdot 1\cdot z^{-1} = 1 = (zbz^{-1})(zb^{-1}z^{-1}) = (zbz^{-1})(zxz^{-1}) = zbxz^{-1},
$$
which, by the direct implication of \eqref{t-equiv}, means that $bx=1$ holds in $T$.
Hence, $b$ must be an invertible element of $T$.

Since $zb^{-1}z^{-1}$ is a unit of $M_{S,T}$ if and only if $zbz^{-1}$ represents a unit, the assumption that $zb^{-1}z^{-1}$ is invertible in $M_{S,T}$ also implies that $b$ is invertible in $T$. 
Therefore, there is a word $b_1\dots b_r\in B^*$ representing the inverse of $b$ in $T$, i.e.\ such that $b\cdot b_1\dots b_r = b_1\dots b_r \cdot b = 1$ holds in $T$. 
By \cite[Proposition 4.2(i)]{DG23}, the complement of the group of units in a right cancellative monoid is always an ideal, so all $b_1,\dots,b_r$ must represent units of $T$.
We have already proved that then all $zb_1z^{-1},\dots,zb_rz^{-1}$ must represent units of $M_{S,T}$.
The converse implication of \eqref{t-equiv} and \cite[Corollary 3.2]{Gr-Inv} yield 
$$
(zbz^{-1})(zb_1\dots b_rz^{-1}) = (zb_1\dots b_rz^{-1})(zbz^{-1}) = zz^{-1} = 1.
$$
From this it quickly follows that $zb^{-1}z^{-1} = zb_1\dots b_rz^{-1} = (zb_1z^{-1})\dots (zb_rz^{-1})$. 
 
We have already seen that this implies all $b_s$ to be invertible in $T$. Thus,
$$
\{zbz^{-1}:\ b\text{ is invertible in }T\}
$$
is a monoid generating set of the group of units of $M_{S,T}$. 
From this, the direct implication of the lemma follows straightforwardly.
\end{proof}

\begin{proof}[Proof of Theorem \ref{thm:dense}]
Let $S$ be a finitely RC-presented monoid and let $T$ be a finitely generated submonoid of $S$. Upon choosing a presentation for $S$ which satisfies the assumptions made in Definition \ref{def:mST}, let $M=M_{S,T}$. 
If $N$ is the submonoid of $M$ generated by the elements represented by $zbz^{-1}$, $b\in B$, then by Lemma \ref{lem-T} we have $T\cong N$. 
By combining Lemmas \ref{lem-T} and \ref{lem-U} we obtain that $N$ contains the group of units of $M$. The theorem is now proved.
\end{proof}


\section{The question of finite RC-presentability}\label{sec:fp}

\subsection{(Non)finite presentability of monoids from the class $\mathcal{RU}$}

Upon looking at Theorem \ref{main-GK} and the RC-presentation of the monoid of right units of $M_{S,T}$, the following general question springs to mind naturally.

\begin{prb}
For what pairs $(S,T)$ consisting of a finitely RC-presented right cancellative monoid $S$ and its finitely generated submonoid $T$ is $\mathrm{RU}(M_{S,T})$ not finitely
RC-presented?
\end{prb}

Given the form of this RC-presentation, it is reasonable to conjecture that its failure to be equivalent to any finite RC-presentation should be quite a frequent occasion.
In this subsection we exhibit one such example which, in addition, has a trivial group of units.

\begin{thm}
    There is a finitely presented special inverse monoid whose submonoid of right units is not finitely RC-presented.  
\end{thm}

\begin{proof}
Set $S=\{a\}^*$ and $B=\es$ so that $T$ is the trivial submonoid of $S$. We want to have $k=1$, and for this it suffices to define the free monoid $S$ by a trivial
relation, say $a=a$. Then the generators of $M_{S,T}$ are $a,p_0,p_1,z,d$ (although one can easily see that $d$ will play no significant role here,
as all $d$-edges in $S\Gamma(1)$ will be loops, and most of the vertices will have such loops -- in fact, all of them except the root and the roots of $p_i$-zones).
Then, by our Theorem \ref{main-GK}, 
$$
\mathrm{RU}(M_{S,T}) = \MonRC\pre{a_0,a_1,p_0,p_1,q_0,q_1}{q_0a_0^mp_0=q_1a_1^mp_1\; (m\geq 0)},
$$
where $a_0$ stands for $p_0ap_0^{-1}$, $a_1$ stands for $p_1ap_1^{-1}$, and $q_0,q_1$ are $zp_0^{-1}$ and $zp_1^{-1}$, respectively. (From this, it is fairly 
straightforward to deduce that the group of units of $M_{S,T}$ is trivial, either from \cite[Corollary 2.15]{ruvskuc1999presentations} or from our Lemma \ref{lem-U}.) 
We claim that the above RC-presentation is not equivalent to any finite RC-presentation. 

To see this, note that the rewriting rules stemming from the relations $q_0a_0^mp_0=q_1a_1^mp_1$ are necessarily non-overlapping. Starting from any word $w$ over
the alphabet $\{a_0,a_1,p_0,p_1,q_0,q_1\}$ all we can do with these rules at our disposal is to take a subword based on a word of the form $qa^mp$ (in the sense of the
index-forgetting map $\xi$ from the proof of Lemma \ref{lem-det}) and change the index, from 0 to 1 or vice versa, of all the letters within that subword. (In fact,
this completely describes the word problem for the considered presentation.) This reasoning shows that the monoid given by the same presentation as above, namely
$$
\Mon\pre{a_0,a_1,p_0,p_1,q_0,q_1}{q_0a_0^mp_0=q_1a_1^mp_1\; (m\geq 0)}
$$
is already (right) cancellative, so in this case $\mathrm{RU}(M_{S,T})$ is given by the same presentation as an ordinary monoid.

Now consider the following sequence of finitely presented monoids:
$$
M_r = \Mon\pre{a_0,a_1,p_0,p_1,q_0,q_1}{q_0a_0^mp_0=q_1a_1^mp_1\ (0\leq m\leq r)}
$$
for $r\geq 0$. In a similar fashion as above we show that each of these monoids is (right) cancellative, so the same generators and relations serve
as a finite RC-presentation for the $M_r$. It follows from \cite{CRR} (i.e.\ the definition of a right cancellative chain) that should $\mathrm{RU}(M_{S,T})$ be 
finitely RC-presented, we would then have that it is isomorphic to $M_r$ for some $r$. However---while referring to plain monoid presentations---it is quite easy 
(based on above remarks) to show that 
$$
q_0a_0^{r+1}p_0 \neq q_1a_1^{r+1}p_1
$$
in $M_r$, as we can also describe the word problem for $M_r$: $u=v$ if and only if $u\xi=v\xi$ and $u$ and $v$ only differ by the indices of some blocks
based on $qa^kp$ with $k\leq r$. This is a contradiction.
\end{proof}

\subsection{The monoid of right units in the Gray-Ru\v skuc construction}

Let us fix an alphabet $A=\{a_1,\dots,a_n\}$ and a letter $t\not\in \ol{A}$. Let $Q=\{r_i:\ i\in I\}$ be a (not necessarily finite) set of words from 
$\ol{A}^*$, and let $K_Q=\Gp\pre{A}{r_i=1\; (i\in I)}$ and $G_Q=K_Q\ast FG(t)$. 
Throughout, we assume that $I$ is not empty.

Further, let $W=\{w_j:\ 1\leq j\leq k\}$ be a finite subset of $\ol{A}^*$, and denote by $T_W$ the submonoid of the group $K_Q$ generated by elements represented by
words from $W$. For a list of words $u_1,\dots,u_m\in\ol{A}^*$ define
$$
e(u_1,\dots,u_m) = u_1u_1^{-1}\cdots u_mu_m^{-1}.
$$
Define the following special inverse monoid:
\begin{equation}\label{mqw}
M_{Q,W} = \Inv\pre{A,t}{fr_1=1,\; r_i=1\; (i\in I\setminus\{1\})},
\end{equation}
where $1\in I$ is a distinguished index, and
$$
f=e(a_1,\dots,a_n,tw_1t^{-1},\dots,tw_kt^{-1},a_1^{-1},\dots,a_n^{-1}).
$$
It was remarked in \cite{GR} that this is always an $E$-unitary inverse monoid, with greatest group image $G_Q$, and that it can be also equivalently given by the following 
defining relations:
\begin{align*}
& r_i=1 & (i\in I), \\
& a_pa_p^{-1}=a_p^{-1}a_p=1 & (1\leq p\leq n), \\
& tw_jt^{-1}tw_j^{-1}t^{-1} =1 & (1\leq j\leq k). 
\end{align*}

With appropriate choices of parameters, this construction was utilised in \cite{GR} to show the following results:
\begin{itemize}
\item[(a)] There exists a one-relator special inverse monoid whose group of units is not a one-relator group (in a sharp contrast to the plain monoid case).
\item[(b)] There exists a one-relator special inverse monoid whose group of units is finitely presented but the monoid of its right units is not (as a plain monoid). 
\item[(c)] There exists a finitely presented special inverse monoid whose group of units is not finitely presented.
\end{itemize}

Here is the main result of this subsection.

\begin{thm}\label{main-GR}
Let $M_{Q,W}$ be the special inverse monoid defined in \eqref{mqw}. Let $B=\{b_1,\dots,b_k\}$ be an alphabet disjoint from $\ol{A}\cup\{t\}$, and let 
\begin{align*}
R_{Q,W} = \MonRC\langle \ol{A},B,t\,|\, & aa^{-1}=a^{-1}a=1\; (a\in A), \\
& r_i=1\; (i\in I),\\
&tw_j=b_jt\; (1\leq j\leq k)\rangle .  
\end{align*}
Then the monoid of right units of $M_{Q,W}$ is isomorphic to $R_{Q,W}$.
\end{thm}

\begin{rmk}
Note that the monoids given by \emph{ordinary} monoid presentations of this type are called \emph{Otto-Pride extensions} \cite{GS} (after the paper \cite{PO}), where
the input is a monoid morphism. When this morphism is injective, as is here the case with the inclusion map $T_W\hookrightarrow  K_Q$, such extension is called \emph{HNN-like}.
So, the previous theorem shows that the monoid of right units of $M_{Q,W}$ is isomorphic to $R_{Q,W}$, the greatest right cancellative image of the HNN-like Otto-Pride extension arising from the group $K_Q$ and the submonoid $T_W$.
\end{rmk}

First we are going to prove that the map $w_j\mapsto b_j$ ($1\leq j\leq k$) extends to a homomorphism $T_W\to R_{Q,W}$ that is actually an embedding, so that
the submonoid of $R_{Q,W}$ generated by $B$ forms an isomorphic copy of $T_W$.

\begin{pro}
$\gen{B}_{R_{Q,W}} \cong \gen{W}_{K_Q} = T_W$.
\end{pro}

\begin{proof}
First of all, note that if 
$$w_{j_1}\cdots w_{j_m} = w_{l_1}\cdots w_{l_r}$$
holds in $K_Q$ for some $w_{j_1},\dots,w_{j_m},w_{l_1},\cdots,w_{l_r}\in W$ then in $R_{Q,W}$ we have
$$
b_{j_1}\cdots b_{j_m} t = tw_{j_1}\cdots w_{j_m} = tw_{l_1}\cdots w_{l_r} = b_{l_1}\cdots b_{l_r} t,
$$ 
which, by the right cancellative property of $R_{Q,W}$, implies $b_{j_1}\cdots b_{j_m} = b_{l_1}\cdots b_{l_r}$. 
Hence, it follows that the map $w_j\mapsto b_j$ ($1\leq j\leq k$) naturally extends to a well-defined homomorphism $T_W\to R_{Q,W}$ given by 
$$w_{j_1}\cdots w_{j_m} \mapsto b_{j_1}\cdots b_{j_m}.$$

It remains to prove that the considered homomorphism is injective, i.e.\ that $b_{j_1}\cdots b_{j_m} = b_{l_1}\cdots b_{l_r}$ implies $w_{j_1}\cdots w_{j_m} = w_{l_1}\cdots w_{l_r}$ holds in $T_W$.
So, assume that $b_{j_1}\cdots b_{j_m} = b_{l_1}\cdots b_{l_r}$ holds in $R_{Q,W}$.
Let $G_{Q,W}$ be the greatest group image of $R_{Q,W}$, and let $\eta:R_{Q,W}\to G_{Q,W}$ be the corresponding canonical homomorphism. Then $G_{Q,W}$ is defined, as a group, by the same presentation as $R_{Q,W}$ (and, in particular, by the same set of generators).
Upon applying $\eta$, we have that $b_{j_1}\cdots b_{j_m} = b_{l_1}\cdots b_{l_r}$ holds in $G_{Q,W}$.
However, the defining relations of $G_{Q,W}$ imply that for all $1\leq j\leq k$ we have
\begin{equation}\label{bj}
b_j = tw_jt^{-1}.
\end{equation}
Therefore, in $G_{Q,W}$ we have
$$
(tw_{j_1}t^{-1})\cdots (tw_{j_m}t^{-1}) = (tw_{l_1}t^{-1})\cdots (tw_{l_r}t^{-1}),
$$
which immediately implies that $w_{j_1}\cdots w_{j_m} = w_{l_1}\cdots w_{l_r}$ holds in $G_{Q,W}$ .

On the other hand, notice that, in the group presentation for $G_{Q,W}$, the relations \eqref{bj} make the generators $b_j$ ($1\leq j\leq k$) redundant in the sense that the generators $b_j$ can be eliminated via Tietze transformations from the presentation along with these relations (in fact, the group-theoretically equivalent relations $tw_j=b_jt$). 
Hence, the group $G_{Q,W}$ is presented by $\Gp\pre{\ol{A},t}{r_i=1\; (i\in I)}$ and so $G_{Q,W}=K_Q\ast FG(t)=G_Q$. 
By the Normal Form Theorem for free products of groups \cite[Theorem IV.1.2]{LSch}, we arrive at the the required conclusion that $w_{j_1}\cdots w_{j_m} = w_{l_1}\cdots w_{l_r}$ holds in $K_Q$, and thus in $T_W$.
\end{proof}

\begin{proof}[Proof of Theorem \ref{main-GR}]
For a right cancellative monoid $P$, the \emph{inverse hull} $IH(P)$ is the inverse submonoid of the \emph{symmetric inverse monoid} $\mathcal{I}_P$ of all partial injective 
maps on $P$ generated by the right translations $\rho_q: x\mapsto xq$ ($q,x\in P$). Because $P$ is assumed to be right cancellative, any $\rho_q$ is an injection from $P$ onto 
its principal left ideal $Pq$. 
It is folklore in semigroup theory (see e.g.\ \cite[Theorem 1.22]{CP}) that the monoid of right units of $IH(P)$ is isomorphic to $P$ and that it consists of all translations $\rho_r$ such that $r$ is a right unit of $Q$; 
hence, every right cancellative monoid arises as the monoid of right units of some inverse monoid. 

For brevity, write $M=M_{Q,W}$ and $R=R_{Q,W}$. Our aim is to show that there is an isomorphism from the right units if $M$ onto the right units of $IH(R)$, where, by the remarks from the preivous paragraph, the latter form a monoid isomorphic to $R$.
The inverse hull $IH(R)$ is generated (as an inverse monoid) by $\rho_a$, $\rho_b$ ($a\in \ol{A}$, $b\in B$) and $\rho_t$, corresponding to the generators of $R$. All of these are right units in $IH(R)$, and all $\rho_a$, $a\in \ol{A}$, are invertible, that is, permutations on $R$, with $\rho_{a^{-1}}$ being the inverse map of $\rho_a$ for all $a\in A$. 
Clearly, for any word $w$ over $\ol{A}\cup B\cup\{t\}$, if $w$ represents the element $s\in R$ then in $IH(R)$ the word $w$ represents $\rho_s$. Therefore, for any $i\in I$ we have 
$\rho_{r_i} = \rho_1 = \id_R$, and in this sense $IH(R)$ satisfies the relations $r_i=1$ (under the previously described correspondence between the generators of $IH(R)$ and $R$). 
Furthermore, for any $w_j\in W$ we must have $\rho_t\rho_{w_j} = \rho_{b_j}\rho_t$.
Because $IH(R)$ is a submonoid of the symmetric inverse monoid $\mathcal{I}_R$, the latter equality implies
$$
\rho_{b_j} = \rho_{b_j}\rho_t\rho_t^{-1} = \rho_t\rho_{w_j}\rho_{t}^{-1}.
$$
This relation shows that the word $tw_jt^{-1}$ represents a right unit of $IH(R)$, so that $IH(R)$ satisfies all the relations $tw_jt^{-1}tw_j^{-1}t^{-1} =1$.
We have just proved that all defining relations of $M$ hold in $IH(R)$ under the map $a\mapsto \rho_a$ ($a\in \ol{A})$, $t\mapsto \rho_t$. 
Hence, this map extends to a homomorphism $\mu: M \to IH(R)$. 
Furthermore, $\mu\restriction_{\mathrm{RU}(M)}$ maps the right units of $M$ onto the right units of $IH(R)$, because the previous displayed equation shows that all generators $\rho_b$, $b\in B$, of $IH(R)$ are redundant.
In the remainder of the proof we argue that the map $\mu\restriction_{\mathrm{RU}(M)}$ must be injective.

To achieve this, it suffices to show that there exists a monoid homomorphism $\nu: R\to \mathrm{RU}(M)$ such that $\nu\mu = \nu\left(\mu\restriction_{\mathrm{RU}(M)}\right)$ is an
isomorphism between $R$ and its copy in $IH(R)$. 
(Since the domain of $\mu\restriction_{\mathrm{RU}(M)}$ coincides with the codomain of $\nu$, the eventual conclusion that $\nu\mu$ is an isomorphism and thus a bijection would imply that both $\nu$ and $\mu\restriction_{\mathrm{RU}(M)}$ must be injective.)
Note that $\ol{A}\cup \{t\}\cup \{tw_jt^{-1}:\ 1\leq j\leq k\}$ represents a generating set for $\mathrm{RU}(M)$ as any other prefix of a relator of $M$ can be expressed, by using \cite[Corollary 3.2]{Gr-Inv}, as a product of these words. 
Consider the identity map on $\ol{A}\cup \{t\}$ along with $b_j\mapsto tw_jt^{-1}$, $1\leq j\leq k$. This map, which we denote by $\nu_0$, assigns to each generator of $R$ an element of the generating set of $\mathrm{RU}(M)$. 
As the relations $r_i=1$ involve only the letters from $\ol{A}$ representing invertible elements in $M$, all of these relations are satisfied in $\mathrm{RU}(M)$. 
Furthermore, for all $1\leq j\leq k$ we have
$$
(b_j\nu_0)(t\nu_0) = tw_jt^{-1}t = tw_j = (t\nu_0)(w_j\nu_0),
$$
where $w_j\nu_0$ is a short-hand for applying $\nu_0$ to each letter in $w_j$ individually. 
The second equation here follows because in $M$ the word $tw_jt^{-1}t$ represents a right unit of $M$, since it is a prefix of a relator word, so \cite[Corollary 3.2]{Gr-Inv} applies once again. 
Thus we conclude that all defining relations of $R$ hold in $\mathrm{RU}(M)$ as well, in the sense of the correspondence $\nu_0$ between the generating sets of these monoids; hence, by Lemma \ref{lem:rc}, $\nu_0$ extends to a homomorphism $\nu: R\to \mathrm{RU}(M)$.
Note that $\nu\mu$ restricted to the generators of $R$ yields precisely the standard correspondence between the generators of $R$ and those of the right units of $IH(R)$.
Therefore, $\nu\mu$ is an isomorphism, as required.
\end{proof}

\begin{cor}
If $K_Q$ is a finitely presented group (i.e.\ when $|Q|$ is finite), then $\mathrm{RU}(M_{Q,W})$ is finitely RC-presented.
\end{cor}

\begin{cor}
There is a finitely presented $E$-unitary special inverse monoid $M$ with the following properties:
\begin{itemize}
\item[(a)] The submonoid $\mathrm{RU}(M)$ of right units of $M$ is finitely RC-presented;
\item[(b)] The group of units $U(M)$ of $M$ is not finitely presented, and, consequently, $\mathrm{RU}(M)$ is not finitely presented as a monoid.
\end{itemize}
\end{cor}

\begin{proof}
By \cite[Theorem 6.3(vi)]{GR} if $\mathrm{RU}(M_{Q,W})$ is finitely presented as a monoid then both the group $K_Q$ and the monoid $T_W$ must be finitely presented.
Similarly, it follows from item (v) of the same theorem that the group of units of $M_{Q,W}$ is isomorphic to $U(T_W)*K_Q$, where $U(T_W)$ denotes the 
group of units of $T_W$ (which is a subgroup of $K_Q$). As shown in \cite[Subsection 7.3]{GR}, if we take $W$ such that $W^{-1}=W$ we then have that $T_W$ coincides with $H_W$,
the subgroup of $K_Q$ generated by $W$.
So, we can choose $Q$ and $W$ such that $K_Q$ is a finitely presented group such that its subgroup $H_W$ is finitely generated but not finitely presented.
For such a choice of $Q,W$, we have that $\mathrm{RU}(M_{Q,W})$ is finitely RC-presented. 
On the other hand, $U(\mathrm{RU}(M_{Q,W}))=U(M_{Q,W})=K_Q * H_W$ is not finitely presented, which follows, as in the proof of \cite[Theorem 6.3(vi)]{GR}, from the fact that $H_W$ is a retract of $K_Q * H_W$. Also, as already remarked, the non-finite presentability of $H_W$ implies that $\mathrm{RU}(M_{Q,W})$ is not finitely presented as a monoid.
\end{proof}

\begin{cor}\label{cor:finRC-nonfinU}
There exists a finitely RC-presented monoid whose group of units is not finitely presented.
\end{cor}

Note that by \cite[Proposition 4.2(i)]{DG23}, the complement of the group of units in a right cancellative monoid is always an ideal. Hence, the previous corollary is 
in sharp contrast with \cite[Proposition 3.1]{ruvskuc1999presentations}, a result telling us that if the complement of the group of units $U$ of a finitely presented monoid $M$ 
is an ideal then $U$ must be finitely presented as well.


\small
\begin{ackn}
We thank Nik Ru\v{s}kuc for   helpful conversations and important contributions that he made to the ideas that led to the proof of Theorem~\ref{thm_monhybrid_inv}.
\end{ackn}
\normalsize



\begin{thebibliography}{99}
\frenchspacing

\bibitem{Adj}
S. I. Adjan,
Defining relations and algorithmic problems for groups and semigroups (in Russian),
\emph{Trudy Mat. Inst. Steklov.} \textbf{85} (1966), 123 pp.

\bibitem{AO}
S. I. Adyan, G. U. Oganesyan,
On the word and divisibility problems for semigroups with one relation (in Russian),
\emph{Mat. Zametki} \textbf{41} (1987), 412--421.

\bibitem{Bel}
V. Y. Belyaev, Imbeddability of recursively defined inverse semigroups in finite presented semigroups (in Russian),
\emph{Sibirsk. Mat. Zh.} \textbf{25} (1984), 50--54.

\bibitem{BridsonHaefliger}
M. R. Bridson, A. Haefliger,
\emph{Metric Spaces of Non-Positive Curvature},
Springer, Berlin, Heidelberg, 1999.

\bibitem{Cain}
A. J. Cain,
\emph{Presentations of Subsemigroups of Groups},
Ph.D. thesis, University of St Andrews, 2005.

\bibitem{CRR}
A. J. Cain, E. F. Robertson, N. Ru\v skuc,
Cancellative and Malcev presentations for finite Rees index subsemigroups and extensions, 
\emph{J. Austral. Math. Soc.} \textbf{84} (2008), 39--61.

\bibitem{CRRT}
C. M. Campbell, E. F. Robertson, N. Ru\v skuc, R. M. Thomas,
Reidemeister-Schreier type rewriting for semigroups,
\emph{Semigroup Forum} \textbf{51} (1995), 47--62.

\bibitem{CP}
A. H. Clifford, G. B. Preston,
\emph{The Algebraic Theory of Semigroups, Vol. I \& II},
Mathematical Surveys No. 7, American Mathematical Society, Providence, R.I., 1961 \& 1967.

\bibitem{de2000topics}
P. de la Harpe,
\emph{Topics in Geometric Group Theory},
University of Chicago Press, 2000.

\bibitem{DG21}
I. Dolinka, R. D. Gray, 
New results on the prefix membership problem for one-relator groups, 
\emph{Trans. Amer. Math. Soc.} \textbf{374} (2021), 4309--4358.

\bibitem{DG23}
I. Dolinka, R. D. Gray, 
Prefix monoids of groups and right units of special inverse monoids,
\emph{Forum of Mathematics, Sigma} \textbf{11} (2023), Article e97, 19 pp.

\bibitem{ghysinfinite}
E. Ghys, P. de la Harpe,
Infinite groups as geometric objects,
in: \emph{Ergodic theory, symbolic dynamics and hyperbolic spaces} (eds. T. Bedford, M. Keane, C. Series), 
pp. 299--314, Oxford University Press, 1991.

\bibitem{Gr-Inv}
R. D. Gray,
Undecidability of the word problem for one-relator inverse monoids via right-angled 
Artin subgroups of one-relator groups,
\emph{Invent. Math.} \textbf{219} (2020), 987--1008.

\bibitem{gray2013groups}
R. D. Gray, M. Kambites, 
Groups acting on semimetric spaces and quasi-isometries of monoids,
\emph{Trans. Amer. Math. Soc.} \textbf{365} (2013), 555--578.

\bibitem{GK}
R. D. Gray, M. Kambites, 
Maximal subgroups of finitely presented special inverse monoids,
\emph{J. Eur. Math. Soc.} (to appear), \texttt{arXiv:\ 2212.04204}.

\bibitem{gray2023subgroups}
R. D. Gray, M. Kambites,
Subgroups of {$E$}-unitary and {$R_1$}-injective special inverse monoids,
\emph{Math. Z.} \textbf{311} (2025), Article {\#}23, 32 pp.

\bibitem{GR}
R. D. Gray, N. Ru\v skuc, 
On groups of units of special and one-relator inverse monoids,
\emph{J. Inst. Math. Jussieu} \textbf{23} (2024), 1875--1918.

\bibitem{GS}
R. D. Gray, B. Steinberg, 
Topological finiteness properties of monoids. Part 2: special monoids, one-relator monoids, amalgamated free products, and HNN extensions,
\emph{Documenta Math.} \textbf{29} (2024), 511--560.

\bibitem{Guba}
V. S. Guba,
On a relation between the word problem and the word divisibility problem for semigroups with one defining relation (in Russian),
\emph{Izv. Ross. Akad. Nauk Ser. Mat.} \textbf{61 (6)} (1997), 27--58.

\bibitem{How}
J. M. Howie,
\emph{Fundamentals of Semigroup Theory}, 
London Math. Soc. Monographs Vol. 12, Clarendon Press, Oxford, 1995.

\bibitem{IMM}
S. V. Ivanov, S. W. Margolis, J. C. Meakin,
On one-relator inverse monoids and one-relator groups,
\emph{J. Pure Appl. Algebra} \textbf{159} (2001), 83--111.

\bibitem{JZL1}
A. Jaikin-Zapirain, M. Linton, 
The strong Atiyah and L\"uck approximation conjectures for one-relator groups, 
\emph{Math. Ann.} \textbf{376} (2020), 1741--1793.

\bibitem{JZL2}
A. Jaikin-Zapirain, M. Linton, 
On the coherence of one-relator groups and their group algebras, 
\emph{Ann. Math.} \textbf{201} (2025), 909--959.

\bibitem{Lal}
G. Lallement,
On monoids presented by a single relation,
\emph{J. Algebra} \textbf{32} (1974), 370--388.

\bibitem{Law}
M. V. Lawson,
\emph{Inverse Semigroups: The Theory of Partial Symmetries},
World Scientific, Singapore, 1998.

\bibitem{LW1}
L. Louder, H. Wilton,
Negative immersions for one-relator groups,
\emph{Duke Math. J.} \textbf{171} (2022), 547--594.

\bibitem{LW2}
L. Louder, H. Wilton,
Uniform negative immersions and the coherence of one-relator groups,
\emph{Invent. Math.} \textbf{236} (2024), 673--712.

\bibitem{LSch}
R. C. Lyndon, P. E. Schupp,
\emph{Combinatorial Group Theory},
Springer-Verlag, Berlin, 1977. 

\bibitem{Ma1}
W. Magnus, 
\"Uber diskontinuierliche Gruppen mit einer definierenden Relation. (Der Freiheitssatz),
\emph{J. reine angew. Math.} \textbf{163} (1930), 141--165.

\bibitem{Ma2}
W. Magnus, 
Das Identit\"atsproblem f\"ur Gruppen mit einer definierenden Relation, 
\emph{Math. Ann.} \textbf{106} (1932), 295--307.

\bibitem{Mak}
G. S. Makanin, On the identity problem in finitely defined semigroups (in Russian),
\emph{Dokl. Akad. Nauk SSSR} \textbf{171} (1966), 285--287.

\bibitem{Mea}
J. Meakin, 
Groups and semigroups: connections and contrasts,
in: \emph{Groups St Andrews 2005, Vol. 2}, pp. 357--400,
London Math. Soc. Lecture Note Ser., Vol. 340, Cambridge Univ. Press, Cambridge, 2007.

\bibitem{Munn}
W. D. Munn,
Free inverse semigroups, 
\emph{Proc. London Math. Soc. (3)} \textbf{29} (1974), 385--404.

\bibitem{Mur}
V. L. Murski\u{\i}, Isomorphic embeddability of a semigroup with enumerable set of defining relations 
to a finitely presented semigroup (in Russian), 
\emph{Mat. Zametki} \textbf{1} (1967), 217--224.

\bibitem{CF}
C.-F. Nyberg Brodda,
The word problem for one-relation monoids: a survey,
\emph{Semigroup Forum} \textbf{103} (2021), 297--355.

\bibitem{Pet}
M. Petrich,
\emph{Inverse Semigroups},
Wiley, 1984.

\bibitem{PO}
S. J. Pride, F. Otto, 
For rewriting systems the topological finiteness conditions FDT and FHT are not equivalent,
\emph{J. London Math. Soc. (2)} \textbf{69} (2004), 363--382.

\bibitem{RuskucPhD}
N. Ru{\v{s}}kuc,
\emph{Semigroup Presentations}, 
PhD Thesis, University of St Andrews, 1995.

\bibitem{ruvskuc1999presentations}
N. Ru{\v{s}}kuc,
Presentations for subgroups of monoids,
\emph{J. Algebra} \textbf{220} (1999), 365--380.

\bibitem{Sch}
H. E. Scheiblich,
Free inverse semigroups, 
\emph{Proc. Amer. Math. Soc.} \textbf{38} (1973), 1--7.

\bibitem{Serre}
J.-P. Serre,
\emph{Trees},
Springer-Verlag, Berlin, Heidelberg, 1980.

\bibitem{steinberg2003topological}
B. Steinberg,
A topological approach to inverse and regular semigroups,
\emph{Pacific J. Math.} \textbf{208} (2003), 367--396.

\bibitem{Stephen}
J. B. Stephen,
Presentations of inverse semigroups,
\emph{J. Pure Appl. Algebra} \textbf{63} (1990), 81--112.

\bibitem{Zh1}
L. Zhang, 
A short proof of a theorem of Adjan,
\emph{Proc. Amer. Math. Soc.} \textbf{116} (1992), 1--3.

\bibitem{Zh2}
L. Zhang, 
Applying rewriting methods to special monoids,
\emph{Math. Proc. Cambridge Philos. Soc.} \textbf{112} (1992), 495--505.

\end{thebibliography}
\end{document}